\RequirePackage{fix-cm}
\documentclass[11pt, a4paper, oneside, DIV11, final, pagebackref]{amsart}
\usepackage[notref,notcite]{showkeys}
\usepackage[latin1]{inputenc}
\usepackage[T1]{fontenc}
\usepackage{amssymb}
\usepackage[stretch=10]{microtype}
\usepackage{fixltx2e}
\usepackage{mathtools}
\usepackage[DIV11, headinclude=true]{typearea}
\usepackage{ifthen}
\usepackage{fancyvrb}
\usepackage{tikz}
\usetikzlibrary{intersections, calc, decorations.markings, decorations.pathmorphing}
\usepackage{hyperref}
\usepackage{attachfile}

\def\stopcompilationbeforecode{
  \bibliography{biblio}
  \bibliographystyle{amsalpha}
  \end{document}
}
\providecommand{\href}[2]{#2}

\providecommand*{\backref}{}
\providecommand*{\backrefalt}{}
\renewcommand*{\backref}[1]{}
\renewcommand*{\backrefalt}[4]{%
	\ifcase #1 %
	\or
	  Cited page~#2.
	\else
	  Cited pages~#2.
	\fi
}

\def\sidesnumber{8}
\def\sideslength{1.21}
\newcommand\RegPolygonOriented[7]{
  \RegPolygonVertices{#1}{#2}{#3}
  \RegPolygonNames{#2}{#4}
  \ifthenelse{#7 = \sidesnumber + 1}
    {\draw (#21)
         \foreach \n in {2,...,\sidesnumber}{-- (#2\n)}
         -- cycle;
    }
    {\draw (#21)
         \foreach \n in {2,...,#7}{-- (#2\n)};
    }
  \begin{scope}[decoration={markings, mark=at position 0.5 with {\arrow{>}}}]
  \foreach \n [evaluate=\n as \m using \n+1] in {#5,...,\numexpr #6-1\relax} {
    \draw[decorate] (#2\n) -- (#2\m);
  }
  \foreach \n [evaluate=\n as \m using \n-1] in {#7,...,\numexpr #6+1\relax} {
    \draw[decorate] (#2\n) -- (#2\m);
  }
  \end{scope}
}

\newcommand\RegPolygonVertices[3]{
  \path #1
      \foreach \n in {1,...,\sidesnumber}{%
      coordinate (#2\n) -- ++ ({#3 + (1-\n)*360/\sidesnumber}:\sideslength)};
  \coordinate (#20) at ($(#21)!0.5!(#2\number\numexpr\sidesnumber/2+1\relax)$);
  \coordinate (#2\number\numexpr\sidesnumber+1\relax) at (#21);
}

\newcommand\RegPolygonNames[2]{
  \foreach \text [count=\p] in {#2}
     \draw ($(#10)!0.87!(#1\p)$) node {$\text$};
}

\newcommand\Faisceau[5]{
  \begin{scope}[decoration={markings, mark=at position 0.7 with {\arrow{>}}}]
  \edef\templabels{#5}
  \foreach \n [count=\p] in \templabels {
    \ifthenelse{#2 = 1}
      {\def\myangle{#3}}
      {\def\myangle{#3 - #4/2 + (\p-1)*#4/(#2-1)}}
    \draw[postaction={decorate}]
          #1 -- ++ ({\myangle}:\sideslength);
    \path #1 -- ++ ({\myangle}:{\sideslength+0.25}) node {$\n$};
  }
  \end{scope}
}
\newcommand{\Pbb}{\mathbb{P}}

\newcommand{\Z}{\mathbb{Z}}
\newcommand{\R}{\mathbb{R}}

\newcommand{\Fbb}{\mathbb{F}}
\newcommand{\Sbb}{\mathbb{S}}
\newcommand{\boS}{\mathcal{S}}

\newcommand{\norm}[1]{\left\| #1 \right\|}
\newcommand{\st}{\::\:}
\newcommand{\abs}[1]{\left| #1 \right|}
\newcommand{\lgth}[1]{\left| #1 \right|}
\newcommand{\tsuffix}{t_{\mathrm{suff}}}
\newcommand{\tessential}{t_{\mathrm{ess}}}

\DeclareMathOperator{\Card}{Card}

\renewcommand{\epsilon}{\varepsilon}

\renewcommand{\phi}{\varphi}
\renewcommand{\leq}{\leqslant}
\renewcommand{\geq}{\geqslant}

\newtheorem{thm}{Theorem}[section]
\newtheorem{prop}[thm]{Proposition}
\newtheorem{defn}[thm]{Definition}

\newtheorem*{prop*}{Proposition}

\theoremstyle{definition}

\newtheorem{rmk}[thm]{Remark}

\numberwithin{equation}{section}

\title[Lower bound for the spectral radius on surface groups]
      {A numerical lower bound for the spectral radius of random walks on surface groups}
\author{S\'ebastien Gou\"ezel}

\address{IRMAR, CNRS UMR 6625,
Universit\'e de Rennes 1, 35042 Rennes, France}
\email{sebastien.gouezel@univ-rennes1.fr}
\date{June 6, 2014}

\begin{document}

\begin{abstract}
Estimating numerically the spectral radius of a random walk on a
nonamenable graph is complicated, since the cardinality of balls
grows exponentially fast with the radius. We propose an algorithm to
get a bound from below for this spectral radius in Cayley graphs with
finitely many cone types (including for instance hyperbolic groups).
In the genus $2$ surface group, it improves by an order of magnitude
the previous best bound, due to Bartholdi.
\end{abstract}

\keywords{spectral radius, cogrowth, numerical algorithm, surface
groups}

\maketitle
\section{Main algorithm}

Let $\Gamma$ be a countable group, generated by a finite symmetric
set $S$ of cardinality $\abs{S}$. The simple random walk
$X_0,X_1,\dotsc$ on $\Gamma$ is defined by $X_0=e$ the identity of
$\Gamma$, and $X_{n+1}=X_n s$ with probability $1/\abs{S}$ for any
$s\in S$. A crucial numerical parameter of this random walk is its
\emph{spectral radius} $\rho=\lim \Pbb(X_{2n} = e)^{1/2n}$.
Equivalently, denote by $W_n$ the number of words of length $n$ in
the generators that represent $e$ in $\Gamma$, then $\Pbb(X_n = e) =
W_n/\abs{S}^n$, so that $\rho = \lim W_{2n}^{1/2n} / \abs{S}$. It is
equivalent to study the spectral radius or the \emph{cogrowth} $\lim
W_{2n}^{1/2n}$.

The spectral radius is at most $1$, and $\rho = 1$ if and only if
$\Gamma$ is amenable. In the free group with $d$ generators, the
generating function $\sum W_n z^n$ can be computed explicitly (it is
algebraic), and the exact value of the spectral radius follows: $\rho
= \sqrt{2d-1}/d$. Since words that reduce to the identity in the free
group also reduce to the identity in any group with the same number
of generators, one infers that in any group $\Gamma$, $\rho \geq
2\sqrt{\abs{S}-1}/\abs{S}$. Moreover equality holds if and only if
the Cayley graph of $\Gamma$ is a tree~\cite{kesten}.

In general, there are no explicit formulas for $\rho$, and even
giving precise numerical estimates is a delicate question. In this
short note, we will describe an algorithm giving bounds from below on
$\rho$ in some classes of groups, particularly for the fundamental
group $\Gamma_g$ of a compact surface of genus $g\geq 2$, given by
its usual presentation
  \begin{equation}
  \label{eq:presentation}
  \Gamma_g = \langle a_1,\dotsc, a_g,b_1,\dotsc, b_g \mid [a_1,b_1]\dotsm [a_g, b_g]=e\rangle.
  \end{equation}
Since there are $4$ generators in $\Gamma_2$, the above trivial bound
obtained by comparison to the free group gives $\rho \geq 0.661437$.
Our main estimate is the following result.
\begin{thm}
\label{main_thm}
In the surface group $\Gamma_2$, one has $\rho \geq 0.662772$.
\end{thm}
This improves on the previously best known result, due to
Bartholdi~\cite{bartholdi_rho}, giving $\rho \geq \rho_{\mathrm{Bar}}
= 0.662421$.\footnote{Bartholdi claims that $\rho\geq 0.662418$, but
implementing his algorithm in multiprecision one gets in fact the
better bound $\rho\geq 0.662421$.} Bartholdi's method is to study a
specific class of paths from the identity to itself (called cactus
trees), for which he can compute the generating function. The radius
of convergence of this generating function is a lower bound for
$\rho$.

The best known upper bound for $\rho$ in $\Gamma_2$ is $\rho \leq
\rho_{\mathrm{Nag}} = 0.662816$, due to
Nagnibeda~\cite{nagnibeda_rho}. Non-rigorous numerical
estimates\footnote{
I obtained this estimate as follows: one can count exactly the number
$W_n$ of words of length $n$ representing the identity in the group,
for reasonable $n$, say up to $n=24$ -- for the record,
$W_{24}=4214946994935248$ -- giving the first values of the sequence
$p_n=\Pbb(X_{2n}=e)$. We know rigorously from~\cite{gouezel_lalley}
that $p_n \sim C \rho^{2n}/n^{3/2}$ when $n\to \infty$. Define
$q_n=\log(n^{3/2} p_n)/(2n)$, it follows that $q_n\to \log \rho$. In
the free group, where $p_n$ is known very explicitly, the sequence
$q_n$ has a further expansion in powers of $1/n$. Assuming that the
same holds in the surface group, we get $q_n = \log \rho +
\sum_{k=1}^K a_k/n^k + o(1/n^K)$ where the $a_k$ are unknown. Using
the known value of $q_{24}$, this gives an estimate for $\log \rho$,
with an error of the order of $1/24$, which is very bad. However, it
is possible to accelerate the convergence of sequences having an
asymptotic expansion in powers of $1/n$: there are explicit recipes
(for instance Richardson extrapolation or Wynn's rho algorithm)
taking such a sequence, and giving a new sequence converging to the
same limit, with an expansion in powers of $1/n$, but starting at
$1/n^2$. Iterating this process, one can eliminate the first few
terms, and get a speed of convergence $O(1/n^L)$ for any $L$ (but one
needs to know enough terms of the initial sequence). Applying this
process to our sequence $q_n$, one gets the claimed estimate for
$\rho$. To make this rigorous, one would need to know that an
asymptotic expansion of $q_n$ exists, with explicit bounds on the
$a_k$ and on the $o(1/n^K)$ term. This seems completely out of reach.
} suggest that $\rho = 0.662812\dotsc$, so the upper bound is still
sharper than our lower bound, although our lower bound is an order of
magnitude better than the bound of~\cite{bartholdi_rho}: indeed,
$\rho_{\mathrm{Nag}} - \rho_{\mathrm{Bar}} \sim 4.10^{-4}$ while our
estimate $\rho$ from Theorem~\ref{main_thm} satisfies
$\rho_{\mathrm{Nag}} - \rho \sim 4.10^{-5}$.

Nagnibeda's upper bound does not rely on a counting argument for
closed paths, but on another spectral interpretation of $\rho$.
Indeed, $\rho$ is also the spectral radius of the Markov operator $Q$
on $\ell^2(\Gamma)$ corresponding to the random walk, i.e., the
convolution with the probability measure $\mu$ which is uniformly
distributed on $S$ (see for instance~\cite[Corollary 10.2]{woess}).
It is also the norm of this operator, since it is symmetric.
Nagnibeda gets the above upper bound by using a lemma of Gabber about
norms of convolution operators on graphs and the precise geometry of
$\Gamma_2$.

\medskip

Our approach to get Theorem~\ref{main_thm} is very similar to
Nagnibeda's. To bound from below the norm of the convolution operator
$Q$, it is sufficient to exhibit one function $u$ (which ought to be
close to an hypothetic eigenfunction for the element $\rho$ of the
spectrum of $Q$) for which $\norm{Qu}/\norm{u}$ is large. This is
exactly what we will do.

For any $\alpha<\rho$, the function $u_\alpha=\sum_{n=0}^\infty
\alpha^n Q^n \delta_e$ is in $\ell^2$, and
$\norm{Qu_\alpha}/\norm{u_\alpha}$ converges to $\rho$ when $\alpha$
tends to $\rho$. Unfortunately, $u_\alpha$ is not explicit enough. To
find estimates, one should rather find an ansatz for the function
$u$, depending on finitely many parameters, and then optimize over
these parameters.

A first strategy would be the following: take a very large ball $B_n$
in the Cayley graph, and compute the function $u$ supported in this
ball such that $\norm{Qu}/\norm{u}$ is largest. This gives a lower
bound $\rho_n$ on $\rho$, and $\rho_n$ converges to $\rho$ when $n$
tends to infinity. However, this strategy is computationally not
efficient at all: one would need to take a very large $n$ to obtain
good estimates (since most mass of $u_\alpha$ is supported close to
infinity if $\alpha$ is close to $\rho$), and the cardinality of
$B_n$ grows exponentially with $n$. On the other hand, it can be
implemented in any finitely presented group for which the word
problem is solvable (see for instance~\cite{cogrowth_thompson} for
examples in Baumslag-Solitar and Thompson groups). We will use a more
efficient method, but which requires more assumptions on the group:
it should have finitely many cone types.

\medskip

To illustrate our method of construction of $u$, let us describe it
quickly in the case of the free group $\Fbb_d$ with $d$ generators.
The sphere $\Sbb^n$ of radius $n\geq 1$ has cardinality
$2d(2d-1)^{n-1}$. Fix some $\alpha<1/\sqrt{2d-1}$, and define a
function $u_\alpha$ by $u_\alpha(x) = \alpha^n$ for $x\in \Sbb^n$,
$n\geq 1$. This function belongs to $\ell^2(\Fbb_d)$. We write $x\sim
y$ if $x$ and $y$ are neighbors in the Cayley graph of $\Gamma$, and
$x\to y$ if $x\sim y$ and $d(e,y)=d(e,x)+1$. Then
  \begin{equation*}
  \langle Qu_\alpha, u_\alpha \rangle = \frac{1}{2d} \sum_{x\sim y} u_\alpha(x)u_\alpha(y)
  = \frac{1}{d} \sum_{x\to y} u_\alpha(x) u_\alpha(y)
  = \frac{1}{d} \sum_{n=1}^\infty (2d-1)\alpha^{2n+1} \abs{\Sbb^n},
  \end{equation*}
since a point in $\Sbb^n$ has $2d-1$ successors in $\Sbb^{n+1}$.
Since $\langle u_\alpha, u_\alpha\rangle = \sum_{n=1}^\infty
\alpha^{2n} \abs{\Sbb^n}$, we get
  \begin{equation*}
  \langle Qu_\alpha, u_\alpha\rangle = \alpha \frac{2d-1}{d} \langle u_\alpha, u_\alpha \rangle.
  \end{equation*}
Hence, $\rho = \norm{Q} \geq \alpha(2d-1)/d$. Letting $\alpha$ tend
to $1/\sqrt{2d-1}$, we finally obtain $\rho\geq \sqrt{2d-1}/d$, which
is the true value of the spectral radius.

\medskip

In the free group, it is natural to take a function $u$ that is
constant of the sphere $\Sbb^n$ of radius $n$, since all the points
in such a sphere are equivalent: the automorphisms of the Cayley
graph of $\Fbb_d$ fixing the identity act transitively on $\Sbb^n$.
In more general groups, for instance surface groups, this is not the
case. Intuitively, we would like to take a function that decays
exponentially as above, but with different values on different
equivalence classes under the automorphism group. However, this
automorphism group is finite in the case of surface groups, so
instead of true equivalence classes (which are finite), we will
consider larger classes, of points that ``locally behave in the same
way'', and we will construct functions that are constants on such
classes of points (leaving only finitely many parameters which one
can optimize using a computer).

This intuition is made precise with the notion of \emph{type} of the
elements of the group (as in~\cite{nagnibeda_rho}). Let $\Gamma$ be a
countable group generated by a finite symmetric set $S$. Assume that
there are no cycles of odd length, so that any edge can be oriented
from the closer point from $e$ to the farther point. Let $\boS(x)$ be
the set of successors of $x$, i.e., the points $y$ which are
neighbors of $x$ with $d(e,y) = d(e,x) + 1$.

\begin{defn}
\label{type_system}
Let $T$ be a finite set, let $t$ be a function from $\Gamma$ to $T$
and let $M$ be a square matrix indexed by $T$. We say that $(T,t,M)$
is a \emph{type system} for $(\Gamma,S)$ if, for all $i$ and $j$ in
$T$, for all but finitely many $x\in \Gamma$ with $t(x)=j$, one has
  \begin{equation*}
  \Card\{y\in \boS(x) \st t(y) = i\} = M_{ij}.
  \end{equation*}
We will often simply say that $t$ is a type system, since it
determines $T$ and $M$.
\end{defn}
In other words, if one knows the type of a point $x$, then one knows
the number of successors of each type, thanks to the matrix $M$. For
instance, in $\Fbb_d$, one can use one single type, with $M_{11} =
2d-1$: every point but the identity has $2d-1$ successors.

Using a type system, we will be able to find a lower bound for the
spectral radius of the simple random walk. While the argument works
in general, it is more convenient to formulate using an additional
assumption, which is satisfied for surface groups.
\begin{defn}
A type system $(T,t,M)$ is Perron-Frobenius if the matrix $M$ is
Perron-Frobenius, i.e., some power $M^n$ has only positive entries.
\end{defn}

The algorithm to estimate the spectral radius follows.
\begin{thm}
\label{main_alg}
Let $(\Gamma,S)$ be a countable group with a finite symmetric
generating set, whose Cayley graph has no cycle of odd length. Let
$(T,t,M)$ be a Perron-Frobenius type system for $(\Gamma,S)$.

Define a new matrix $\tilde M$ by $\tilde M_{ij} = M_{ij}/p_i$, where
$p_i$ is the number of predecessors of a point of type $i$ (it is
given by $p_j = \abs{S}-\sum_i M_{ij}$). Since it is
Perron-Frobenius, its dominating eigenvalue $e^v$ is simple. Let
$(A_1,\dotsc, A_k)$ be a corresponding eigenvector, with positive
entries, let $D$ be the diagonal matrix with entries $A_i$, and let
$M'=D^{-1/2}MD^{1/2}$. Define
  \begin{equation}
  \label{def:lambda}
  \lambda=\max_{\abs{q}=1} \langle M'q,q\rangle.
  \end{equation}
Then
  \begin{equation}
  \label{eq_main_rho}
  \rho \geq \frac{2 e^{-v/2} \lambda}{\abs{S}}.
  \end{equation}
\end{thm}
\begin{proof}
Let $s_n(i)=\Card\{x\in \Sbb^n \st t(x) = i\}$. By definition of a
type system, if $n$ is large enough (say $n\geq n_0$),
  \begin{equation*}
  p_i s_{n+1}(i) = \sum_{y\in \Sbb^n} \Card\{x\in \boS(y) \st t(x) = i\}
  = \sum_j M_{ij} s_n(j).
  \end{equation*}
This shows that $s_{n+1}=\tilde M s_n$. Therefore, the cardinality of
$\Sbb^n$ grows like $c e^{nv}$ for some $c>0$. Moreover, $s_n(i) = c'
A_i e^{nv} + O(e^{n(v-\epsilon)})$ for some $\epsilon>0$.

Take some parameters $b_1,\dotsc, b_k>0$ to be chosen later, and let
$\alpha< e^{-v/2}$. We define a function $u_\alpha$ by $u_\alpha(x) =
\alpha^n b_i$ if $x\in \Sbb^n$ and $t(x)=i$ with $n\geq n_0$. For
$n<n_0$, let $u_\alpha(x) = 0$. We have when $\alpha$ tends to
$e^{-v/2}$
  \begin{align*}
  \langle Qu_\alpha, u_\alpha\rangle
  &
  =\frac{1}{\abs{S}} \sum_{x\sim y} u_\alpha(x) u_\alpha(y)
  =\frac{2}{\abs{S}} \adjustlimits \sum_x \sum_{y\in \boS(x)} u_\alpha(x) u_\alpha(y)
  \\&
  =\frac{2}{\abs{S}} \adjustlimits \sum_{n\geq n_0} \sum_{i,j} s_n(j) b_j\alpha^n M_{ij}b_i \alpha^{n+1}
  = \frac{2\alpha}{\abs{S}} \adjustlimits \sum_{n\geq n_0} \sum_{i,j} c'A_j e^{nv} b_j\alpha^{2n} M_{ij}b_i
  +O(1)
  \\&
  = \frac{2\alpha}{\abs{S}} \sum_{i,j} c' A_j b_j M_{ij}b_i / (1-\alpha^2 e^v)
  +O(1).
  \end{align*}
On the other hand,
  \begin{align*}
  \langle u_\alpha, u_\alpha \rangle
  & = \sum_{n\geq n_0} \sum_i s_n(i) b_i^2 \alpha^{2n}
  = \sum_{n\geq n_0} \sum_i c'A_i e^{nv} b_i^2 \alpha^{2n} + O(1)
  \\&
  = \sum_i c' A_i b_i^2 /(1-\alpha^2 e^v) + O(1).
  \end{align*}
We have $\rho \geq \langle Qu_\alpha, u_\alpha\rangle/\langle
u_\alpha, u_\alpha \rangle$. Comparing the above two equations and
letting $\alpha$ tend to $e^{-v/2}$, we get
  \begin{equation*}
  \rho \geq \frac{2e^{-v/2}}{\abs{S}} \frac{ \sum_{i,j} A_j b_j M_{ij} b_i}{\sum_i A_i b_i^2}.
  \end{equation*}
To conclude, we need to optimize in $b_i$. Writing $b_i$ as
$A_i^{-1/2}c_i$, this lower bound becomes
  \begin{equation*}
  \frac{2e^{-v/2}}{\abs{S}} \frac{ \sum A_j^{1/2} c_j M_{ij} A_i^{-1/2} c_i}{\sum c_i^2}.
  \end{equation*}
The maximum of the last factor is the maximum on the unit sphere of
the quadratic form with matrix $M' = D^{-1/2} M D^{1/2}$. This
proves~\eqref{eq_main_rho}.
\end{proof}

\begin{rmk}
\label{rmk:sym}
It follows from the formula~\eqref{def:lambda} that $\lambda$ is the
maximum on the unit sphere of $\langle M'' q, q\rangle$, where $M''$
is the symmetric matrix $(M'+{M'}^*)/2$. Since any symmetric matrix
is diagonal in some orthogonal basis, it also follows that $\lambda$
is the maximal eigenvalue of $M''$, i.e., its spectral radius. Hence,
it is easy to compute using standard algorithms.
\end{rmk}

The formula given by Theorem~\ref{main_alg} depends not only on the
geometry of the group, but also on the choice of a type system: in a
given group (with a given system of generators), there may be several
type systems, giving different estimates. We will take advantage of
this fact for surface groups in Section~\ref{sec:surface}: applying
Theorem~\ref{main_alg} with the canonical type system for surface
groups, constructed by Cannon, we obtain in~\eqref{eq:lowerrho1} an
estimate for the spectral radius which is weaker than the estimate of
Theorem~\ref{main_thm}. This stronger estimate is proved by applying
Theorem~\ref{main_alg} to a different type system, constructed as a
refinement of the canonical type system.

This dependence on the choice of a type system should be contrasted
with the upper bound of Nagnibeda in~\cite{nagnibeda_rho}. Indeed, it
is shown in~\cite{nagnibeda_geometric} that this upper bound,
computed using a type system, has a purely geometric interpretation
(it is the spectral radius of a random walk on the tree of geodesics
of the group), which does not depend on the choice of the type
system. In particular, the refined type system we use to prove
Theorem~\ref{main_thm} can not improve the upper bound of Nagnibeda.

\medskip

\textbf{Acknowledgments:} We thank the anonymous referee for pointing
out the reference~\cite{nagnibeda_geometric}, and suggesting that
there might be a possible geometric interpretation to the lower bound
in Theorem~\ref{main_alg}, as in~\cite{nagnibeda_geometric}. This led
to Section~\ref{sec:geometric} below.

\section{Geometric interpretation}
\label{sec:geometric}

In this section, we describe a geometric interpretation of
Theorem~\ref{main_alg}, similar to Nagnibeda's interpretation
in~\cite{nagnibeda_geometric} of the bound she obtained
in~\cite{nagnibeda_rho}.

\medskip

We first recall Nagnibeda's construction. Consider a group $\Gamma$
with a finite system of generators $S$, whose Cayley graph has no
cycle of odd length. Let $X$ be its tree of geodesics, i.e., the
graph whose vertices are the finite geodesics in $\Gamma$ originating
from the identity $e$, and where one puts an edge from a geodesic
with length $n$ to its extensions with length $n+1$. There is a
canonical projection $\pi_X$ from $X$ to $\Gamma$, taking a geodesic
to its endpoint. One can think of $X$ as obtained from $\Gamma$ by
unfolding the loops based at $e$.

Consider the random walk in $X$ whose transitions are as follows:
from $x$, one goes to any of its successors with probability
$1/\abs{S}$, and to its unique predecessor with probability
$p_x/\abs{S}$ where $p_x$ is $\abs{S}$ minus the number of successors
of $x$ (it is the number of predecessors in $\Gamma$ of the
projection $\pi_X(x)$). This random walk on $X$ does \emph{not}
project to the simple random walk on $\Gamma$, since it does not
follow loops in $\Gamma$ (the projected random walk is not Markov in
general). The transition probabilities coincide when going towards
infinity, but not when going back towards the identity. One expects
that the probability to come back to the identity is higher in $X$
than in $\Gamma$, thanks to the following heuristic: since the
process in $X$ is less random when coming back toward the identity,
once the walk is in a subset where it comes back often to the
identity, it can not escape easily from this subset, and therefore
returns even more.

To illustrate this heuristic, suppose that two points $x$ and $x'$ in
$X$ (with $p_x=p_{x'}=2$) have successors $y$ and $y'$ in $X$, and
consider a new random walk in which $y$ and $y'$ are identified (this
is what the projection $\pi_X$ does, all over the place), so that
from this new point one can either jump back to $x$ or to $x'$ with
probability $1/\abs{S}$. Let $u_n$ and $u'_n$ be the probabilities in
$X$ to be at time $n$ at $x$ and $x'$. For the sake of the argument,
we will assume some form of symmetry, i.e., $u_n$ and $u'_n$ are also
the probabilities to reach $e$ at time $n$ starting respectively from
$x$ or $x'$. In $X$, one can form paths from $e$ to itself of length
$2n+2$ by jumping to $x$ in time $n$, then to $y$, then back to $x$,
and then from $x$ to $e$. This happens with probability
  \begin{equation*}
  u_n\cdot \frac{1}{\abs{S}} \cdot \frac{2}{\abs{S}} \cdot u_n.
  \end{equation*}
One can do the same with $x'$, giving an overall probability
$a=\frac{2}{\abs{S}^2} (u_n^2+{u'_n}^2)$. On the other hand, if $y$
and $y'$ are identified, then from this new point one can either jump
back to $x$ or to $x'$. The corresponding probability to come back to
$e$ at time $2n+2$ following such paths is therefore
$b=\frac{1}{\abs{S}^2}(u_n+u'_n)^2$. As $2(v^2+w^2) \geq (v+w)^2$, we
have $a\geq b$, i.e., the probability of returning to $e$ using
corresponding paths is bigger in $X$ than in the random walk where
$y$ and $y'$ are identified. This explains our heuristic that more
randomness in the choice of predecessors in the graph creates a
mixing effect that decreases the spectral radius.

In~\cite{nagnibeda_rho} and~\cite{nagnibeda_geometric}, Nagnibeda
justifies this heuristic rigorously as follows. Consider a group
$\Gamma$ and a generating system $S$ such that the Cayley graph of
$\Gamma$ with respect to $S$ has no cycle of odd length, and finitely
many cone types. By applying a spectral lemma of Gabber, she gets an
upper bound $\rho$ (given by a minimax formula, complicated to
estimate in general) for the spectral radius $\rho_\Gamma$ of the
simple random walk on $(\Gamma,S)$. Since the tree of geodesics $X$
also has finitely many cone types, she is able to compute exactly the
spectral radius $\rho_X$ of the random walk in $X$. It turns out that
this is exactly $\rho$. Hence, $\rho_\Gamma \leq \rho_X$, the
interest of this formula being that $\rho_X$ can be easily computed
(it is algebraic as the tree $X$ has finitely many cone types,
see~\cite{nagnibeda_tree_cone_types}). Note however that this bound
does not come from a direct argument using the projection $\pi_X:X\to
\Gamma$, but rather from two separate computations in $X$ and in
$\Gamma$.

\medskip

We now turn to a similar geometric interpretation of the lower bound
given in Theorem~\ref{main_alg}. We are looking for a natural random
walk, related to the original random walk on $\Gamma$, whose spectral
radius can be computed exactly and coincides with the lower bound
given in~\eqref{eq_main_rho}. Following the above heuristic, this
random walk should have more randomness than the original random walk
regarding the choice of predecessors, to decrease the probability to
return to the origin.

We use the setting and notations of Theorem~\ref{main_alg}. In
particular, $\Gamma$ is a group with a type system $(T,t,M)$, and
$(A_1,\dotsc, A_k)$ and $\tilde M$ are defined in the statement of
this theorem. We define a random walk as follows: It is a walk on the
space $Y = \Z \times T$ (where $T$ is the space of types), whose
transition probabilities are given by:
  \begin{equation}
  \label{eq:trans_Y}
  p( (n,j)\to (n+1,i)) = M_{ij}/\abs{S}, \quad p((n,j)\to (n-1,i))=e^{-v} \frac{A_i M_{ji}}{A_j \abs{S}}.
  \end{equation}
The spectral radius of this random walk is by definition $\rho_Y=\lim
\Pbb(X_{2n} = e)^{1/2n}$ (this is \emph{not} a spectral definition).
Note that this random walk admits a quasi-transitive $\Z$-action
(i.e., $Y$ is endowed with a free action of $\Z$, with finite
quotient, and the transition probabilities are invariant under $\Z$).
Such random walks are well studied, see for instance~\cite[Section
8.B]{woess}.

\begin{thm}
\label{thm:geometric}
With the notations of Theorem~\ref{main_alg}, the random walk on $Y$
has spectral radius $\rho_Y=2e^{-v/2} \lambda / \abs{S}$.
\end{thm}
Hence, the result of Theorem~\ref{main_alg} reads $\rho_\Gamma \geq
\rho_Y$, and $\rho_Y$ is easy to compute.

Let us first explain why this random walk is natural, and related to
the random walk on $\Gamma$. Starting from a point $x\in \Gamma$, of
type $j$ and length $n$, the original random walk goes to any of its
successors with probability $1/\abs{S}$. In particular, it reaches
points of type $i$ and length $n+1$ with probability
$M_{ij}/\abs{S}$, just like the probability given
in~\eqref{eq:trans_Y}. On the other hand, it goes to any of its
predecessors with probability $1/\abs{S}$, but the types of these
predecessors depend on $x$, not only on $j$. A random walk which is
simpler to estimate may be constructed by randomizing the
predecessors: from $x$, one chooses to go to any point of length
$n-1$ and type $i$, provided that there is an edge from type $i$ to
type $j$, i.e., $M_{ji}>0$. The probability to go from $x$ to such
points should be given by the average number of predecessors of type
$i$ to a point of type $j$. Writing $s_n(i)=\Card\{x\in \Sbb^n \st
t(x) = i\}$, this quantity is
  \begin{equation}
  \label{eq:rapport_limite}
  \frac{\sum_{\abs{x}=n, t(x)=j} \sum_{\abs{y}=n-1, t(y)=i} 1(x\in \boS(y))}{\sum_{\abs{x}=n, t(x)=j} 1}
  =\frac{ s_{n-1}(i)M_{ji}}{s_n(j)}.
  \end{equation}
As $s_n(i)\sim c' A_i e^{nv}$ for some $c'>0$ (see the proof of
Theorem~\ref{main_alg}), the quantity in~\eqref{eq:rapport_limite}
converges when $n\to\infty$ to $e^{-v} A_i M_{ji}/A_j$, giving the
transition probability~\eqref{eq:trans_Y} in the limit. Note that, in
this randomized random walk, all the points of the same length and
the same type are equivalent. Hence, we may identify them, to get a
smaller space and a simpler random walk. This is precisely our random
walk on $Y$.

Thus, the random walk on $Y$ is obtained by starting from the random
walk on $\Gamma$, randomizing the choice of predecessors, going in
the asymptotic regime $n\to \infty$, and identifying the points on
the sphere that are equivalent. One can define a projection map
$\pi_Y: \Gamma \to Y$ by $\pi_Y(x) = (\abs{x}, t(x))$, under which
the two random walks correspond in a loose sense (going towards
infinity, the transition probabilities are the same, but coming back
towards the identity they differ, just like for the projection
$\pi_X$ in Nagnibeda's construction, with more randomness in $Y$ than
in $\Gamma$).

\begin{proof}[Proof of Theorem~\ref{thm:geometric}]
We define two matrices $P^+$ and $P^-$ giving the transition
probabilities of the walk on $Y$ respectively to the right and to the
left, i.e.,
  \begin{equation*}
  P^+_{ij} = M_{ij}/\abs{S},\quad P^-_{ij}=e^{-v} \frac{A_i M_{ji}}{A_j \abs{S}}.
  \end{equation*}
In other words, $P^+=M/\abs{S}$ and $P^-=e^{-v} D M^*
D^{-1}/\abs{S}$, where $D$ is the diagonal matrix with entries $A_i$
and $M^*$ is the transpose of $M$.

Although this is clear from the geometric construction, let us first
check algebraically that the transition probabilities
in~\eqref{eq:trans_Y} indeed define probabilities, i.e., $\sum_i
P^+_{ij} + P^-_{ij}=1$ for all $j$. Let $p_j$ be the number of
predecessors of a point of type $j$ in $\Gamma$. By definition of
$A$, the matrix $\tilde M_{ji}=M_{ji}/p_j$ satisfies $\tilde M A =
e^v A$. Hence,
  \begin{equation*}
  \sum_i M_{ji} A_i = p_j \sum_i \tilde M_{ji} A_i = p_j e^v A_j.
  \end{equation*}
Therefore,
  \begin{equation*}
  \sum_i M_{ij} + \sum_i e^{-v} \frac{A_i M_{ji}}{A_j}
  = \abs{S}-p_j + e^{-v} \frac{p_j e^v A_j}{A_j} = \abs{S},
  \end{equation*}
proving that~\eqref{eq:trans_Y} defines transition probabilities.

While one can give a pedestrian proof of the equality
$\rho_Y=2e^{-v/2} \lambda / \abs{S}$, it is more efficient to use
available results of the literature. Define a function $\phi$ on $\R$
by $\phi(c)=\rho(e^c P^+ + e^{-c} P^-)$ (where this quantity is the
spectral radius of a bona fide finite dimensional matrix). It is
proved in~\cite[Proposition 8.20 and Theorem 8.23]{woess} that $\phi$
is convex, that it tends to infinity at $\pm\infty$, and that its
minimum is precisely $\rho_Y$. Since the spectral radius is invariant
under transposition and conjugation, we have
  \begin{align*}
  \abs{S}\phi(c) & = \rho( e^c M + e^{-c} e^{-v}DM^* D^{-1})
  =\rho(e^c M^* + e^{-c-v} D^{-1}M D)
  \\&
  =\rho(e^c DM^* D^{-1}+ e^{-c-v} M)
  =\abs{S} \phi(-c-v).
  \end{align*}
Hence, the function $c\mapsto \phi(c)$ is symmetric around $c=-v/2$.
As it is convex, it attains its minimum at $c=-v/2$. Therefore,
  \begin{equation*}
  \rho_Y = \phi(-v/2)
  =\frac{e^{-v/2}}{\abs{S}}\rho(M +D M^* D^{-1})
  =\frac{e^{-v/2}}{\abs{S}}\rho(D^{-1/2}MD^{1/2} +D^{1/2} M^*D^{-1/2}).
  \end{equation*}
By Remark~\ref{rmk:sym}, the last term is equal to $2\lambda$.
\end{proof}

\section{Application to surface groups}

\label{sec:surface}

\subsection{Cannon's types}

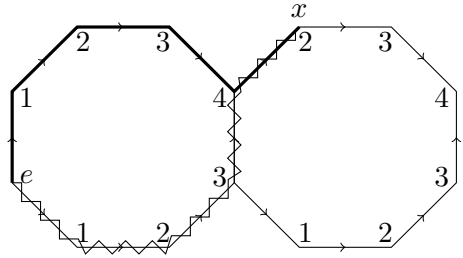
\begin{figure}[htb]
\centering
\begin{tikzpicture}
  \RegPolygonOriented{(0, 0)}{S}{90} {e, 1, 2, 3, 4, 3, 2, 1}{1}{5}{9};
  \RegPolygonOriented{(S5)}  {T}{45} { , 2, 3, 4, 3, 2, 1,  }{1}{4}{8};
  \path (T2) node[above] {$x$};
  \draw[very thick]
      (S1) -- (S2) -- (S3) -- (S4) -- (S5) -- (T2);
  \draw[decorate,decoration=zigzag]
      (S1) -- (S8) -- (S7) -- (S6) -- (S5) -- (T2);
\end{tikzpicture}
\caption{Two examples of geodesics from $e$ to $x$}
\label{fig:geod}
\end{figure}

Consider a countable group with the word distance coming from a
finite generating set $S$. The \emph{cone} of a point $x$ is the set
of points $y$ for which there is a geodesic from $e$ to $y$ going
through $x$. The \emph{cone type} of $x$ is the set $\{x^{-1}y\}$,
for $y$ in the cone of $x$. Note that knowing the cone type of a
point determines the number of its successors, and the number of its
successors having any given cone type. Cannon proved that, in any
hyperbolic group, there are finitely many cone types. Therefore, such
a group admits a type system in the sense of
Definition~\ref{type_system}. This is in particular the case of the
surface groups $\Gamma_g$. However, the number of cone types is too
large, and it is more convenient for practical purposes to reduce
them using symmetries. We obtain Cannon's canonical types for the
surface groups, described in~\cite{cannon_unpublished}
or~\cite{floyd_plotnick} as follows.

The hyperbolic plane can be tessellated by regular $4g$-gons, with
$4g$ of them around each vertex. The Cayley graph of $\Gamma_g$ (with
its usual presentation~\eqref{eq:presentation}) is dual to this
(self-dual) tessellation, and is therefore isomorphic to it. Define
the type of a point $x\in \Gamma_g$ as the maximal length along the
last $4g$-gon of a geodesic starting from $e$ and ending at $x$.
Beware that one really has to take the maximum: for instance, in
Figure~\ref{fig:geod}, the thick geodesic from $e$ to $x$ shares only
one edge with the last octagon, while the wiggly one shares two
edges. Hence, the type of $x$ is $2$.

The type can also be described combinatorially as follows: write
$x=c_1\dotsm c_{\lgth{x}}$ as a product of minimal length in the
generators $a_1,\dotsc, b_g$, look at the length $n$ of its longest
common suffix with a fundamental relator (i.e., a cyclic permutation
of the basic relation $[a_1, b_1]\dotsm [a_g, b_g]$
in~\eqref{eq:presentation} or its inverse: $c_{\lgth{x}-n+1}\dotsm
c_{\lgth{x}}$ should be a subword of the basic relation or its
inverse, up to cyclic permutation), and take the maximum of all such
$n$ over all ways to write $x=c_1\dotsm c_{\lgth{x}}$. It is obvious
that the geometric and combinatorial descriptions are equivalent, we
will mostly rely on the geometric one.

The type of a group element $x$ can be at most $2g$ (otherwise,
taking the same path but going the other way around the last
$4g$-gon, one would get a strictly shorter path, contradicting the
fact that the initial path is geodesic), and it is $0$ only for the
identity. Points $x$ of type $i < 2g$ have only one predecessor, and
$4g-1$ successors. Among them, $2$ are followers of $x$ on the
$4g$-gon to its left and to its right, while the other ones
correspond to newly created $4g$-gons (whose closest point to $e$ is
$x$). It follows that those points have $4g-3$ successors of type
$1$, one of type $2$ and one of type $i+1$. Points of type $2g$ are
special, since they have two predecessors (one can reach them with a
geodesic either from the left or from the right of a single
$4g$-gon). They have $2$ successors of type $2$, corresponding to the
extremal outgoing edges of $x$ (they extend the two $4g$-gons
adjacent to both incoming edges to $x$), and the $4g-4$ remaining
successors are on newly created $4g$-gons, and are of type $1$. See
Figure~\ref{fig:types} for an illustration in genus $2$ (of course,
additional octagons should be drawn around all outgoing edges, but
since this is notoriously difficult to do in a Euclidean drawing, we
have to rely on the reader's imagination).

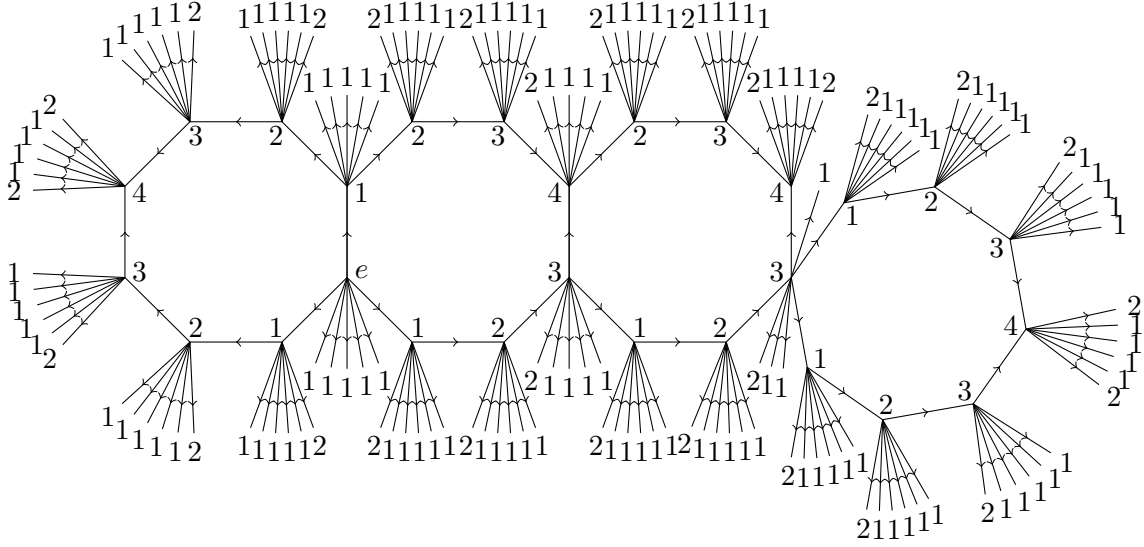
\begin{figure}[htp]
\centering
\begin{tikzpicture}
  \RegPolygonOriented{(0, 0)}{S}{90} {e, 1, 2, 3, 4, 3, 2, 1}{1}{5}{9};
  \RegPolygonOriented{(S5)}  {T}{45} { , 2, 3, 4, 3, 2, 1,  }{1}{4}{8};
  \RegPolygonOriented{(S1)}  {U}{-135}{ , 1, 2, 3, 4, 3, 2, }{1}{5}{8};
  \RegPolygonOriented{(T5)}  {V}{55} { , 1, 2, 3, 4, 3, 2, 1}{1}{5}{9};
  \Faisceau{(S2)}{5}{90}   {40}{1,1,1,1,1};
  \Faisceau{(S3)}{6}{90}   {40}{1,1,1,1,1,2};
  \Faisceau{(S4)}{6}{90}   {40}{1,1,1,1,1,2};
  \Faisceau{(S5)}{5}{90}   {40}{1,1,1,1,2};
  \Faisceau{(T2)}{6}{90}   {40}{1,1,1,1,1,2};
  \Faisceau{(T3)}{6}{90}   {40}{1,1,1,1,1,2};
  \Faisceau{(T4)}{6}{90}   {40}{2,1,1,1,1,2};
  \Faisceau{(T6)}{6}{-92}  {40}{2,1,1,1,1,1};
  \Faisceau{(T7)}{6}{-90}  {40}{2,1,1,1,1,1};
  \Faisceau{(S6)}{5}{-90}  {40}{2,1,1,1,1};
  \Faisceau{(S7)}{6}{-90}  {40}{2,1,1,1,1,1};
  \Faisceau{(S8)}{6}{-90}  {40}{2,1,1,1,1,1};
  \Faisceau{(S1)}{5}{-90}  {40}{1,1,1,1,1};
  \Faisceau{(U2)}{6}{-90}  {40}{1,1,1,1,1,2};
  \Faisceau{(U3)}{6}{247.5}{50}{1,1,1,1,1,2};
  \Faisceau{(U4)}{6}{202.5}{50}{1,1,1,1,1,2};
  \Faisceau{(U5)}{6}{157.5}{50}{2,1,1,1,1,2};
  \Faisceau{(U6)}{6}{112.5}{50}{2,1,1,1,1,1};
  \Faisceau{(U7)}{6}{90}   {40}{2,1,1,1,1,1};
  \Faisceau{(T5)}{1}{(55+90)/2}{0}{1};
  \Faisceau{(T5)}{3}{(55+90)/2+180+5}{15}{2,1,1};
  \Faisceau{(V2)}{6}{55}{40}{1,1,1,1,1,2};
  \Faisceau{(V3)}{6}{55}{40}{1,1,1,1,1,2};
  \Faisceau{(V4)}{6}{32.5} {50}{1,1,1,1,1,2};
  \Faisceau{(V5)}{6}{-12.5} {50}{2,1,1,1,1,2};
  \Faisceau{(V6)}{6}{-57.5}{50}{2,1,1,1,1,1};
  \Faisceau{(V7)}{6}{-80}{40}{2,1,1,1,1,1};
  \Faisceau{(V8)}{6}{-80}{40}{2,1,1,1,1,1};
\end{tikzpicture}
\caption{Types of the points in (part of) the Cayley graph of $\Gamma_2$}
\label{fig:types}
\end{figure}

Keeping only the types from $1$ to $2g$ (since type $0$ only happens
for the identity, while Definition~\ref{type_system} allows to
discard finitely many points), we obtain a type system for $\Gamma_g$
with $T=\{1,\dotsc, 2g\}$, where the matrix $M$ has been described in
the previous paragraph. For instance, in genus $2$,
  \begin{equation*}
  M = \left(\begin{matrix}
  5 & 5 & 5 & 4 \\
  2 & 1 & 1 & 2 \\
  0 & 1 & 0 & 0 \\
  0 & 0 & 1 & 0
  \end{matrix}\right).
  \end{equation*}
One can now apply the algorithm of Theorem~\ref{main_alg} to this
matrix to bound the spectral radius of the simple random walk from
below. All points but points of type $4$ have $1$ predecessor, so the
matrix $\tilde M$ is
  \begin{equation*}
  \tilde M = \left(\begin{matrix}
  5 & 5 & 5 & 4 \\
  2 & 1 & 1 & 2 \\
  0 & 1 & 0 & 0 \\
  0 & 0 & 1/2 & 0
  \end{matrix}\right).
  \end{equation*}
The dominating eigenvalue of this matrix is $e^v = 6.979835\dotsc$
(this is also the growth of the group), while the corresponding
eigenvector is
  \begin{equation*}
  A = (0.715987\dotsc, 0.246211\dotsc, 0.035274\dotsc, 0.002526\dotsc).
  \end{equation*}
One gets that the matrix $M''$ of Remark~\ref{rmk:sym} is
  \begin{equation*}
  M''
  = \left(\begin{matrix}
  5              & 3.171316\dotsc & 0.554905\dotsc & 0.118814\dotsc \\
  3.171316\dotsc & 1              & 1.510223\dotsc & 0.101307\dotsc \\
  0.554905\dotsc & 1.510223\dotsc & 0              & 1.868132\dotsc \\
  0.118814\dotsc & 0.101307\dotsc & 1.868132\dotsc & 0
  \end{matrix}\right),
  \end{equation*}
with dominating eigenvalue $\lambda = 7.000902\dotsc$. Finally,
  \begin{equation}
  \label{eq:lowerrho1}
  \rho \geq \frac{2 e^{-v/2} \lambda}{8} = 0.662477\dotsc.
  \end{equation}
This is already slightly better than Bartholdi's estimate $\rho \geq
0.662421$, but much weaker than the estimate $\rho\geq 0.662772$ that
we claimed in Theorem~\ref{main_thm} (to be compared with the
``true'' value $\rho \sim 0.662812$).

In the next sections, we will explain how to get better estimates by
using different type systems, that distinguish between more points
(but, of course, give rise to larger matrices $M$ and therefore to
more computer-intensive computations).

\begin{rmk}
Using cone types instead of Cannon's canonical types does not give
rise to better estimates for $\rho$ (although the number of types is
much larger). Indeed, if some type system $t'$ is obtained from some
type system $t$ by quotienting by some symmetries of $t$, then the
dominating eigenvector of $\tilde M(t)$, being unique, is invariant
by those symmetries and reduces to the dominating eigenvector of
$\tilde M(t')$. It follows that the dominating eigenvector of
$M''(t)$ is also invariant by those symmetries, and that the
dominating eigenvalue of $M''(t)$ is the same as that of $M''(t')$.
Hence, the estimates on $\rho$ given by Theorem~\ref{main_thm} for
$t$ and $t'$ are the same.
\end{rmk}

\subsection{Suffix types}
\label{subsec_suffix}

There are many ways to define new type systems in surface groups,
that separate more points. If a type system is finer than another
one, then the estimate on the spectral radius coming from
Theorem~\ref{main_alg} is better, but the matrix involved in the
computation is larger. To get manageable estimates, we should find
the right balance.

In this paragraph, we describe a very simple extension of Cannon's
canonical type systems in surface groups, that we call suffix types.
Given a point $x\in \Gamma_g$, there can be several geodesics from
$e$ to $x$. Consider the longest ending that is common to all these
geodesics, say $x_{n-k+1},\dotsc, x_n$ (with $x_n=x$), and define the
suffix type of $x$ to be
  \begin{equation*}
  \tsuffix(x)=(t(x_n), t(x_{n-1}),\dotsc, t(x_{n-k+1})),
  \end{equation*}
where $t$ is the canonical type of Cannon.

For any $x$, $\tsuffix(x)$ is easy to compute inductively:
\begin{itemize}
\item If $t(x)=0$, i.e., $x=e$, then $\tsuffix(x)=(0)$.
\item If $x$ is of type $2g$, it has two predecessors, so the
    common ending to all geodesics ending at $x$ is simply $x$,
    and $\tsuffix(x)=(2g)$.
\item If $t(x) \in \{1,\dotsc, 2g-1\}$, then $x$ has a unique
    predecessor $z$. The common ending to all geodesics ending at
    $x$ is the common ending to all geodesics ending at $z$,
    followed with $x$. Hence, $\tsuffix(x) = (t(x),
    \tsuffix(z))$.
\end{itemize}
It also follows from this description that, if one knows
$\tsuffix(x)$, it is easy to determine $\tsuffix(y)$ for any
successor $y$ of $x$: if $t(y)=2g$, then $\tsuffix(y)=(2g)$,
otherwise $x$ is the only predecessor of $y$ and $\tsuffix(y)=(t(y),
\tsuffix(x))$.

We have shown that $\tsuffix$ shares most properties of type systems
as described in Definition~\ref{type_system}, except that it does not
take its values in a finite set. To ensure this additional property,
one should truncate the suffix type. For instance, one can fix some
maximal length $k$, and define the $k$-truncated suffix type
$\tsuffix^{(k)}(x)$ by keeping only the first $k$ elements of
$\tsuffix(x)$ if its length is $>k$.

The following proposition is obvious from the previous discussion.
\begin{prop}
For any $k\geq 1$, the $k$-truncated suffix type system
$\tsuffix^{(k)}$ is a (Perron-Frobenius) type system in the sense of
Definition~\ref{type_system}.
\end{prop}

The matrix size increases with $k$, but the estimates on the spectral
radius following from Theorem~\ref{main_alg} get better. For
instance, in $\Gamma_2$, for $k=5$, the matrix size is $148$, and we
get $\rho \geq 0.662694$.

A drawback of this truncation process is that it truncates uniformly,
independently of the likeliness of the type, while it should be more
efficient to extend mostly those types that are more likely to
happen. This intuition leads to another truncation process: fix a
system of weights $w=(w_0,\dotsc, w_{2g}) \in [0,+\infty)^{2g+1}$, a
threshold $k$, and truncate a suffix type $(t_0,t_1,\dotsc)$ at the
smallest $n$ such that $t_0+\dotsc+t_n > k$. This gives another type
system denoted by $\tsuffix^{(k,w)}$ ($\tsuffix^{(k)}$ corresponds to
the weights $w=(1,\dotsc, 1)$ and the threshold $k-1$). Define for
instance a weight system $\bar w$ by $\bar w_0=1$ and $\bar w_i=i$
for $i\geq 1$: the corresponding type system $\tsuffix^{(k, \bar w)}$
truncates more quickly the suffix types involving a lot of large
types, that happen less often in the group. Hence, it should give a
smaller matrix than the naive truncation only according to length,
while retaining a comparatively good estimate for the spectral
radius.

This intuition is correct: for instance, in $\Gamma_2$, using
$\tsuffix^{(k, \bar w)}$ with $k=6$, one gets a matrix with size
$109$ and an estimate $\rho\geq 0.662697$: the matrix is smaller than
for the naive truncation $\tsuffix^{(5)}$, while the estimate on the
spectral radius is better.

We can now push the computations, to a larger matrix size: using in
$\Gamma_2$ the weight $\bar w$ and the truncation threshold $k=25$,
one obtains a type system where the matrix is of size $2,774,629$,
and the following estimate on the spectral radius.
\begin{prop}
In $\Gamma_2$, one has $\rho \geq 0.662757$.
\end{prop}
This is definitely better than~\eqref{eq:lowerrho1}, but not yet as
good as Theorem~\ref{main_thm}.

\bigskip

A few comments on the practical implementation. There are three main
steps in the algorithm of Theorem~\ref{main_alg}:
\begin{enumerate}
\item Compute the matrix $M$ corresponding to the type system.
\item Find the eigenvector $A$, to define the matrix $M'$.
\item Find the maximal expansion rate of $M'$.
\end{enumerate}
Computing the matrix of the type system is a matter of simple
combinatorics: we explained above all the transitions from one suffix
type to the next ones. The resulting matrix $M$ is very sparse: each
type has at most $2g$ successors. However, it is extremely large, so
that finding the eigenvector $A$ and then the maximal expansion rate
of $M'$ might seem computationally expensive. This is not the case,
as we explain now.

Let $A^{(0)}$ be the eigenvector for the original Cannon type, so
that $\Card\{x\in \Sbb^n \st t(x)=i\} \sim  A^{(0)}_i e^{nv}$. Let
also $M^{(0)}$ be the matrix for the original Cannon type. Given a
new type $\bar i = (i_0,\dotsc, i_m)$, the entry $A_{\bar i}$ of the
eigenvector $A$ for the new type $\tsuffix^{(w,k)}$ is such that
$\Card\{x\in \Sbb^n \st \tsuffix^{(w,k)}(x)=\bar i\} \sim A_{\bar i}
e^{nv}$. Such a point $x$ can be obtained uniquely by starting from a
point $y\in \Sbb^{n-m}$ with type $i_m$, and then taking successors
respectively of type $i_{m-1},\dotsc, i_0$. Hence,
  \begin{multline}
  \label{eq:bijection}
  \Card\{x\in \Sbb^n \st \tsuffix^{(w,k)}(x)=(i_0,\dotsc, i_m)\}
  \\ = M^{(0)}_{i_0 i_1}\dotsm M^{(0)}_{i_{m-1}i_m}\Card\{y\in \Sbb^{n-m} \st t(y)=i_m\}.
  \end{multline}
It follows that the new eigenvector is given by
  \begin{equation*}
  A_{\bar i} = M^{(0)}_{i_0 i_1}\dotsm M^{(0)}_{i_{m-1}i_m} A^{(0)}_{i_m} e^{-mv}.
  \end{equation*}
This shows that $A$ is very easy to compute.

By Remark~\ref{rmk:sym}, to determine the maximal expansion rate
$\lambda$ of $M'$, it suffices the find the maximal eigenvalue of
$M''=(M'+{M'}^*)/2$. This matrix is real, symmetric, with nonnegative
coefficients, and it is Perron-Frobenius (i.e., it has one single
maximal eigenvalue). It follows that, for any vector $v$ with
positive coefficients, $\lambda = \lim \norm{{M''}^n
v}/\norm{{M''}^{n-1}v}$ (and moreover this sequence is
non-decreasing, see for instance~\cite[Corollary 10.2]{woess}).
Hence, one can readily estimate $\lambda$ from below, by starting
from a fixed vector $v$ and iterating $M''$. Again, there is no issue
of instability or complexity.

\subsection{Essential types}

To improve the suffix types, to separate even more points, one can
for instance replace the canonical Cannon types with the true cone
types. However, the matrix size increases so quickly that this is not
usable in practice. Moreover, this does not solve the main problem of
suffix types: they do not separate at all points with Cannon type
$2g$, although such points are clearly not always equivalent. In this
paragraph, we introduce a new type system that can separate such
points, that we call the \emph{essential type}.

The basic idea (that will not work directly) is to memorize not only
the common ending of all geodesics ending at a point $x$, but all the
parts that are common to such geodesics: i.e., the sequence
$F_{\mathrm{ess}}(x)=(x_0=e, x_1,\dotsc, x_n=x)$ (with $n=\lgth{x}$)
where $x_i=*$ if there are two geodesics from $e$ to $x$ that differ
at position $i$, and $x_i$ is the point that is common to all those
geodesics at position $i$ otherwise. We then associate to $x$ the
sequence $\tessential(x)=(t(x_n), t(x_{n-1}),\dotsc, t(x_0))$ where
$t(x_i)$ is the Cannon type of $x_i$ if $x_i\not=*$, and $t(*)=*$.

\begin{figure}[htb]
\centering
\begin{tikzpicture}
  \RegPolygonOriented{(0, 0)}{S}{90} {e, 1, 2, 3, 4, 3, 2, 1}{1}{5}{9};
  \RegPolygonOriented{(S4)}  {T}{120}{ , 1, 2, 3, 4, 3, 2, 1}{1}{5}{9};
  \Faisceau{(S4)}{1}{-30}{0}{1};
  \Faisceau{(S4)}{3}{150}{30}{1,1,2};
  \RegPolygonOriented{(6, 0)}{U}{90} {e, 1, 2, 3, 4, 3, 2, 1}{1}{5}{9};
  \RegPolygonOriented{(U4)}  {V}{90} { , 1, 2, 3, 4, 3, 2,  }{1}{5}{8};
  \path (T4) node[above] {$x$}
        (T5) node[above] {$y$}
        (V4) node[above] {$x'$}
        (V5) node[above right] {$y'$};
  \draw[very thick]
      (S1) -- (S2) -- (S3) -- (S4) -- (T2) -- (T3) -- (T4) -- (T5);
  \draw[decorate,decoration=zigzag]
      (S1) -- (S2) -- (S3) -- (S4) -- (T8) -- (T7) -- (T6) -- (T5);
  \draw[very thick]
      (U1) -- (U2) -- (U3) -- (U4) -- (V2) -- (V3) -- (V4) -- (V5);
  \draw[decorate,decoration=zigzag]
      (U1) -- (U8) -- (U7) -- (U6) -- (U5) -- (V7) -- (V6) -- (V5);
\end{tikzpicture}
\caption{The points $x$ and $x'$ have the same essential type, contrary to $y$ and $y'$.}
\label{fig:tessentialnotunique}
\end{figure}
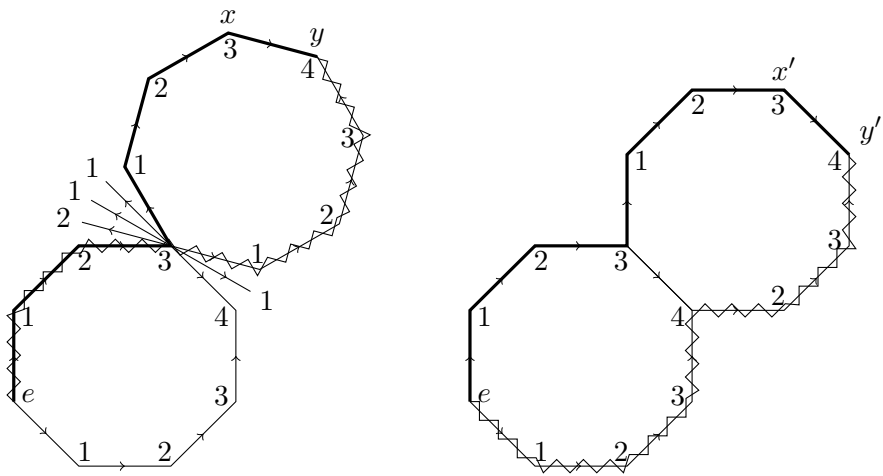

The problem with this notion is that $\tessential(x)$ does not
determine $\tessential(y)$ for $y$ a successor of $x$: in
Figure~\ref{fig:tessentialnotunique}, the points $x$ and $x'$ have
the same essential type $(3,2,1,3,2,1,0)$, while their successors $y$
and $y'$ have respective essential types $(4,*,*,*,3,2,1,0)$ and
$(4,*,*,*,*,*,*,0)$ (this follows from the fact that the thick paths
and wiggly paths are geodesics). This shows that, as we have defined
it, $\tessential$ can not be used to define a type system.

This problem can be solved if we do not use directly the Cannon types
in the definition of $\tessential$, but a slightly refined notion,
the Cannon modified type, taking values in $\{1,\dotsc, 2g, 1',
2'\}$. The modified type of a point is the same as its Cannon type,
except for some points of Cannon types $1$ and $2$, that have
modified types $1'$ and $2'$ respectively: considering any point $y$
of type $2g-1$, it has a unique successor $z$ of type $2g$, and a
unique successor $x$ of type $1$ that is on the same $4g$-gon as $z$.
We say that $x$ is of modified type $1'$. Moreover, $x$ has a unique
successor of type $2$ that is also on the same $4g$-gon as $z$, we
say that it is of modified type $2'$. See
Figure~\ref{fig:types_prime} for an example in genus $2$. By
definition, the modified Cannon type $t'$ is also a type system. The
transition matrix is the same as for the usual Cannon type, except
that
\begin{itemize}
\item A point of type $2g-1$ has one successor of type $1'$, one
    successor of type $2g$, one successor of type $2$, and $4g-4$
    successors of type $1$.
\item A point of type $1'$ has one successor of type $2$, one
    successor of type $2'$, and $4g-3$ successors of type $1$.
\item A point of type $2'$ has one successor of type $2$, one
    successor of type $3$, and $4g-3$ successors of type $1$.
\end{itemize}

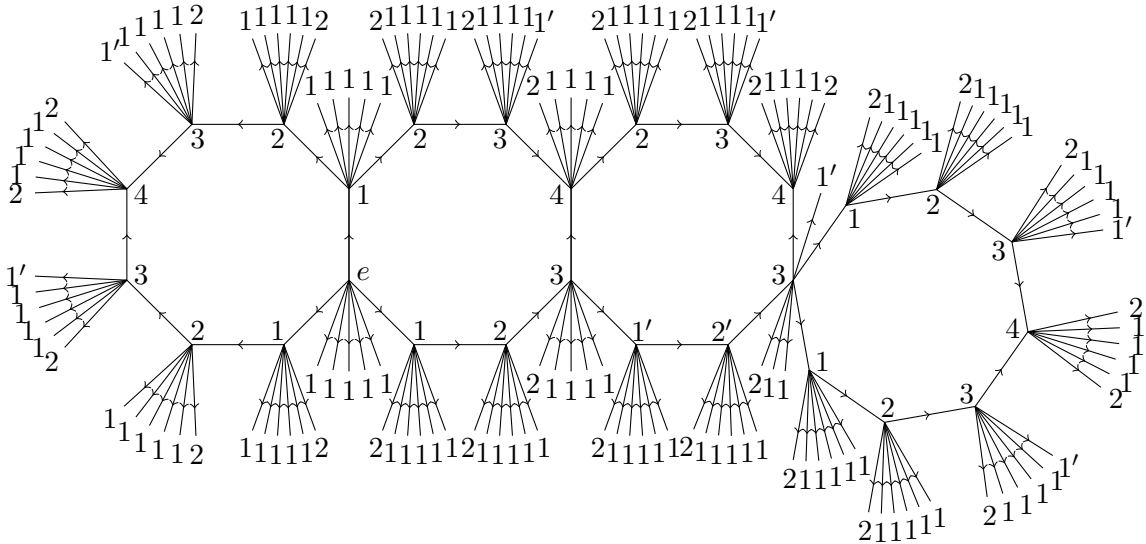
\begin{figure}[htp]
\centering
\begin{tikzpicture}
  \RegPolygonOriented{(0, 0)}{S}{90} {e, 1, 2, 3, 4, 3, 2, 1}{1}{5}{9};
  \RegPolygonOriented{(S5)}  {T}{45} { , 2, 3, 4, 3, 2', 1',  }{1}{4}{8};
  \RegPolygonOriented{(S1)}  {U}{-135}{ , 1, 2, 3, 4, 3, 2, }{1}{5}{8};
  \RegPolygonOriented{(T5)}  {V}{55} { , 1, 2, 3, 4, 3, 2, 1}{1}{5}{9};
  \Faisceau{(S2)}{5}{90}   {40}{1,1,1,1,1};
  \Faisceau{(S3)}{6}{90}   {40}{1,1,1,1,1,2};
  \Faisceau{(S4)}{6}{90}   {40}{1',1,1,1,1,2};
  \Faisceau{(S5)}{5}{90}   {40}{1,1,1,1,2};
  \Faisceau{(T2)}{6}{90}   {40}{1,1,1,1,1,2};
  \Faisceau{(T3)}{6}{90}   {40}{1',1,1,1,1,2};
  \Faisceau{(T4)}{6}{90}   {40}{2,1,1,1,1,2};
  \Faisceau{(T6)}{6}{-92}  {40}{2,1,1,1,1,1};
  \Faisceau{(T7)}{6}{-90}  {40}{2,1,1,1,1,1};
  \Faisceau{(S6)}{5}{-90}  {40}{2,1,1,1,1};
  \Faisceau{(S7)}{6}{-90}  {40}{2,1,1,1,1,1};
  \Faisceau{(S8)}{6}{-90}  {40}{2,1,1,1,1,1};
  \Faisceau{(S1)}{5}{-90}  {40}{1,1,1,1,1};
  \Faisceau{(U2)}{6}{-90}  {40}{1,1,1,1,1,2};
  \Faisceau{(U3)}{6}{247.5}{50}{1,1,1,1,1,2};
  \Faisceau{(U4)}{6}{202.5}{50}{1',1,1,1,1,2};
  \Faisceau{(U5)}{6}{157.5}{50}{2,1,1,1,1,2};
  \Faisceau{(U6)}{6}{112.5}{50}{2,1,1,1,1,1'};
  \Faisceau{(U7)}{6}{90}   {40}{2,1,1,1,1,1};
  \Faisceau{(T5)}{1}{(55+90)/2}{0}{1'};
  \Faisceau{(T5)}{3}{(55+90)/2+180+5}{15}{2,1,1};
  \Faisceau{(V2)}{6}{55}{40}{1,1,1,1,1,2};
  \Faisceau{(V3)}{6}{55}{40}{1,1,1,1,1,2};
  \Faisceau{(V4)}{6}{32.5} {50}{1',1,1,1,1,2};
  \Faisceau{(V5)}{6}{-12.5} {50}{2,1,1,1,1,2};
  \Faisceau{(V6)}{6}{-57.5}{50}{2,1,1,1,1,1'};
  \Faisceau{(V7)}{6}{-80}{40}{2,1,1,1,1,1};
  \Faisceau{(V8)}{6}{-80}{40}{2,1,1,1,1,1};
\end{tikzpicture}
\caption{Modified Cannon types of the points in (part of) the Cayley graph of $\Gamma_2$}
\label{fig:types_prime}
\end{figure}

We define $\tessential(x)=(t'(x_n),\dotsc, t'(x_0))$ where
$(x_0,\dotsc, x_n)=F_{\mathrm{ess}}(x)$ and $t'(*)=*$.

\begin{prop}
\label{prop:tessential_type}
The essential type $\tessential(x)$ of a point $x$ determines the
essential type of its successors.
\end{prop}
\begin{proof}
We argue by induction on $\lgth{x}=d(x,e)$.

Consider a point $x$, and one of its successors $y$. If the type of
$y$ is not $2g$, then $x$ is the unique predecessor of $y$, and
$\tessential(y)=(t'(y), \tessential(x))$. Assume now that $t'(y)=2g$
(so that $t'(x)=2g-1$). Let $z_{2g-1}=x$, and define inductively
$z_i$ as the unique predecessor of $z_{i+1}$ for $i\geq 1$, so that
$F(x)=(e,\dotsc, z_1,z_2,\dotsc, z_{2g-1}=x)$. Those points are on a
common $4g$-gon $R$, and $t(z_i)=i$ for $i\geq 2$, while $t(z_1)$ can
be anything. In the same way, let $\tilde z_{2g-1},\dotsc, \tilde
z_1$ be the successive preimages of $y$ going around $R$ in the other
direction. They also satisfy $t(\tilde z_i)=i$ for $i\geq 2$.

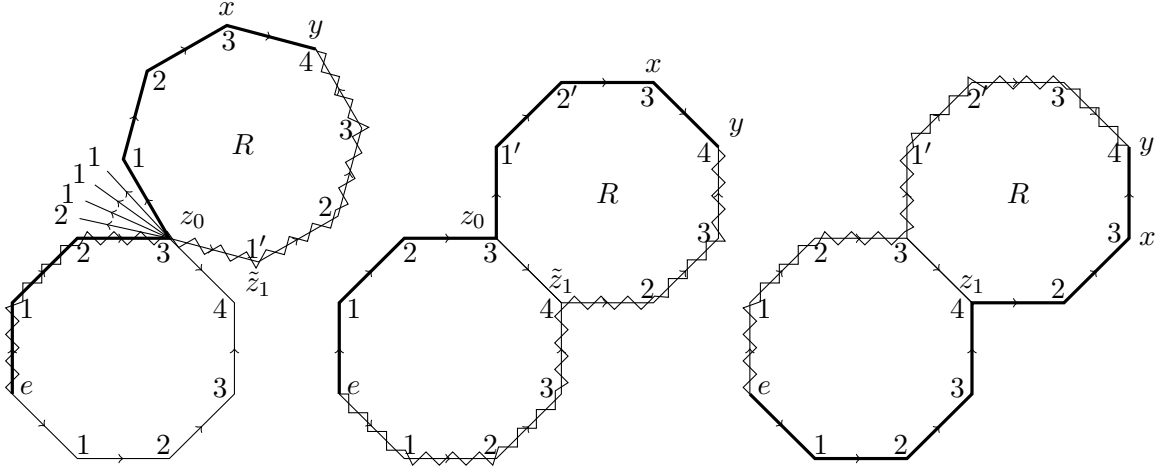
\begin{figure}[htb]
\centering
\begin{tikzpicture}
  \RegPolygonOriented{(0, 0)}{S}{90} {e, 1, 2, 3, 4, 3, 2, 1}{1}{5}{9};
  \RegPolygonOriented{(S4)}  {T}{120}{ , 1, 2, 3, 4, 3, 2, 1'}{1}{5}{9};
  \Faisceau{(S4)}{4}{150}{35}{1,1,1,2};
  \RegPolygonOriented{(4.3, 0)}{U}{90} {e, 1 , 2 , 3, 4, 3, 2, 1}{1}{5}{9};
  \RegPolygonOriented{(U4)}    {V}{90} { , 1', 2', 3, 4, 3, 2,  }{1}{5}{8};
  \RegPolygonOriented{(9.7, 0)}{W}{90} {e, 1 , 2 , 3, 4, 3, 2, 1}{1}{5}{9};
  \RegPolygonOriented{(W4)}    {X}{90} { , 1', 2', 3, 4, 3, 2,  }{1}{5}{8};
  \path (S4) node[above right] {$z_0$}
        (T8) node[below] {$\tilde z_1$}
        (T4) node[above] {$x$}
        (T5) node[above] {$y$}
        (U4) node[above left] {$z_0$}
        (V4) node[above] {$x$}
        (V5) node[above right] {$y$}
        (V8) node[above] {$\tilde z_1$}
        (W5) node[above] {$z_1$}
        (X6) node[right] {$x$}
        (X5) node[right] {$y$}
        (T0) node {$R$}
        (V0) node {$R$}
        (X0) node {$R$};
  \draw[very thick]
      (S1) -- (S2) -- (S3) -- (S4) -- (T2) -- (T3) -- (T4) -- (T5);
  \draw[decorate,decoration=zigzag]
      (S1) -- (S2) -- (S3) -- (S4) -- (T8) -- (T7) -- (T6) -- (T5);
  \draw[very thick]
      (U1) -- (U2) -- (U3) -- (U4) -- (V2) -- (V3) -- (V4) -- (V5);
  \draw[decorate,decoration=zigzag]
      (U1) -- (U8) -- (U7) -- (U6) -- (U5) -- (V7) -- (V6) -- (V5);
  \draw[decorate,decoration=zigzag]
      (W1) -- (W2) -- (W3) -- (W4) -- (X2) -- (X3) -- (X4) -- (X5);
  \draw[very thick]
      (W1) -- (W8) -- (W7) -- (W6) -- (W5) -- (X7) -- (X6) -- (X5);
\end{tikzpicture}
\caption{Determining the essential type of $y$ from that of $x$}
\label{fig:tessentialunique}
\end{figure}

If $t(z_1)\not =2g$, then $z_1$ has a unique preimage $z_0$, which
also belongs to $R$. Moreover, $z_0$ is the unique closest point to
$e$ on $R$. The path $P=(z_0,z_1,\dotsc, z_ {2g-1}, y)$ is a geodesic
path going around $R$ in one direction, and $\tilde P=(z_0, \tilde
z_1,\dotsc, \tilde z_{2g-1}, y)$ is also geodesic and goes around $R$
in the other direction. If, along $\tilde P$, all points different
from $z_0,y$ have type $<2g$ (so that they have a unique preimage),
then any geodesic from $e$ to $y$ has to follow either $P$ or $\tilde
P$, so that $\tessential(y) = (2g, *,\dotsc,*, \tessential(z_0))$,
with $2g-1$ ambiguous points. See the first part of
Figure~\ref{fig:tessentialunique} for an illustration.

The only way to have a point of type $2g$ along $\tilde P$ is if
$t(z_0)=2g-1$ and $t(\tilde z_1)=2g$, since $t(\tilde z_i)=i$ for
$i\geq 2$ (see the second part of Figure~\ref{fig:tessentialunique}).
By definition of the modified type, this happens exactly when
$t'(z_1)=1'$ and $t'(z_2)=2'$. In this case, a geodesic from $e$ to
$y$ can either go through $z_0$ and then follow $P$, or go through
$\tilde z_1$ and then follow $(\tilde z_2,\dotsc, \tilde z_{2g-1},
y)$. In the first case, if one truncates the geodesic when it reaches
$z_0$ and then adds the edge $[z_0\tilde z_1]$, one gets a geodesic
from $e$ to $\tilde z_1$. It follows that the essential type of $y$
is $(2g, *,\dotsc, *, \widehat\tessential(\tilde z_1))$, where
$\widehat\tessential(\tilde z_1)$ is the essential type of $\tilde
z_1$ minus its first entry (i.e., the type $2g$ of $\tilde z_1$), and
where there are $2g-1$ stars. Since $\tilde z_1$ is a successor of
$z_0$, the induction hypothesis ensures that its essential type can
be determined from that of $z_0$, which is given by that of $x$. When
$t(z_1)\not=2g$, we have shown that in all situations
$\tessential(x)$ determines $\tessential(y)$ in an algorithmic way.

Assume now that $t(z_1)=2g$ (see the third part of
Figure~\ref{fig:tessentialunique}). This case is very similar to the
previous one. To reach $y$, one has to reach $z_1$ or its preimage on
$R$, and then reach $y$ by going around $R$ in one direction or the
other. It follows that $\tessential(y) = (2g, *,\dotsc, *,
\widehat\tessential(z_1))$ as above. Since $\tessential(z_1)$ is
obtained by removing the last $2g-2$ entries of $\tessential(x)$, it
follows again that $\tessential(x)$ determines $\tessential(y)$.
\end{proof}

For $k>0$, let $\tessential^{(k)}(x)$ be the truncated essential
type, obtained by keeping the first $k$ entries of the essential type
of $x$. The above proof also shows that $\tessential^{(k)}(x)$
determines $\tessential^{(k+1)}(y)$ (and therefore
$\tessential^{(k)}(y)$) for any successor $y$ of $x$. In the same
way, if one considers a truncation according to a weight $w=(w_0,
w_1,\dotsc, w_{2g}, w_{1'}, w_{2'},w_*)$ and a threshold $k$, then
$\tessential^{(k,w)}(x)$ determines $\tessential^{(k,w)}(y)$ for any
successor $y$ of $x$, if the weight $w_*$ is maximal among all
weights. This last requirement is necessary for the following reason:
the essential type of $y$ can be obtained from that of $x$ by
replacing some entries with stars; if this could decrease the weight
of the resulting sequence, one might need to look further to
determine $\tessential^{(k,w)}(y)$, and $\tessential^{(k,w)}(x)$
might not be sufficient.

Under these conditions, it follows that $\tessential^{(k)}$ and
$\tessential^{(k,w)}$ are (Perron-Frobenius) type systems, in the
sense of Definition~\ref{type_system}. Hence, we can use
Theorem~\ref{main_alg} to estimate the spectral radius of the
corresponding random walk. Again, it turns out that it is more
efficient to truncate using weights than length.

In genus $2$, taking the weights $w=(1,2,3,4,1,2,4)$ and the
threshold $k=25$, we obtain a matrix of size $8,999,902$. The
corresponding bound on the spectral radius is $\rho\geq 0.662772$,
proving Theorem~\ref{main_thm}. Those bounds were obtained on a
personal computer with a memory of $12 \text{GB}$ (memory is indeed
the main limiting factor, since one should store all truncated
essential types to create the matrix $M$). With more memory, one
would get better estimates, but it is unlikely that those estimates
converge to the true spectral radius when $k$ tends to infinity: to
recover it, it is probably necessary to distinguish even more points,
for instance by using Cannon's cone types instead of the canonical
types (but this would become totally impractical).

In higher genus, here are the bounds we obtain.
\begin{thm}
In genus $3$, using $\tessential^{(k,w)}$ with
$w=(1,2,3,4,5,6,1,2,6)$ and $k=25$, we get a matrix of size
$7,307,293$ and the estimate $\rho\geq 0.5527735593$.

In genus $4$, using $\tsuffix^{(k,w)}$ with $w=(1,2,3,4,5,6,7,8)$ and
$k=24$, we get a matrix of size $4,120,495$ and the estimate $\rho
\geq 0.48412292068$.
\end{thm}
When the genus increases, the groups look more and more like free
groups. This means that the spectral radius is very close to that of
the random walk on a tree with the corresponding number of generators
(i.e., $\sqrt{4g-1}/(2g)$ in general, specializing to $0.55277079$ in
genus $3$ and $0.48412291827$ in genus $4$), and to get a significant
improvement one needs to take very large matrices. The paths counting
arguments of Bartholdi~\cite{bartholdi_rho}, on the other hand, are
more and more precise when the groups looks more and more like a free
group: in genus $3$, he gets $\rho\geq 0.5527735401$, which is just
slightly worse than our estimate, while requiring considerably less
computer power. In genus $4$, he gets $\rho\geq 0.48412292074$, which
is already better than our estimate, and the situation is certainly
the same in higher genus.

In genus $3$, the upper bound of Nagnibeda~\cite{nagnibeda_rho} is
$\rho\leq 0.552792$, while our lower bound (or Bartholdi's) is much
closer to the naive lower bound coming from the free group. It is
unclear which one is closer to the real value of the spectral radius.

\medskip

For the practical implementation, as in the end of
Paragraph~\ref{subsec_suffix}, it is important to know the
asymptotics of the number of points on $\Sbb^n$ having a given
truncated essential type. We illustrate how to compute such
asymptotics in three significant examples, that can be combined to
handle the general case.
\begin{itemize}
\item Assume first that the type $(i_0,\dotsc, i_m)$ does not
    contain any ambiguous letter, i.e., $i_\ell\not=*$ for all
    $\ell$. In this case, the formula~\eqref{eq:bijection} still
    holds.
\item Assume now that the type is of the form $(i,*,\dotsc, *,
    j)$ for some types $i,j\not=*$, and some number $N$ of stars.
    Let $x$ be a point with the above truncated essential type,
    on a sphere $\Sbb^n$, and let $y$ be the point of modified
    type $j$, on the sphere $\Sbb^{n-N-1}$, such that any
    geodesic from $e$ to $x$ goes through $y$. Since the type
    $2g$ is the only one to have several predecessors,
    necessarily $i=2g$. On the other hand, $j$ can be any type in
    $\{1,1',2,2',3,\dotsc, 2g\}$.

    The discussion in the proof of
    Proposition~\ref{prop:tessential_type} (see in particular the
    last $2$ cases of Figure~\ref{fig:tessentialunique}) implies
    that $N = (2g-1)m$, for some integer $m$: a geodesic from $e$
    to $x$ goes through $y$, and then it follows $m$ $4g$-gons.

    Let us first study the case $m=1$. Consider a point $y$ of
    type $j$, and a $4g$-gon $R$ based at $y$ (i.e., $y$ is the
    closest point to $e$ on this $4g$-gon). There are $4g-2$ such
    $R$ if $j\not= 2g$ and $4g-3$ if $j=2g$ (since a point of
    type $2g$ has two incoming edges). On $R$, consider the point
    $x$ that is the farthest from $e$: it has type $2g$, and one
    can reach it from $y$ by going around $R$ in one direction or
    the other. It follows that
    \begin{multline*}
    \Card\{x\in \Sbb^n \st \tsuffix^{(w,k)}(x)=(2g,*,\dotsc,*,j)\text{ with $N=2g-1$ stars}\}
    \\ = a_j\Card\{y\in \Sbb^{n-N-1} \st t'(y)=j\},
    \end{multline*}
    where $a_j = 4g-3$ if $j\in \{2g-1, 2g\}$, and $a_j=4g-2$
    otherwise. The case of a point $y$ of type $j=2g-1$ is
    special since, on the $4g$-gon $R$ containing the successors
    of $y$ of types $1'$ and $2'$, one of the geodesic paths from
    $y$ to the opposite point $x$ goes through a vertex of type
    $2g$, giving rise to further ambiguities. Hence, this
    $4g$-gon should be discarded from the above counting, leaving
    only $2g-3$ suitable $4g$-gons.

    In the case of a general $m\geq 1$, the number of points $x$
    corresponding to a given point $y$ of type $j$ may be
    obtained first by choosing a suitable $4g$-gon $R_1$ based at
    $y$ (giving $a_j$ choices). Further ambiguities can only be
    obtained by choosing one of the two predecessors (of type
    $2g-1$) of the point that is opposite to $y$ on $R_1$, and
    then following the $4g$-gon $R_2$ based at this point and
    containing its successors of type $1'$ and $2'$. This gives
    $2$ choices for $R_2$, then $2$ more choices for the next
    $4g$-gon $R_3$, and so on. In the end, we obtain
    \begin{multline*}
    \Card\{x\in \Sbb^n \st \tsuffix^{(w,k)}(x)=(2g,*,\dotsc,*,j)\text{ with $N=(2g-1)m$ stars}\}
    \\ = a_j 2^{m-1}\Card\{y\in \Sbb^{n-N-1} \st t'(y)=j\}.
    \end{multline*}
\item Finally, assume that the type is $(i, *,\dotsc, *)$ with
    some number $N$ of stars (the situation is very similar to
    the previous one). Necessarily, as above, $i=2g$. Write
    $N=(2g-1)m+k$ for some $k\in \{1,\dotsc, 2g-1\}$.

    For $m=0$, a point of type $2g$ has two predecessors, so
    there are always ambiguities regarding his first $2g-1$
    ancestors. Hence, the set of points we are considering is
    simply the set of points of type $2g$ on the sphere $\Sbb^n$,
    and there is nothing to do.

    For $m=1$, there can be further ambiguities only if $x$ has
    an ancestor $x'$ of generation $2g-1$ that has type $2g$, and
    $x$ is at the tip of the $4g$-gon $R$ based at one of the two
    predecessors of $x'$ (of type $2g-1$), and containing $x'$.
    There are two choices for $R$.

    Proceeding inductively in the case of a general $m$, we get
    \begin{multline*}
    \Card\{x\in \Sbb^n \st \tsuffix^{(w,k)}(x)=(2g,*,\dotsc,*)\text{ with $N=(2g-1)m+k$ stars}\}
    \\ = 2^{m}\Card\{y\in \Sbb^{n-(2g-1)m} \st t(y)=2g\}.
    \end{multline*}
\end{itemize}
A general truncated essential type is the concatenation of successive
types of the form just described in the examples. We can thus count
the number of points with a given type just by combining the above
formulas.

\stopcompilationbeforecode

\appendix

\section{Implementation}

In this appendix (not intended for publication), we include the code
of the program used to get the numerical estimates of this paper,
written using the \verb+sage+ mathematical software and converted to
\LaTeX{} form thanks to \verb+sws2tex+.

\makeatletter
\def\PY@reset{\let\PY@it=\relax \let\PY@bf=\relax%
    \let\PY@ul=\relax \let\PY@tc=\relax%
    \let\PY@bc=\relax \let\PY@ff=\relax}
\def\PY@tok#1{\csname PY@tok@#1\endcsname}
\def\PY@toks#1+{\ifx\relax#1\empty\else%
    \PY@tok{#1}\expandafter\PY@toks\fi}
\def\PY@do#1{\PY@bc{\PY@tc{\PY@ul{%
    \PY@it{\PY@bf{\PY@ff{#1}}}}}}}
\def\PY#1#2{\PY@reset\PY@toks#1+\relax+\PY@do{#2}}

\def\PY@tok@gd{\def\PY@tc##1{\textcolor[rgb]{0.63,0.00,0.00}{##1}}}
\def\PY@tok@gu{\let\PY@bf=\textbf\def\PY@tc##1{\textcolor[rgb]{0.50,0.00,0.50}{##1}}}
\def\PY@tok@gt{\def\PY@tc##1{\textcolor[rgb]{0.00,0.25,0.82}{##1}}}
\def\PY@tok@gs{\let\PY@bf=\textbf}
\def\PY@tok@gr{\def\PY@tc##1{\textcolor[rgb]{1.00,0.00,0.00}{##1}}}
\def\PY@tok@cm{\let\PY@it=\textit\def\PY@tc##1{\textcolor[rgb]{0.25,0.50,0.56}{##1}}}
\def\PY@tok@vg{\def\PY@tc##1{\textcolor[rgb]{0.73,0.38,0.84}{##1}}}
\def\PY@tok@m{\def\PY@tc##1{\textcolor[rgb]{0.13,0.50,0.31}{##1}}}
\def\PY@tok@mh{\def\PY@tc##1{\textcolor[rgb]{0.13,0.50,0.31}{##1}}}
\def\PY@tok@cs{\def\PY@tc##1{\textcolor[rgb]{0.25,0.50,0.56}{##1}}\def\PY@bc##1{\colorbox[rgb]{1.00,0.94,0.94}{##1}}}
\def\PY@tok@ge{\let\PY@it=\textit}
\def\PY@tok@vc{\def\PY@tc##1{\textcolor[rgb]{0.73,0.38,0.84}{##1}}}
\def\PY@tok@il{\def\PY@tc##1{\textcolor[rgb]{0.13,0.50,0.31}{##1}}}
\def\PY@tok@go{\def\PY@tc##1{\textcolor[rgb]{0.19,0.19,0.19}{##1}}}
\def\PY@tok@cp{\def\PY@tc##1{\textcolor[rgb]{0.00,0.44,0.13}{##1}}}
\def\PY@tok@gi{\def\PY@tc##1{\textcolor[rgb]{0.00,0.63,0.00}{##1}}}
\def\PY@tok@gh{\let\PY@bf=\textbf\def\PY@tc##1{\textcolor[rgb]{0.00,0.00,0.50}{##1}}}
\def\PY@tok@ni{\let\PY@bf=\textbf\def\PY@tc##1{\textcolor[rgb]{0.84,0.33,0.22}{##1}}}
\def\PY@tok@nl{\let\PY@bf=\textbf\def\PY@tc##1{\textcolor[rgb]{0.00,0.13,0.44}{##1}}}
\def\PY@tok@nn{\let\PY@bf=\textbf\def\PY@tc##1{\textcolor[rgb]{0.05,0.52,0.71}{##1}}}
\def\PY@tok@no{\def\PY@tc##1{\textcolor[rgb]{0.38,0.68,0.84}{##1}}}
\def\PY@tok@na{\def\PY@tc##1{\textcolor[rgb]{0.25,0.44,0.63}{##1}}}
\def\PY@tok@nb{\def\PY@tc##1{\textcolor[rgb]{0.00,0.44,0.13}{##1}}}
\def\PY@tok@nc{\let\PY@bf=\textbf\def\PY@tc##1{\textcolor[rgb]{0.05,0.52,0.71}{##1}}}
\def\PY@tok@nd{\let\PY@bf=\textbf\def\PY@tc##1{\textcolor[rgb]{0.33,0.33,0.33}{##1}}}
\def\PY@tok@ne{\def\PY@tc##1{\textcolor[rgb]{0.00,0.44,0.13}{##1}}}
\def\PY@tok@nf{\def\PY@tc##1{\textcolor[rgb]{0.02,0.16,0.49}{##1}}}
\def\PY@tok@si{\let\PY@it=\textit\def\PY@tc##1{\textcolor[rgb]{0.44,0.63,0.82}{##1}}}
\def\PY@tok@s2{\def\PY@tc##1{\textcolor[rgb]{0.25,0.44,0.63}{##1}}}
\def\PY@tok@vi{\def\PY@tc##1{\textcolor[rgb]{0.73,0.38,0.84}{##1}}}
\def\PY@tok@nt{\let\PY@bf=\textbf\def\PY@tc##1{\textcolor[rgb]{0.02,0.16,0.45}{##1}}}
\def\PY@tok@nv{\def\PY@tc##1{\textcolor[rgb]{0.73,0.38,0.84}{##1}}}
\def\PY@tok@s1{\def\PY@tc##1{\textcolor[rgb]{0.25,0.44,0.63}{##1}}}
\def\PY@tok@gp{\let\PY@bf=\textbf\def\PY@tc##1{\textcolor[rgb]{0.78,0.36,0.04}{##1}}}
\def\PY@tok@sh{\def\PY@tc##1{\textcolor[rgb]{0.25,0.44,0.63}{##1}}}
\def\PY@tok@ow{\let\PY@bf=\textbf\def\PY@tc##1{\textcolor[rgb]{0.00,0.44,0.13}{##1}}}
\def\PY@tok@sx{\def\PY@tc##1{\textcolor[rgb]{0.78,0.36,0.04}{##1}}}
\def\PY@tok@bp{\def\PY@tc##1{\textcolor[rgb]{0.00,0.44,0.13}{##1}}}
\def\PY@tok@c1{\let\PY@it=\textit\def\PY@tc##1{\textcolor[rgb]{0.25,0.50,0.56}{##1}}}
\def\PY@tok@kc{\let\PY@bf=\textbf\def\PY@tc##1{\textcolor[rgb]{0.00,0.44,0.13}{##1}}}
\def\PY@tok@c{\let\PY@it=\textit\def\PY@tc##1{\textcolor[rgb]{0.25,0.50,0.56}{##1}}}
\def\PY@tok@mf{\def\PY@tc##1{\textcolor[rgb]{0.13,0.50,0.31}{##1}}}
\def\PY@tok@err{\def\PY@bc##1{\fcolorbox[rgb]{1.00,0.00,0.00}{1,1,1}{##1}}}
\def\PY@tok@kd{\let\PY@bf=\textbf\def\PY@tc##1{\textcolor[rgb]{0.00,0.44,0.13}{##1}}}
\def\PY@tok@ss{\def\PY@tc##1{\textcolor[rgb]{0.32,0.47,0.09}{##1}}}
\def\PY@tok@sr{\def\PY@tc##1{\textcolor[rgb]{0.14,0.33,0.53}{##1}}}
\def\PY@tok@mo{\def\PY@tc##1{\textcolor[rgb]{0.13,0.50,0.31}{##1}}}
\def\PY@tok@mi{\def\PY@tc##1{\textcolor[rgb]{0.13,0.50,0.31}{##1}}}
\def\PY@tok@kn{\let\PY@bf=\textbf\def\PY@tc##1{\textcolor[rgb]{0.00,0.44,0.13}{##1}}}
\def\PY@tok@o{\def\PY@tc##1{\textcolor[rgb]{0.40,0.40,0.40}{##1}}}
\def\PY@tok@kr{\let\PY@bf=\textbf\def\PY@tc##1{\textcolor[rgb]{0.00,0.44,0.13}{##1}}}
\def\PY@tok@s{\def\PY@tc##1{\textcolor[rgb]{0.25,0.44,0.63}{##1}}}
\def\PY@tok@kp{\def\PY@tc##1{\textcolor[rgb]{0.00,0.44,0.13}{##1}}}
\def\PY@tok@w{\def\PY@tc##1{\textcolor[rgb]{0.73,0.73,0.73}{##1}}}
\def\PY@tok@kt{\def\PY@tc##1{\textcolor[rgb]{0.56,0.13,0.00}{##1}}}
\def\PY@tok@sc{\def\PY@tc##1{\textcolor[rgb]{0.25,0.44,0.63}{##1}}}
\def\PY@tok@sb{\def\PY@tc##1{\textcolor[rgb]{0.25,0.44,0.63}{##1}}}
\def\PY@tok@k{\let\PY@bf=\textbf\def\PY@tc##1{\textcolor[rgb]{0.00,0.44,0.13}{##1}}}
\def\PY@tok@se{\let\PY@bf=\textbf\def\PY@tc##1{\textcolor[rgb]{0.25,0.44,0.63}{##1}}}
\def\PY@tok@sd{\let\PY@it=\textit\def\PY@tc##1{\textcolor[rgb]{0.25,0.44,0.63}{##1}}}

\def\PYZbs{\char`\\}
\def\PYZus{\char`\_}
\def\PYZob{\char`\{}
\def\PYZcb{\char`\}}
\def\PYZca{\char`\^}
% for compatibility with earlier versions
\def\PYZat{@}
\def\PYZlb{[}
\def\PYZrb{]}
\makeatother

\attachfile[description=You can get the Sage worksheet by clicking this icon.]{estimateRho.sws}

 Consider a group without odd cycle, orient every edge towards infinity. Assume that every point outside a finite neighborhood of the identity has a type, taking finitely many values, such that the type of a point determines the type of its successors. The transition rules (i.e., the number of successors) are given by a matrix $M$, we also assume that this matrix is transitive and aperiodic (i.e., it is a Perron-Frobenius matrix). More precisely, $M_{ij}$ is the number of successors of type $i$ of any point of type $j$.

 The number of points of type $i$ in the sphere of radius $n$ can be determined as follows: let $\tilde M_{ij} = M_{ij}/p_i$, where $p_i$ is the number of predecessors of any point of type $i$. Then $\tilde M^n$ (plus some information related to the finite ball around $e$ where types are not well defined) dictates the growth of points of type $i$: denoting by $(A_0,\dots, A_{t-1})$ the eigenvector of $\tilde M$ corresponding to the maximal eigenvalue $e^v$ of $\tilde M$, then $Card(\mathbb{S}^n \cap type_i) \sim c A_i e^{nv}$, for some constant $c$ not depending on the type.

 One can then bound from below the spectral radius of the simple random walk, as follows: choose some $\alpha  <  e^{-v/2}$. We define a function $u$ equal to $b_i \alpha^n$ on points of type $i$ at distance $n$ of the identity, for some $b_i$ yet to be chosen. Then $u$ belongs to $\ell^2$. Moreover, if $Q$ is the Markov operator of the random walk and $d$ is the number of neighbors of any point,

  \begin{align*}
  \langle Qu, u\rangle &= \frac{1}{d} \sum_{x\sim y} u(x) u(y) = \frac{2}{d} \sum_{x\to y} u(x)u(y) \sim \frac{2}{d} \sum_{i,j,n} cA_j e^{nv} b_j M_{i,j} b_i \alpha^{2n+1}
  \\&
  = \frac{2}{d} \sum_{i,j,n} cA_j b_j M_{i,j} b_i \alpha/(1-e^{v}\alpha^2).
  \end{align*}

 On the other hand, $\langle u, u \rangle \sim \sum cA_i b_i^2 /(1-\alpha^2 e^v)$. Letting $\alpha$ tend to $e^{-v/2}$, one obtains

 \[ \rho \geq \frac{e^{-v/2}}{d/2} \frac{ \sum A_j b_j M_{i,j} b_i}{\sum A_i b_i^2}.\]

 Note that the unknown constant $c$ has disappeared. To obtain good results, one is reduced to an optimization problem for a quadratic form. Writing $c_i = A_i^{1/2}b_i$, one should maximize

 \[\frac{\sum c_j c_i M_{i,j} A_j^{1/2}/A_i^{1/2}}{\sum d_i^2}.\]

 Writing $M'_{i,j} = M_{i,j} A_j^{1/2}/A_i^{1/2}$, this is the maximum on the unit sphere of the quadratic form defined by $M'$. This is also the maximal eigenvalue of $M''=(M'+(M')^t)/2$. Hence, all this can be readily implemented algorithmically. Let $D$ be the matrix with $A_i$ on the diagonal, then $M' = D^{-1/2} M D^{1/2}$, and $M''= D^{-1/2}M D^{1/2} + D^{1/2}M^t D^{-1/2}$. Conjugating by $D^{-1/2}$, we remark that this matrix is similar to $U = M + DM^tD^{-1}$, hence the value we are seeking is also the maximal eigenvalue of $U$. The interest of this remark is that there is no square root any more, everything can be computed in the algebraic field generated by $e^v$ if we want to do everything algebraically.

 In the implementation, we split everything into small tasks, since we will improve several of them later on. Since we will deal with large sparse matrices, we try to iterate only over the nonzero entries of the matrices, not over all entries.

\begin{Verbatim}[commandchars=\\\{\}]
\PY{k}{def} \PY{n+nf}{apply\PYZus{}map\PYZus{}and\PYZus{}ring}\PY{p}{(}\PY{n}{M}\PY{p}{,} \PY{n}{R}\PY{p}{,} \PY{n}{f}\PY{p}{)}\PY{p}{:}
    \PY{l+s+sd}{"""f is a function of the form f((i,j), a)=b, with b=0 if a=0.}
\PY{l+s+sd}{    R is a ring.}
\PY{l+s+sd}{    Returns a matrix with coefficients in R with }
\PY{l+s+sd}{    coefficient f((i,j), M[i,j]) at position (i,j).}
\PY{l+s+sd}{    """}
    \PY{n}{N} \PY{o}{=} \PY{n}{Matrix}\PY{p}{(}\PY{n}{R}\PY{p}{,} \PY{n}{M}\PY{o}{.}\PY{n}{nrows}\PY{p}{(}\PY{p}{)}\PY{p}{,} \PY{n}{M}\PY{o}{.}\PY{n}{ncols}\PY{p}{(}\PY{p}{)}\PY{p}{,} \PY{l+m+mi}{0}\PY{p}{,} \PY{n}{sparse} \PY{o}{=} \PY{n}{M}\PY{o}{.}\PY{n}{is\PYZus{}sparse}\PY{p}{(}\PY{p}{)}\PY{p}{)}
    \PY{k}{for} \PY{p}{(}\PY{n}{i}\PY{p}{,}\PY{n}{j}\PY{p}{)} \PY{o+ow}{in} \PY{n}{M}\PY{o}{.}\PY{n}{nonzero\PYZus{}positions}\PY{p}{(}\PY{n}{copy}\PY{o}{=}\PY{n}{false}\PY{p}{)}\PY{p}{:}
        \PY{n}{N}\PY{p}{[}\PY{n}{i}\PY{p}{,}\PY{n}{j}\PY{p}{]} \PY{o}{=} \PY{n}{f}\PY{p}{(}\PY{p}{(}\PY{n}{i}\PY{p}{,}\PY{n}{j}\PY{p}{)}\PY{p}{,} \PY{n}{M}\PY{p}{[}\PY{n}{i}\PY{p}{,}\PY{n}{j}\PY{p}{]}\PY{p}{)}
    \PY{k}{return} \PY{n}{N}


\PY{k}{class} \PY{n+nc}{RhoEstimatorBasic}\PY{p}{:}
    \PY{l+s+sd}{"""Given a Perron-Frobenius matrix M such that M[i,j] is the number }
\PY{l+s+sd}{    of transitions from type j to type i, and an integer d corresponding to the}
\PY{l+s+sd}{    number of neighbors of every point in the graph,}
\PY{l+s+sd}{    estimates from below the spectral radius of the simple random}
\PY{l+s+sd}{    walk using the algorithm described in the previous paragraph.}
\PY{l+s+sd}{    estimate() returns this estimate.}
\PY{l+s+sd}{    """}
    \PY{k}{def} \PY{n+nf}{\PYZus{}\PYZus{}init\PYZus{}\PYZus{}}\PY{p}{(}\PY{n+nb+bp}{self}\PY{p}{,} \PY{n}{M}\PY{p}{,} \PY{n}{d}\PY{p}{)}\PY{p}{:}
        \PY{n+nb+bp}{self}\PY{o}{.}\PY{n}{M} \PY{o}{=} \PY{n}{M}
        \PY{n+nb+bp}{self}\PY{o}{.}\PY{n}{d} \PY{o}{=} \PY{n}{d}
        \PY{n+nb+bp}{self}\PY{o}{.}\PY{n}{types} \PY{o}{=} \PY{n}{M}\PY{o}{.}\PY{n}{nrows}\PY{p}{(}\PY{p}{)}

    \PY{k}{def} \PY{n+nf}{Mtilde}\PY{p}{(}\PY{n+nb+bp}{self}\PY{p}{)}\PY{p}{:}
        \PY{n}{predecessors} \PY{o}{=} \PY{p}{[}\PY{n+nb+bp}{self}\PY{o}{.}\PY{n}{d} \PY{o}{-} \PY{n+nb}{sum}\PY{p}{(}\PY{n+nb+bp}{self}\PY{o}{.}\PY{n}{M}\PY{o}{.}\PY{n}{column}\PY{p}{(}\PY{n}{i}\PY{p}{)}\PY{p}{)} \PY{k}{for} \PY{n}{i} \PY{o+ow}{in} \PY{n+nb}{xrange}\PY{p}{(}\PY{n+nb+bp}{self}\PY{o}{.}\PY{n}{types}\PY{p}{)}\PY{p}{]}
        \PY{k}{return} \PY{n}{apply\PYZus{}map\PYZus{}and\PYZus{}ring}\PY{p}{(}\PY{n+nb+bp}{self}\PY{o}{.}\PY{n}{M}\PY{p}{,} \PY{n}{RDF}\PY{p}{,} \PY{k}{lambda} \PY{p}{(}\PY{n}{i}\PY{p}{,}\PY{n}{j}\PY{p}{)}\PY{p}{,}\PY{n}{a}\PY{p}{:} \PY{n}{a}\PY{o}{/}\PY{n}{predecessors}\PY{p}{[}\PY{n}{i}\PY{p}{]}\PY{p}{)}

    \PY{k}{def} \PY{n+nf}{growth\PYZus{}and\PYZus{}A}\PY{p}{(}\PY{n+nb+bp}{self}\PY{p}{)}\PY{p}{:}
        \PY{l+s+sd}{"""returns the growth e\PYZca{}v and the eigenvector A of the graph"""}
        \PY{n}{Mtilde\PYZus{}eigenvectors} \PY{o}{=} \PY{n+nb+bp}{self}\PY{o}{.}\PY{n}{Mtilde}\PY{p}{(}\PY{p}{)}\PY{o}{.}\PY{n}{dense\PYZus{}matrix}\PY{p}{(}\PY{p}{)}\PY{o}{.}\PY{n}{right\PYZus{}eigenvectors}\PY{p}{(}\PY{p}{)}
        \PY{n}{Mtilde\PYZus{}eigenvalues}  \PY{o}{=} \PY{p}{[}\PY{n}{Mtilde\PYZus{}eigenvectors}\PY{p}{[}\PY{n}{j}\PY{p}{]}\PY{p}{[}\PY{l+m+mi}{0}\PY{p}{]}\PY{o}{.}\PY{n}{real}\PY{p}{(}\PY{p}{)}
                                 \PY{k}{for} \PY{n}{j} \PY{o+ow}{in} \PY{n+nb}{xrange}\PY{p}{(}\PY{n+nb+bp}{self}\PY{o}{.}\PY{n}{types}\PY{p}{)}\PY{p}{]}
        \PY{n}{growth}         \PY{o}{=} \PY{n+nb}{max}\PY{p}{(}\PY{n}{Mtilde\PYZus{}eigenvalues}\PY{p}{)}
        \PY{n}{A}              \PY{o}{=} \PY{n}{Mtilde\PYZus{}eigenvectors}\PY{p}{[}\PY{n}{Mtilde\PYZus{}eigenvalues}\PY{o}{.}\PY{n}{index}\PY{p}{(}\PY{n}{growth}\PY{p}{)}\PY{p}{]}\PY{p}{[}\PY{l+m+mi}{1}\PY{p}{]}\PY{p}{[}\PY{l+m+mi}{0}\PY{p}{]}
        \PY{k}{return} \PY{n}{growth}\PY{p}{,} \PY{n}{A}

    \PY{k}{def} \PY{n+nf}{eig\PYZus{}Msym}\PY{p}{(}\PY{n+nb+bp}{self}\PY{p}{)}\PY{p}{:}
        \PY{l+s+sd}{"""returns the maximal expansion of M' on the unit sphere"""}
        \PY{n+nb+bp}{self}\PY{o}{.}\PY{n}{growth}\PY{p}{,} \PY{n+nb+bp}{self}\PY{o}{.}\PY{n}{A} \PY{o}{=} \PY{n+nb+bp}{self}\PY{o}{.}\PY{n}{growth\PYZus{}and\PYZus{}A}\PY{p}{(}\PY{p}{)}
        \PY{n}{Mprime} \PY{o}{=} \PY{n}{apply\PYZus{}map\PYZus{}and\PYZus{}ring}\PY{p}{(}\PY{n+nb+bp}{self}\PY{o}{.}\PY{n}{M}\PY{p}{,} \PY{n}{RDF}\PY{p}{,}
                                    \PY{k}{lambda} \PY{p}{(}\PY{n}{i}\PY{p}{,}\PY{n}{j}\PY{p}{)}\PY{p}{,}\PY{n}{a}\PY{p}{:} \PY{n}{a}\PY{o}{*}\PY{n+nb}{abs}\PY{p}{(}\PY{n+nb+bp}{self}\PY{o}{.}\PY{n}{A}\PY{p}{[}\PY{n}{j}\PY{p}{]}\PY{o}{/}\PY{n+nb+bp}{self}\PY{o}{.}\PY{n}{A}\PY{p}{[}\PY{n}{i}\PY{p}{]}\PY{p}{)}\PY{o}{\PYZca{}}\PY{p}{(}\PY{l+m+mi}{1}\PY{o}{/}\PY{l+m+mi}{2}\PY{p}{)}\PY{p}{)}
        \PY{k}{return} \PY{n+nb+bp}{self}\PY{o}{.}\PY{n}{max\PYZus{}expansion}\PY{p}{(}\PY{n}{Mprime}\PY{p}{)}

    \PY{k}{def} \PY{n+nf}{max\PYZus{}expansion}\PY{p}{(}\PY{n+nb+bp}{self}\PY{p}{,} \PY{n}{Mprime}\PY{p}{)}\PY{p}{:}
        \PY{l+s+sd}{"""returns the maximum of (x, Mprime x) for x in the unit sphere"""}
        \PY{n}{Msym}   \PY{o}{=} \PY{p}{(}\PY{n}{Mprime} \PY{o}{+} \PY{n}{Mprime}\PY{o}{.}\PY{n}{transpose}\PY{p}{(}\PY{p}{)}\PY{p}{)}\PY{o}{/}\PY{l+m+mi}{2}
        \PY{k}{return} \PY{n+nb}{max}\PY{p}{(}\PY{n}{Msym}\PY{o}{.}\PY{n}{dense\PYZus{}matrix}\PY{p}{(}\PY{p}{)}\PY{o}{.}\PY{n}{eigenvalues}\PY{p}{(}\PY{p}{)}\PY{p}{)}

    \PY{k}{def} \PY{n+nf}{estimate}\PY{p}{(}\PY{n+nb+bp}{self}\PY{p}{)}\PY{p}{:}
        \PY{n}{eig} \PY{o}{=} \PY{n+nb+bp}{self}\PY{o}{.}\PY{n}{eig\PYZus{}Msym}\PY{p}{(}\PY{p}{)}
        \PY{k}{return} \PY{l+m+mi}{2}\PY{o}{*}\PY{n+nb+bp}{self}\PY{o}{.}\PY{n}{growth}\PY{o}{\PYZca{}}\PY{p}{(}\PY{o}{-}\PY{l+m+mi}{1}\PY{o}{/}\PY{l+m+mi}{2}\PY{p}{)}\PY{o}{*}\PY{n}{eig}\PY{o}{/}\PY{n+nb+bp}{self}\PY{o}{.}\PY{n}{d}
\end{Verbatim}

 In surface groups, types have been described by Cannon. The next function constructs the transition matrix for types, using his results.

\begin{Verbatim}[commandchars=\\\{\}]
\PY{k}{def} \PY{n+nf}{Msurface}\PY{p}{(}\PY{n}{g}\PY{p}{)}\PY{p}{:}
    \PY{l+s+sd}{"""Returns a (2g times 2g) matrix M.}
\PY{l+s+sd}{    M[i,j] is the number of successors of type i of a point of type j, }
\PY{l+s+sd}{    in the surface group Sigma\PYZus{}g.}
\PY{l+s+sd}{    We only consider types from 1 to 2g, since type 0, being not recurrent,}
\PY{l+s+sd}{    is not relevant for our purposes.}
\PY{l+s+sd}{    """}
    \PY{n}{M} \PY{o}{=} \PY{n}{Matrix}\PY{p}{(}\PY{l+m+mi}{2}\PY{o}{*}\PY{n}{g}\PY{p}{,} \PY{l+m+mi}{2}\PY{o}{*}\PY{n}{g}\PY{p}{,} \PY{l+m+mi}{0}\PY{p}{)}
    \PY{k}{for} \PY{n}{j} \PY{o+ow}{in} \PY{n+nb}{xrange}\PY{p}{(}\PY{l+m+mi}{0}\PY{p}{,} \PY{l+m+mi}{2}\PY{o}{*}\PY{n}{g}\PY{o}{-}\PY{l+m+mi}{1}\PY{p}{)}\PY{p}{:}
        \PY{n}{M}\PY{p}{[}\PY{l+m+mi}{0}  \PY{p}{,}\PY{n}{j}\PY{p}{]} \PY{o}{=} \PY{l+m+mi}{4}\PY{o}{*}\PY{n}{g}\PY{o}{-}\PY{l+m+mi}{3}
        \PY{n}{M}\PY{p}{[}\PY{l+m+mi}{1}  \PY{p}{,}\PY{n}{j}\PY{p}{]} \PY{o}{=} \PY{l+m+mi}{1}
        \PY{n}{M}\PY{p}{[}\PY{n}{j}\PY{o}{+}\PY{l+m+mi}{1}\PY{p}{,}\PY{n}{j}\PY{p}{]} \PY{o}{+}\PY{o}{=} \PY{l+m+mi}{1}
    \PY{n}{M}\PY{p}{[}\PY{l+m+mi}{0}\PY{p}{,}\PY{l+m+mi}{2}\PY{o}{*}\PY{n}{g}\PY{o}{-}\PY{l+m+mi}{1}\PY{p}{]} \PY{o}{=} \PY{l+m+mi}{4}\PY{o}{*}\PY{n}{g}\PY{o}{-}\PY{l+m+mi}{4}
    \PY{n}{M}\PY{p}{[}\PY{l+m+mi}{1}\PY{p}{,}\PY{l+m+mi}{2}\PY{o}{*}\PY{n}{g}\PY{o}{-}\PY{l+m+mi}{1}\PY{p}{]} \PY{o}{=} \PY{l+m+mi}{2}
    \PY{k}{return} \PY{n}{M}
\end{Verbatim}

\begin{Verbatim}[commandchars=\\\{\}]
\PY{n}{Msurface}\PY{p}{(}\PY{l+m+mi}{2}\PY{p}{)}
\end{Verbatim}

{\color{blue}
$\newcommand{\Bold}[1]{\mathbf{#1}}\left(\begin{array}{rrrr}
5 & 5 & 5 & 4 \\
2 & 1 & 1 & 2 \\
0 & 1 & 0 & 0 \\
0 & 0 & 1 & 0
\end{array}\right)$
}

\begin{Verbatim}[commandchars=\\\{\}]
\PY{n}{RhoEstimatorBasic}\PY{p}{(}\PY{n}{Msurface}\PY{p}{(}\PY{l+m+mi}{2}\PY{p}{)}\PY{p}{,} \PY{l+m+mi}{8}\PY{p}{)}\PY{o}{.}\PY{n}{estimate}\PY{p}{(}\PY{p}{)}
\end{Verbatim}

{\color{blue}
$\newcommand{\Bold}[1]{\mathbf{#1}}0.662477976598$
}

\begin{Verbatim}[commandchars=\\\{\}]
\PY{n}{RhoEstimatorBasic}\PY{p}{(}\PY{n}{Msurface}\PY{p}{(}\PY{l+m+mi}{3}\PY{p}{)}\PY{p}{,} \PY{l+m+mi}{12}\PY{p}{)}\PY{o}{.}\PY{n}{estimate}\PY{p}{(}\PY{p}{)}
\end{Verbatim}

{\color{blue}
$\newcommand{\Bold}[1]{\mathbf{#1}}0.552772892866$
}

\begin{Verbatim}[commandchars=\\\{\}]
\PY{n}{RhoEstimatorBasic}\PY{p}{(}\PY{n}{Msurface}\PY{p}{(}\PY{l+m+mi}{4}\PY{p}{)}\PY{p}{,} \PY{l+m+mi}{16}\PY{p}{)}\PY{o}{.}\PY{n}{estimate}\PY{p}{(}\PY{p}{)}
\end{Verbatim}

{\color{blue}
$\newcommand{\Bold}[1]{\mathbf{#1}}0.484122920106$
}

 These results are better than Bartholdi's in genus $2$, worse in genus $3$.

 To go further, one can use extended types: some points have a unique predecessor, i.e., the last step of geodesics ending at a point may be uniquely determined. In this case, the point also determines the Cannon type of its predecessor (note however that the Cannon type of the point does not determine the Cannon type of its predecessor). This may be used to separate the points into further categories, refining Cannon's types (while retaining the property that an extended type determines the extended types of the successors). What we have just described is extended types of length $2$. More generally, one can define sets of extended types as follows, inductively. Consider a set of words $W$ of the form $w_n=(w_n[0],\dots, w_n[\ell_n-1])$ with $M_{w[i]w[i+1]} > 0$ and each letter $w[i]$, $i < \ell_n-1$, has a unique predecessor. Choose a word $w\in W$ whose last letter has a unique predecessor, and let $E(w)$ be the words extending it by one letter. If $W$ defines an extended type, then $W'= (W\backslash \{w\}) \cup E(w)$ also does.

 Equivalently, the sets $W$ defining extended types are the sets of finite words as above such that, for any $w\in W$ and any strict prefix $p$ of $w$, for any letter $a$ that can follow the last letter of $p$, the word $pa$ is a prefix of a word in $W$. One also requires that any letter is a prefix of a word in $W$.

 Given such a set, consider a point $x\in \Gamma$, with some type $w[0]$. If $w[0]$ is in $W$, then $t(x) = (w[0])$. Otherwise, $x$ has a unique predecessor, of some type $w[1]$, and $w[0]w[1]$ is again a prefix of a word in $W$. One can go on until one reaches a word that is in $W$, and this is the $W$-extended type of $x$.

 Here is a last equivalent point of view on sets of extended types: consider the set of all admissible words, and let $\tau$ be a bounded stopping time, i.e., a function such that if $\tau(w) = k$ then for all words $w'$ coinciding with $w$ up to position $k$, $\tau(w') = k$, and moreover if $w$ is a prefix of $w'$ then $\tau(w) \geq \tau(w')$. Then the set $\{w : \tau(w) = |w| \}$ defines an extended type.

 For instance, the set of admissible words of length at most $k$, ending with a letter with several predecessors if length $ < k$, defines an extended type.

 For another example, associate to every letter a weight. The weight of a word is the sum of the weights of its letters. We define $W$ by saying that a word $w$ belongs to $W$ if $weight(w)  >  K$, for some threshold $K$, while this inequality is false for all strict prefixes of $w$ (or, as usual, the last letter of $w$ has several predecessors). Then $W$ defines again an extended type.

 One can also mix those two examples, by stopping when $length(w)= C$ or $weight(w)  >  K$ (or, of course, when a letter has several predecessors).

 Experimentally, it turns out that using the weight $[1,2,3,4]$ gives extended types yielding good bounds on the spectral radius (compared to the size of the matrix it involves).

 ~

 A useful remark is that it is easy to compute the growth and type asymptotics for an extended type. Indeed, let $M_0$ and $A_0$ be the matrix and the growth rates for the  original types (very cheap to compute): for some constant $c$, the  number of points on the sphere $S^n$ with type $i$ is asymptotically $c  A_0(i) growth^n$. Let $w=(w_0\dots w_{k-1})$ be an admissible word.  Points on the sphere $S^n$ with this type are parameterized by a point  on $S^{n-k+1}$ with type $w_{k-1}$ and then a sequence of successors on  the spheres $S^{n-i}$ with type $w_i$. The number of such points is  therefore asymptotically

 \[ c A_0(w_{k-1}) growth^{n-k+1} M_0(w_0, w_1) M_0(w_1, w_2)\dots M_0(w_{k-2}, w_{k-1}).\]

 It follows that the vector $A$ of growth rates is given by

 \[ A(w) = A_0(w_{k-1}) growth^{-k+1} M_0(w_0, w_1) M_0(w_1, w_2)\dots M_0(w_{k-2}, w_{k-1}). \]

 It is very easy to compute if one knows not only the matrix $M$, but  also the words it is coming from. Hence, this computation can be done by the extended type builder.

\begin{Verbatim}[commandchars=\\\{\}]
\PY{k}{class} \PY{n+nc}{ExtendedTypesBuilderWeightLength}\PY{p}{:}
    \PY{l+s+sd}{"""this class constructs the extended type corresponding}
\PY{l+s+sd}{    to its initialization parameters weights, max\PYZus{}weight, length:}
\PY{l+s+sd}{    it constructs all admissible words, stopping when the length }
\PY{l+s+sd}{    of the word becomes equal to length, or when the weight }
\PY{l+s+sd}{    becomes larger than max\PYZus{}weight.}
\PY{l+s+sd}{    }
\PY{l+s+sd}{    matrix() returns the transition matrix for the extended type}
\PY{l+s+sd}{    extended\PYZus{}types() returns all the words in the extended type}
\PY{l+s+sd}{    growth\PYZus{}and\PYZus{}A() returns the growth of spheres in the group,}
\PY{l+s+sd}{    and asymptotics for each type}
\PY{l+s+sd}{    """}
    \PY{k}{def} \PY{n+nf}{\PYZus{}\PYZus{}init\PYZus{}\PYZus{}}\PY{p}{(}\PY{n+nb+bp}{self}\PY{p}{,} \PY{n}{M}\PY{p}{,} \PY{n}{d}\PY{p}{,} \PY{n}{weights}\PY{p}{,} \PY{n}{max\PYZus{}weight}\PY{p}{,} \PY{n}{length}\PY{p}{)}\PY{p}{:}
        \PY{n+nb+bp}{self}\PY{o}{.}\PY{n}{M}            \PY{o}{=} \PY{n}{M}
        \PY{n+nb+bp}{self}\PY{o}{.}\PY{n}{d}            \PY{o}{=} \PY{n}{d}
        \PY{n+nb+bp}{self}\PY{o}{.}\PY{n}{types}        \PY{o}{=} \PY{n+nb+bp}{self}\PY{o}{.}\PY{n}{M}\PY{o}{.}\PY{n}{nrows}\PY{p}{(}\PY{p}{)}
        \PY{n+nb+bp}{self}\PY{o}{.}\PY{n}{predecessors} \PY{o}{=} \PY{p}{[}\PY{n+nb+bp}{self}\PY{o}{.}\PY{n}{d} \PY{o}{-} \PY{n+nb}{sum}\PY{p}{(}\PY{n+nb+bp}{self}\PY{o}{.}\PY{n}{M}\PY{o}{.}\PY{n}{column}\PY{p}{(}\PY{n}{i}\PY{p}{)}\PY{p}{)} \PY{k}{for} \PY{n}{i} \PY{o+ow}{in} \PY{n+nb}{xrange}\PY{p}{(}\PY{n+nb+bp}{self}\PY{o}{.}\PY{n}{types}\PY{p}{)}\PY{p}{]}
        \PY{n+nb+bp}{self}\PY{o}{.}\PY{n}{words}        \PY{o}{=} \PY{n+nb+bp}{None}
        \PY{n+nb+bp}{self}\PY{o}{.}\PY{n}{weights}      \PY{o}{=} \PY{n}{weights}
        \PY{n+nb+bp}{self}\PY{o}{.}\PY{n}{length}       \PY{o}{=} \PY{n}{length}
        \PY{n+nb+bp}{self}\PY{o}{.}\PY{n}{max\PYZus{}weight}   \PY{o}{=} \PY{n}{max\PYZus{}weight}
        \PY{c}{# compute an upper bound for the length of admissible words}
        \PY{k}{if} \PY{n+nb}{min}\PY{p}{(}\PY{n}{weights}\PY{p}{)} \PY{o}{>} \PY{l+m+mi}{0}\PY{p}{:}
            \PY{n}{a} \PY{o}{=} \PY{n}{ceil}\PY{p}{(}\PY{n}{max\PYZus{}weight}\PY{o}{/}\PY{n+nb}{min}\PY{p}{(}\PY{n}{weights}\PY{p}{)}\PY{p}{)} \PY{o}{+} \PY{l+m+mi}{1}
        \PY{k}{else}\PY{p}{:}
            \PY{n}{a} \PY{o}{=} \PY{n}{length}
        \PY{n+nb+bp}{self}\PY{o}{.}\PY{n}{max\PYZus{}length} \PY{o}{=} \PY{n+nb}{min}\PY{p}{(}\PY{n}{length}\PY{p}{,} \PY{n}{a}\PY{p}{)} \PY{o}{+} \PY{l+m+mi}{5}

    \PY{k}{def} \PY{n+nf}{truncate}\PY{p}{(}\PY{n+nb+bp}{self}\PY{p}{,} \PY{n}{t}\PY{p}{)}\PY{p}{:}
        \PY{l+s+sd}{"""truncates a word t to its admissible part.}
\PY{l+s+sd}{        Only works if all the transitions in t are admissible, i.e.,}
\PY{l+s+sd}{        M[t[i], t[i+1]] > 0 for all i.}
\PY{l+s+sd}{        """}
        \PY{n}{total} \PY{o}{=} \PY{l+m+mi}{0}
        \PY{n}{i}     \PY{o}{=} \PY{l+m+mi}{0}
        \PY{k}{for} \PY{n}{i} \PY{o+ow}{in} \PY{n+nb}{xrange}\PY{p}{(}\PY{n+nb}{min}\PY{p}{(}\PY{n+nb}{len}\PY{p}{(}\PY{n}{t}\PY{p}{)}\PY{p}{,} \PY{n+nb+bp}{self}\PY{o}{.}\PY{n}{length}\PY{p}{)}\PY{p}{)}\PY{p}{:}
            \PY{k}{if} \PY{n+nb+bp}{self}\PY{o}{.}\PY{n}{predecessors}\PY{p}{[}\PY{n}{t}\PY{p}{[}\PY{n}{i}\PY{p}{]}\PY{p}{]} \PY{o}{>} \PY{l+m+mi}{1}\PY{p}{:}
                \PY{k}{break}
            \PY{n}{total} \PY{o}{+}\PY{o}{=} \PY{n+nb+bp}{self}\PY{o}{.}\PY{n}{weights}\PY{p}{[}\PY{n}{t}\PY{p}{[}\PY{n}{i}\PY{p}{]}\PY{p}{]}
            \PY{k}{if} \PY{n}{total} \PY{o}{>} \PY{n+nb+bp}{self}\PY{o}{.}\PY{n}{max\PYZus{}weight}\PY{p}{:}
                \PY{k}{break}
        \PY{k}{return} \PY{n}{t}\PY{p}{[}\PY{p}{:}\PY{n}{i}\PY{o}{+}\PY{l+m+mi}{1}\PY{p}{]}

    \PY{k}{def} \PY{n+nf}{extended\PYZus{}types}\PY{p}{(}\PY{n+nb+bp}{self}\PY{p}{)}\PY{p}{:}
        \PY{l+s+sd}{"""returns a list of all admissible types"""}
        \PY{k}{if} \PY{n+nb+bp}{self}\PY{o}{.}\PY{n}{words} \PY{o+ow}{is} \PY{n+nb+bp}{None}\PY{p}{:}
            \PY{n+nb+bp}{self}\PY{o}{.}\PY{n}{compute\PYZus{}extended\PYZus{}types}\PY{p}{(}\PY{p}{)}
        \PY{k}{return} \PY{n+nb+bp}{self}\PY{o}{.}\PY{n}{words}

    \PY{k}{def} \PY{n+nf}{compute\PYZus{}extended\PYZus{}types}\PY{p}{(}\PY{n+nb+bp}{self}\PY{p}{)}\PY{p}{:}
        \PY{l+s+sd}{"""Puts in self.words a list of all admissible types.}
\PY{l+s+sd}{        The construction starts from a collection of types, }
\PY{l+s+sd}{        and tries to extend them.}
\PY{l+s+sd}{        If they are not extendable, put them in result.}
\PY{l+s+sd}{        Otherwise, go on with the new types.}
\PY{l+s+sd}{        """}
        \PY{n}{result} \PY{o}{=} \PY{p}{[}\PY{p}{]}
        \PY{n}{current\PYZus{}pool} \PY{o}{=} \PY{p}{[}\PY{p}{(}\PY{p}{(}\PY{n}{i}\PY{p}{,}\PY{p}{)}\PY{p}{,} \PY{n+nb+bp}{self}\PY{o}{.}\PY{n}{weights}\PY{p}{[}\PY{n}{i}\PY{p}{]}\PY{p}{)} \PY{k}{for} \PY{n}{i} \PY{o+ow}{in} \PY{n+nb}{xrange}\PY{p}{(}\PY{n+nb+bp}{self}\PY{o}{.}\PY{n}{types}\PY{p}{)}\PY{p}{]}
        \PY{n}{length}       \PY{o}{=} \PY{l+m+mi}{1}
        \PY{k}{while} \PY{n}{length} \PY{o}{<} \PY{n+nb+bp}{self}\PY{o}{.}\PY{n}{length}\PY{p}{:}
            \PY{n}{length}  \PY{o}{+}\PY{o}{=} \PY{l+m+mi}{1}
            \PY{n}{temp\PYZus{}pool} \PY{o}{=} \PY{p}{[}\PY{p}{(}\PY{n}{t}\PY{o}{+}\PY{p}{(}\PY{n}{i}\PY{p}{,}\PY{p}{)}\PY{p}{,} \PY{n}{w} \PY{o}{+} \PY{n+nb+bp}{self}\PY{o}{.}\PY{n}{weights}\PY{p}{[}\PY{n}{i}\PY{p}{]}\PY{p}{)}
                          \PY{k}{for} \PY{p}{(}\PY{n}{t}\PY{p}{,}\PY{n}{w}\PY{p}{)} \PY{o+ow}{in} \PY{n}{current\PYZus{}pool}
                          \PY{k}{if} \PY{n}{w} \PY{o}{<}\PY{o}{=} \PY{n+nb+bp}{self}\PY{o}{.}\PY{n}{max\PYZus{}weight} \PY{o+ow}{and} \PY{n+nb+bp}{self}\PY{o}{.}\PY{n}{predecessors}\PY{p}{[}\PY{n}{t}\PY{p}{[}\PY{o}{-}\PY{l+m+mi}{1}\PY{p}{]}\PY{p}{]} \PY{o}{<}\PY{o}{=} \PY{l+m+mi}{1}
                          \PY{k}{for} \PY{n}{i} \PY{o+ow}{in} \PY{n+nb}{xrange}\PY{p}{(}\PY{n+nb+bp}{self}\PY{o}{.}\PY{n}{types}\PY{p}{)}
                          \PY{k}{if} \PY{n+nb+bp}{self}\PY{o}{.}\PY{n}{M}\PY{p}{[}\PY{n}{t}\PY{p}{[}\PY{o}{-}\PY{l+m+mi}{1}\PY{p}{]}\PY{p}{,} \PY{n}{i}\PY{p}{]} \PY{o}{>} \PY{l+m+mi}{0}\PY{p}{]}
            \PY{n}{result} \PY{o}{+}\PY{o}{=} \PY{p}{[}\PY{n}{t} \PY{k}{for} \PY{p}{(}\PY{n}{t}\PY{p}{,}\PY{n}{w}\PY{p}{)} \PY{o+ow}{in} \PY{n}{current\PYZus{}pool}
                         \PY{k}{if} \PY{n}{w} \PY{o}{>} \PY{n+nb+bp}{self}\PY{o}{.}\PY{n}{max\PYZus{}weight} \PY{o+ow}{or} \PY{n+nb+bp}{self}\PY{o}{.}\PY{n}{predecessors}\PY{p}{[}\PY{n}{t}\PY{p}{[}\PY{o}{-}\PY{l+m+mi}{1}\PY{p}{]}\PY{p}{]} \PY{o}{>} \PY{l+m+mi}{1}\PY{p}{]}
            \PY{n}{current\PYZus{}pool} \PY{o}{=} \PY{n}{temp\PYZus{}pool}
        \PY{n}{result}    \PY{o}{+}\PY{o}{=} \PY{p}{[}\PY{n}{t} \PY{k}{for} \PY{p}{(}\PY{n}{t}\PY{p}{,}\PY{n}{w}\PY{p}{)} \PY{o+ow}{in} \PY{n}{current\PYZus{}pool}\PY{p}{]}
        \PY{n+nb+bp}{self}\PY{o}{.}\PY{n}{words} \PY{o}{=} \PY{n}{result}

    \PY{k}{def} \PY{n+nf}{growth\PYZus{}and\PYZus{}Azero}\PY{p}{(}\PY{n+nb+bp}{self}\PY{p}{)}\PY{p}{:}
        \PY{l+s+sd}{"""returns the growth and the eigenvalues corresponding to the }
\PY{l+s+sd}{        initial matrix M"""}
        \PY{n}{r} \PY{o}{=} \PY{n}{RhoEstimatorBasic}\PY{p}{(}\PY{n+nb+bp}{self}\PY{o}{.}\PY{n}{M}\PY{p}{,} \PY{n+nb+bp}{self}\PY{o}{.}\PY{n}{d}\PY{p}{)}
        \PY{k}{return} \PY{n}{r}\PY{o}{.}\PY{n}{growth\PYZus{}and\PYZus{}A}\PY{p}{(}\PY{p}{)}

    \PY{k}{def} \PY{n+nf}{growth\PYZus{}and\PYZus{}A}\PY{p}{(}\PY{n+nb+bp}{self}\PY{p}{)}\PY{p}{:}
        \PY{l+s+sd}{"""returns the growth and the asymptotics of extended types"""}
        \PY{n}{growth}\PY{p}{,} \PY{n}{Azero} \PY{o}{=} \PY{n+nb+bp}{self}\PY{o}{.}\PY{n}{growth\PYZus{}and\PYZus{}Azero}\PY{p}{(}\PY{p}{)}
        \PY{n}{types}         \PY{o}{=} \PY{n+nb+bp}{self}\PY{o}{.}\PY{n}{M}\PY{o}{.}\PY{n}{nrows}\PY{p}{(}\PY{p}{)}

        \PY{c}{# precompute all the possible values of A0(w\PYZus{}\PYZob{}k-1\PYZcb{})growth\PYZca{}\PYZob{}-k+1\PYZcb{}}
        \PY{n}{words}      \PY{o}{=} \PY{n+nb+bp}{self}\PY{o}{.}\PY{n}{extended\PYZus{}types}\PY{p}{(}\PY{p}{)}
        \PY{n}{precompute} \PY{o}{=} \PY{p}{[}\PY{p}{[}\PY{l+m+mi}{0} \PY{k}{for} \PY{n}{i} \PY{o+ow}{in} \PY{n+nb}{xrange}\PY{p}{(}\PY{n+nb+bp}{self}\PY{o}{.}\PY{n}{max\PYZus{}length}\PY{p}{)}\PY{p}{]} \PY{k}{for} \PY{n}{j} \PY{o+ow}{in} \PY{n+nb}{xrange}\PY{p}{(}\PY{n}{types}\PY{p}{)}\PY{p}{]}
        \PY{k}{for} \PY{n}{j} \PY{o+ow}{in} \PY{n+nb}{xrange}\PY{p}{(}\PY{n}{types}\PY{p}{)}\PY{p}{:}
            \PY{n}{precompute}\PY{p}{[}\PY{n}{j}\PY{p}{]}\PY{p}{[}\PY{l+m+mi}{0}\PY{p}{]} \PY{o}{=} \PY{n}{Azero}\PY{p}{[}\PY{n}{j}\PY{p}{]}
            \PY{k}{for} \PY{n}{i} \PY{o+ow}{in} \PY{n+nb}{xrange}\PY{p}{(}\PY{l+m+mi}{1}\PY{p}{,} \PY{n+nb+bp}{self}\PY{o}{.}\PY{n}{max\PYZus{}length}\PY{p}{)}\PY{p}{:}
                \PY{n}{precompute}\PY{p}{[}\PY{n}{j}\PY{p}{]}\PY{p}{[}\PY{n}{i}\PY{p}{]} \PY{o}{=} \PY{n}{precompute}\PY{p}{[}\PY{n}{j}\PY{p}{]}\PY{p}{[}\PY{n}{i}\PY{o}{-}\PY{l+m+mi}{1}\PY{p}{]}\PY{o}{/}\PY{n}{growth}

        \PY{c}{# compute all the values of A}
        \PY{n}{A} \PY{o}{=} \PY{p}{[}\PY{l+m+mi}{0} \PY{k}{for} \PY{n}{i} \PY{o+ow}{in} \PY{n+nb}{xrange}\PY{p}{(}\PY{n+nb}{len}\PY{p}{(}\PY{n}{words}\PY{p}{)}\PY{p}{)}\PY{p}{]}
        \PY{k}{for} \PY{p}{(}\PY{n}{i}\PY{p}{,}\PY{n}{t}\PY{p}{)} \PY{o+ow}{in} \PY{n+nb}{enumerate}\PY{p}{(}\PY{n}{words}\PY{p}{)}\PY{p}{:}
            \PY{n}{m} \PY{o}{=} \PY{l+m+mi}{1}
            \PY{k}{for} \PY{n}{j} \PY{o+ow}{in} \PY{n+nb}{xrange}\PY{p}{(}\PY{n+nb}{len}\PY{p}{(}\PY{n}{t}\PY{p}{)}\PY{o}{-}\PY{l+m+mi}{1}\PY{p}{)}\PY{p}{:}
                \PY{n}{m} \PY{o}{*}\PY{o}{=} \PY{n+nb+bp}{self}\PY{o}{.}\PY{n}{M}\PY{p}{[}\PY{n}{t}\PY{p}{[}\PY{n}{j}\PY{p}{]}\PY{p}{,} \PY{n}{t}\PY{p}{[}\PY{n}{j}\PY{o}{+}\PY{l+m+mi}{1}\PY{p}{]}\PY{p}{]}
            \PY{n}{A}\PY{p}{[}\PY{n}{i}\PY{p}{]} \PY{o}{=} \PY{n}{m}\PY{o}{*}\PY{n}{precompute}\PY{p}{[}\PY{n}{t}\PY{p}{[}\PY{o}{-}\PY{l+m+mi}{1}\PY{p}{]}\PY{p}{]}\PY{p}{[}\PY{n+nb}{len}\PY{p}{(}\PY{n}{t}\PY{p}{)}\PY{o}{-}\PY{l+m+mi}{1}\PY{p}{]}
        \PY{k}{return} \PY{n}{growth}\PY{p}{,} \PY{n}{A}

    \PY{k}{def} \PY{n+nf}{successor}\PY{p}{(}\PY{n+nb+bp}{self}\PY{p}{,} \PY{n}{orig}\PY{p}{,} \PY{n}{letter}\PY{p}{)}\PY{p}{:}
        \PY{l+s+sd}{"""returns the admissible type  }
\PY{l+s+sd}{           obtained by adding letter at the beginning of orig"""}
        \PY{k}{return} \PY{n+nb+bp}{self}\PY{o}{.}\PY{n}{truncate}\PY{p}{(}\PY{p}{(}\PY{n}{letter}\PY{p}{,}\PY{p}{)} \PY{o}{+} \PY{n}{orig}\PY{p}{)}

    \PY{k}{def} \PY{n+nf}{matrix}\PY{p}{(}\PY{n+nb+bp}{self}\PY{p}{)}\PY{p}{:}
        \PY{l+s+sd}{"""returns the transition matrix for the extended type"""}
        \PY{n}{new\PYZus{}types}  \PY{o}{=} \PY{n+nb+bp}{self}\PY{o}{.}\PY{n}{extended\PYZus{}types}\PY{p}{(}\PY{p}{)}
        \PY{n}{types\PYZus{}dict} \PY{o}{=} \PY{n+nb}{dict}\PY{p}{(}\PY{n+nb}{zip}\PY{p}{(}\PY{n}{new\PYZus{}types}\PY{p}{,} \PY{n+nb}{range}\PY{p}{(}\PY{n+nb}{len}\PY{p}{(}\PY{n}{new\PYZus{}types}\PY{p}{)}\PY{p}{)}\PY{p}{)}\PY{p}{)}
        \PY{n}{N} \PY{o}{=} \PY{n}{Matrix}\PY{p}{(}\PY{n+nb}{len}\PY{p}{(}\PY{n}{new\PYZus{}types}\PY{p}{)}\PY{p}{,} \PY{n+nb}{len}\PY{p}{(}\PY{n}{new\PYZus{}types}\PY{p}{)}\PY{p}{,} \PY{l+m+mi}{0}\PY{p}{,} \PY{n}{sparse} \PY{o}{=} \PY{n+nb+bp}{True}\PY{p}{)}
        \PY{k}{for} \PY{n}{orig} \PY{o+ow}{in} \PY{n}{new\PYZus{}types}\PY{p}{:}
            \PY{k}{for} \PY{n}{letter} \PY{o+ow}{in} \PY{n+nb}{xrange}\PY{p}{(}\PY{n+nb+bp}{self}\PY{o}{.}\PY{n}{types}\PY{p}{)}\PY{p}{:}
                \PY{k}{if} \PY{n+nb+bp}{self}\PY{o}{.}\PY{n}{M}\PY{p}{[}\PY{n}{letter}\PY{p}{,} \PY{n}{orig}\PY{p}{[}\PY{l+m+mi}{0}\PY{p}{]}\PY{p}{]} \PY{o}{>} \PY{l+m+mi}{0}\PY{p}{:}
                    \PY{n}{s} \PY{o}{=} \PY{n+nb+bp}{self}\PY{o}{.}\PY{n}{successor}\PY{p}{(}\PY{n}{orig}\PY{p}{,} \PY{n}{letter}\PY{p}{)}
                    \PY{n}{N}\PY{p}{[}\PY{n}{types\PYZus{}dict}\PY{p}{[}\PY{n}{s}\PY{p}{]}\PY{p}{,} \PY{n}{types\PYZus{}dict}\PY{p}{[}\PY{n}{orig}\PY{p}{]}\PY{p}{]} \PY{o}{=} \PY{n+nb+bp}{self}\PY{o}{.}\PY{n}{M}\PY{p}{[}\PY{n}{letter}\PY{p}{,} \PY{n}{orig}\PY{p}{[}\PY{l+m+mi}{0}\PY{p}{]}\PY{p}{]}
        \PY{k}{return} \PY{n}{N}


\PY{k}{class} \PY{n+nc}{ExtendedTypesBuilderLength}\PY{p}{(}\PY{n}{ExtendedTypesBuilderWeightLength}\PY{p}{)}\PY{p}{:}
    \PY{l+s+sd}{"""Extended type builder that only selects admissible}
\PY{l+s+sd}{    words through their length"""}
    \PY{k}{def} \PY{n+nf}{\PYZus{}\PYZus{}init\PYZus{}\PYZus{}}\PY{p}{(}\PY{n+nb+bp}{self}\PY{p}{,} \PY{n}{M}\PY{p}{,} \PY{n}{d}\PY{p}{,} \PY{n}{length}\PY{p}{)}\PY{p}{:}
        \PY{n}{ExtendedTypesBuilderWeightLength}\PY{o}{.}\PY{n}{\PYZus{}\PYZus{}init\PYZus{}\PYZus{}}\PY{p}{(}\PY{n+nb+bp}{self}\PY{p}{,}
                 \PY{n}{M}\PY{p}{,} \PY{n}{d}\PY{p}{,} \PY{p}{[}\PY{l+m+mi}{0} \PY{k}{for} \PY{n}{i} \PY{o+ow}{in} \PY{n+nb}{xrange}\PY{p}{(}\PY{n}{M}\PY{o}{.}\PY{n}{nrows}\PY{p}{(}\PY{p}{)}\PY{p}{)}\PY{p}{]}\PY{p}{,} \PY{l+m+mi}{1}\PY{p}{,} \PY{n}{length}\PY{p}{)}


\PY{k}{class} \PY{n+nc}{ExtendedTypesBuilderWeight}\PY{p}{(}\PY{n}{ExtendedTypesBuilderWeightLength}\PY{p}{)}\PY{p}{:}
    \PY{l+s+sd}{"""Extended type builder that only selects admissible}
\PY{l+s+sd}{    words through their weight"""}
    \PY{k}{def} \PY{n+nf}{\PYZus{}\PYZus{}init\PYZus{}\PYZus{}}\PY{p}{(}\PY{n+nb+bp}{self}\PY{p}{,} \PY{n}{M}\PY{p}{,} \PY{n}{d}\PY{p}{,} \PY{n}{weights}\PY{p}{,} \PY{n}{max\PYZus{}weight}\PY{p}{)}\PY{p}{:}
        \PY{n}{ExtendedTypesBuilderWeightLength}\PY{o}{.}\PY{n}{\PYZus{}\PYZus{}init\PYZus{}\PYZus{}}\PY{p}{(}\PY{n+nb+bp}{self}\PY{p}{,}
                 \PY{n}{M}\PY{p}{,} \PY{n}{d}\PY{p}{,} \PY{n}{weights}\PY{p}{,} \PY{n}{max\PYZus{}weight}\PY{p}{,} \PY{l+m+mi}{10000}\PY{p}{)}
\end{Verbatim}

 In surface groups, one can estimate from below the spectral radius by using $k$-truncated extended types instead of types: the above algorithm still applies. Moreover, if $k$ is larger, one distinguishes more categories of points, and one may therefore expect better bounds on $\rho$. However, the size of the transition matrix increases, which means that computations are more and more intensive...

\begin{Verbatim}[commandchars=\\\{\}]
\PY{n}{genus} \PY{o}{=} \PY{l+m+mi}{2}
\PY{k}{for} \PY{n}{max\PYZus{}length} \PY{o+ow}{in} \PY{n+nb}{range}\PY{p}{(}\PY{l+m+mi}{1}\PY{p}{,}\PY{l+m+mi}{8}\PY{p}{,}\PY{l+m+mi}{2}\PY{p}{)}\PY{p}{:}
    \PY{n}{M} \PY{o}{=} \PY{n}{ExtendedTypesBuilderLength}\PY{p}{(}\PY{n}{Msurface}\PY{p}{(}\PY{n}{genus}\PY{p}{)}\PY{p}{,} \PY{l+m+mi}{4}\PY{o}{*}\PY{n}{genus}\PY{p}{,} \PY{n}{max\PYZus{}length}\PY{p}{)}\PY{o}{.}\PY{n}{matrix}\PY{p}{(}\PY{p}{)}
    \PY{k}{print} \PY{l+s}{"}\PY{l+s}{genus = }\PY{l+s}{"}\PY{p}{,} \PY{n}{genus}\PY{p}{,} \PY{l+s}{"}\PY{l+s}{, max\PYZus{}length = }\PY{l+s}{"}\PY{p}{,} \PY{n}{max\PYZus{}length}\PY{p}{,}\PYZbs{}
          \PY{l+s}{"}\PY{l+s}{, matrix size = }\PY{l+s}{"}\PY{p}{,} \PY{n}{M}\PY{o}{.}\PY{n}{nrows}\PY{p}{(}\PY{p}{)}
    \PY{n}{time} \PY{n}{rho} \PY{o}{=} \PY{n}{RhoEstimatorBasic}\PY{p}{(}\PY{n}{M}\PY{p}{,} \PY{l+m+mi}{4}\PY{o}{*}\PY{n}{genus}\PY{p}{)}\PY{o}{.}\PY{n}{estimate}\PY{p}{(}\PY{p}{)}
    \PY{k}{print} \PY{l+s}{"}\PY{l+s}{estimate = }\PY{l+s}{"}\PY{p}{,} \PY{n}{rho}\PY{p}{,} \PY{l+s}{"}\PY{l+s+se}{\PYZbs{}n}\PY{l+s}{"}
\end{Verbatim}
\begin{Verbatim}[formatcom=\color{blue}]
genus =  2 , max_length =  1 , matrix size =  4
Time: CPU 0.00 s, Wall: 0.02 s
estimate =  0.662477976598

genus =  2 , max_length =  3 , matrix size =  25
Time: CPU 0.01 s, Wall: 0.11 s
estimate =  0.6626394462

genus =  2 , max_length =  5 , matrix size =  148
Time: CPU 0.16 s, Wall: 0.18 s
estimate =  0.662694226446

genus =  2 , max_length =  7 , matrix size =  865
Time: CPU 12.27 s, Wall: 7.83 s
estimate =  0.662720574395
\end{Verbatim}

 There are two bottlenecks in the previous computations, when the matrix size increases: the computation of the growth rates $A$, which is done by the RhoEstimatorBasic class, and the computation of the maximal expansion of the matrix $M'$. We will now explain how to optimize those. The first issue is easy to deal with, since we have already observed that $A$ can be cheaply computed by the extended types builder itself.

 Let us now deal with the maximal expansion of $M'$, i.e., the dominating eigenvalue of $Msym = (M' + (M')^*))/2$. Luckily, we do not need to compute all the eigenvalues, contrary to what we did before. Since this is a positive matrix, if one starts with a positive vector and iterates $Msym$, one converges exponentially fast to the eigenvector for the dominating eigenvalue. In particular, $\|Msym^{n+1} v\|/\|Msym^n v\|$ converges to the dominating eigenvalue (and one checks easily that this sequence is non-decreasing, see Woess, Corollary 10.2). This gives a simple algorithm to estimate the dominating eigenvalue from below.

\begin{Verbatim}[commandchars=\\\{\}]
\PY{k}{class} \PY{n+nc}{RhoEstimator}\PY{p}{(}\PY{n}{RhoEstimatorBasic}\PY{p}{)}\PY{p}{:}
    \PY{l+s+sd}{"""In this class, based on RhoEstimator, the computation of the maximal}
\PY{l+s+sd}{    eigenvalue of Msym is made using a simple iterative procedure."""}
    \PY{k}{def} \PY{n+nf}{max\PYZus{}expansion}\PY{p}{(}\PY{n+nb+bp}{self}\PY{p}{,} \PY{n}{M}\PY{p}{,} \PY{n}{precision} \PY{o}{=} \PY{l+m+mi}{10}\PY{o}{\PYZca{}}\PY{p}{(}\PY{o}{-}\PY{l+m+mi}{50}\PY{p}{)}\PY{p}{)}\PY{p}{:}
        \PY{l+s+sd}{"""M is a matrix with nonnegative entries.}
\PY{l+s+sd}{        returns a number that is less than or equal to the}
\PY{l+s+sd}{        maximum of (q, M q) for q in the unit sphere. }
\PY{l+s+sd}{        Computations are done within the given precision.}
\PY{l+s+sd}{        """}
        \PY{c}{# Sparse matrices with RDF coefficients have no specific}
        \PY{c}{# implementation in Sage 5.10, and are very slow.}
        \PY{c}{# We provide a custom cython implementation, at the end of}
        \PY{c}{# this worksheet for clarity.}
        \PY{k}{if} \PY{n}{M}\PY{o}{.}\PY{n}{is\PYZus{}sparse}\PY{p}{(}\PY{p}{)}\PY{p}{:}
            \PY{k}{return} \PY{n}{max\PYZus{}expansion\PYZus{}sparse\PYZus{}double\PYZus{}cython}\PY{p}{(}\PY{n}{M}\PY{p}{,} \PY{l+m+mi}{10}\PY{o}{\PYZca{}}\PY{p}{(}\PY{o}{-}\PY{l+m+mi}{50}\PY{p}{)}\PY{p}{)}

        \PY{n}{u} \PY{o}{=} \PY{n}{vector}\PY{p}{(}\PY{n}{M}\PY{o}{.}\PY{n}{base\PYZus{}ring}\PY{p}{(}\PY{p}{)}\PY{p}{,} \PY{p}{[}\PY{l+m+mi}{1} \PY{k}{for} \PY{n}{i} \PY{o+ow}{in} \PY{n+nb}{xrange}\PY{p}{(}\PY{n}{M}\PY{o}{.}\PY{n}{nrows}\PY{p}{(}\PY{p}{)}\PY{p}{)}\PY{p}{]}\PY{p}{)}
        \PY{n}{u} \PY{o}{=} \PY{n}{u}\PY{o}{/}\PY{n}{u}\PY{o}{.}\PY{n}{norm}\PY{p}{(}\PY{p}{)}
        \PY{n}{expansion} \PY{o}{=} \PY{l+m+mi}{0}
        \PY{n}{Msym} \PY{o}{=} \PY{p}{(}\PY{n}{M} \PY{o}{+} \PY{n}{M}\PY{o}{.}\PY{n}{transpose}\PY{p}{(}\PY{p}{)}\PY{p}{)}\PY{o}{/}\PY{l+m+mi}{2}
        \PY{k}{while} \PY{n+nb+bp}{True}\PY{p}{:}
            \PY{n}{v}         \PY{o}{=} \PY{n}{Msym} \PY{o}{*} \PY{n}{u}
            \PY{n}{vnorm}     \PY{o}{=} \PY{n}{v}\PY{o}{.}\PY{n}{norm}\PY{p}{(}\PY{p}{)}
            \PY{n}{diff}      \PY{o}{=} \PY{n}{vnorm} \PY{o}{-} \PY{n}{expansion}
            \PY{n}{expansion} \PY{o}{=} \PY{n}{vnorm}
            \PY{n}{u}         \PY{o}{=} \PY{n}{v}\PY{o}{/}\PY{n}{vnorm}
            \PY{k}{if} \PY{n}{diff} \PY{o}{<} \PY{n}{precision}\PY{p}{:}
                \PY{k}{return} \PY{n}{expansion}


\PY{k}{class} \PY{n+nc}{RhoEstimatorExt}\PY{p}{(}\PY{n}{RhoEstimator}\PY{p}{)}\PY{p}{:}
    \PY{l+s+sd}{"""this estimator takes as a parameter the extended types builder.}
\PY{l+s+sd}{    Hence, type asymptotics are readily computed"""}
    \PY{k}{def} \PY{n+nf}{\PYZus{}\PYZus{}init\PYZus{}\PYZus{}}\PY{p}{(}\PY{n+nb+bp}{self}\PY{p}{,} \PY{n}{builder}\PY{p}{)}\PY{p}{:}
        \PY{n+nb+bp}{self}\PY{o}{.}\PY{n}{builder} \PY{o}{=} \PY{n}{builder}
        \PY{n}{RhoEstimator}\PY{o}{.}\PY{n}{\PYZus{}\PYZus{}init\PYZus{}\PYZus{}}\PY{p}{(}\PY{n+nb+bp}{self}\PY{p}{,} \PY{n+nb+bp}{self}\PY{o}{.}\PY{n}{builder}\PY{o}{.}\PY{n}{matrix}\PY{p}{(}\PY{p}{)}\PY{p}{,} \PY{n+nb+bp}{self}\PY{o}{.}\PY{n}{builder}\PY{o}{.}\PY{n}{d}\PY{p}{)}

    \PY{k}{def} \PY{n+nf}{growth\PYZus{}and\PYZus{}A}\PY{p}{(}\PY{n+nb+bp}{self}\PY{p}{)}\PY{p}{:}
        \PY{k}{return} \PY{n+nb+bp}{self}\PY{o}{.}\PY{n}{builder}\PY{o}{.}\PY{n}{growth\PYZus{}and\PYZus{}A}\PY{p}{(}\PY{p}{)}


\PY{k}{def} \PY{n+nf}{test\PYZus{}length}\PY{p}{(}\PY{n}{genus}\PY{p}{,} \PY{n}{max\PYZus{}length}\PY{p}{)}\PY{p}{:}
    \PY{n}{e} \PY{o}{=} \PY{n}{ExtendedTypesBuilderLength}\PY{p}{(}\PY{n}{Msurface}\PY{p}{(}\PY{n}{genus}\PY{p}{)}\PY{p}{,} \PY{l+m+mi}{4}\PY{o}{*}\PY{n}{genus}\PY{p}{,} \PY{n}{max\PYZus{}length}\PY{p}{)}
    \PY{n}{r} \PY{o}{=} \PY{n}{RhoEstimatorExt}\PY{p}{(}\PY{n}{e}\PY{p}{)}
    \PY{k}{print} \PY{l+s}{"}\PY{l+s}{genus = }\PY{l+s}{"}\PY{p}{,} \PY{n}{genus}\PY{p}{,} \PY{l+s}{"}\PY{l+s}{, max\PYZus{}length = }\PY{l+s}{"}\PY{p}{,} \PY{n}{max\PYZus{}length}\PY{p}{,}\PYZbs{}
          \PY{l+s}{"}\PY{l+s}{, matrix size = }\PY{l+s}{"}\PY{p}{,} \PY{n}{r}\PY{o}{.}\PY{n}{M}\PY{o}{.}\PY{n}{nrows}\PY{p}{(}\PY{p}{)}
    \PY{k}{print} \PY{l+s}{"}\PY{l+s}{estimate = }\PY{l+s}{"}\PY{p}{,} \PY{n}{r}\PY{o}{.}\PY{n}{estimate}\PY{p}{(}\PY{p}{)}

\PY{k}{def} \PY{n+nf}{test\PYZus{}weight}\PY{p}{(}\PY{n}{genus}\PY{p}{,} \PY{n}{weights}\PY{p}{,} \PY{n}{max\PYZus{}weight}\PY{p}{)}\PY{p}{:}
    \PY{n}{e} \PY{o}{=} \PY{n}{ExtendedTypesBuilderWeight}\PY{p}{(}\PY{n}{Msurface}\PY{p}{(}\PY{n}{genus}\PY{p}{)}\PY{p}{,} \PY{l+m+mi}{4}\PY{o}{*}\PY{n}{genus}\PY{p}{,} \PY{n}{weights}\PY{p}{,} \PY{n}{max\PYZus{}weight}\PY{p}{)}
    \PY{n}{r} \PY{o}{=} \PY{n}{RhoEstimatorExt}\PY{p}{(}\PY{n}{e}\PY{p}{)}
    \PY{k}{print} \PY{l+s}{"}\PY{l+s}{genus = }\PY{l+s}{"}\PY{p}{,} \PY{n}{genus}\PY{p}{,} \PY{l+s}{"}\PY{l+s}{, max\PYZus{}weight = }\PY{l+s}{"}\PY{p}{,} \PY{n}{max\PYZus{}weight}\PY{p}{,}\PYZbs{}
          \PY{l+s}{"}\PY{l+s}{, matrix size = }\PY{l+s}{"}\PY{p}{,} \PY{n}{r}\PY{o}{.}\PY{n}{M}\PY{o}{.}\PY{n}{nrows}\PY{p}{(}\PY{p}{)}
    \PY{k}{print} \PY{l+s}{"}\PY{l+s}{estimate = }\PY{l+s}{"}\PY{p}{,} \PY{n}{r}\PY{o}{.}\PY{n}{estimate}\PY{p}{(}\PY{p}{)}
\end{Verbatim}

\begin{Verbatim}[commandchars=\\\{\}]
\PY{k}{for} \PY{n}{max\PYZus{}length} \PY{o+ow}{in} \PY{n+nb}{range}\PY{p}{(}\PY{l+m+mi}{1}\PY{p}{,}\PY{l+m+mi}{8}\PY{p}{,}\PY{l+m+mi}{2}\PY{p}{)}\PY{p}{:}
    \PY{n}{time} \PY{n}{test\PYZus{}length}\PY{p}{(}\PY{l+m+mi}{2}\PY{p}{,} \PY{n}{max\PYZus{}length}\PY{p}{)}
    \PY{k}{print} \PY{l+s}{"}\PY{l+s}{"}
\end{Verbatim}
\begin{Verbatim}[formatcom=\color{blue}]
genus =  2 , max_length =  1 , matrix size =  4
estimate =  0.662477976598
Time: CPU 0.00 s, Wall: 0.01 s

genus =  2 , max_length =  3 , matrix size =  25
estimate =  0.6626394462
Time: CPU 0.00 s, Wall: 0.00 s

genus =  2 , max_length =  5 , matrix size =  148
estimate =  0.662694226446
Time: CPU 0.01 s, Wall: 0.01 s

genus =  2 , max_length =  7 , matrix size =  865
estimate =  0.662720574395
Time: CPU 0.06 s, Wall: 0.06 s
\end{Verbatim}

 Now that we have a fast enough algorithm, let us compare what we get using lengths or using weights.

\begin{Verbatim}[commandchars=\\\{\}]
\PY{k}{for} \PY{n}{max\PYZus{}weight} \PY{o+ow}{in} \PY{n+nb}{range}\PY{p}{(}\PY{l+m+mi}{2}\PY{p}{,}\PY{l+m+mi}{9}\PY{p}{,}\PY{l+m+mi}{2}\PY{p}{)}\PY{p}{:}
    \PY{n}{time} \PY{n}{test\PYZus{}weight}\PY{p}{(}\PY{l+m+mi}{2}\PY{p}{,} \PY{p}{[}\PY{l+m+mi}{1}\PY{p}{,}\PY{l+m+mi}{2}\PY{p}{,}\PY{l+m+mi}{3}\PY{p}{,}\PY{l+m+mi}{4}\PY{p}{]}\PY{p}{,} \PY{n}{max\PYZus{}weight}\PY{p}{)}
    \PY{k}{print} \PY{l+s}{"}\PY{l+s}{"}
\end{Verbatim}
\begin{Verbatim}[formatcom=\color{blue}]
genus =  2 , max_weight =  2 , matrix size =  13
estimate =  0.662607354086
Time: CPU 0.01 s, Wall: 0.01 s

genus =  2 , max_weight =  4 , matrix size =  37
estimate =  0.662663626794
Time: CPU 0.01 s, Wall: 0.01 s

genus =  2 , max_weight =  6 , matrix size =  109
estimate =  0.66269793275
Time: CPU 0.02 s, Wall: 0.01 s

genus =  2 , max_weight =  8 , matrix size =  319
estimate =  0.662717774996
Time: CPU 0.03 s, Wall: 0.03 s
\end{Verbatim}

 At comparable matrix size, the estimates with weights are better than the corresponding estimates with length. For instance, a matrix size of 109 generated using weights gives a better estimate than a matrix size of 148 generated using length. This is not surprising, since using weights separates more those points that are more typical. To get the best possible estimates, we will therefore use weights (and very large matrices!)

\begin{Verbatim}[commandchars=\\\{\}]
\PY{n}{time} \PY{n}{test\PYZus{}weight}\PY{p}{(}\PY{n}{genus} \PY{o}{=} \PY{l+m+mi}{2}\PY{p}{,} \PY{n}{weights} \PY{o}{=} \PY{p}{[}\PY{l+m+mi}{1}\PY{p}{,}\PY{l+m+mi}{2}\PY{p}{,}\PY{l+m+mi}{3}\PY{p}{,}\PY{l+m+mi}{4}\PY{p}{]}\PY{p}{,} \PY{n}{max\PYZus{}weight} \PY{o}{=} \PY{l+m+mi}{25}\PY{p}{)}
\end{Verbatim}
\begin{Verbatim}[formatcom=\color{blue}]
genus =  2 , max_weight =  25 , matrix size =  2774629
estimate =  0.66275789907
Time: CPU 2153.67 s, Wall: 2159.49 s
\end{Verbatim}

\begin{Verbatim}[commandchars=\\\{\}]
\PY{n}{time} \PY{n}{test\PYZus{}weight}\PY{p}{(}\PY{n}{genus} \PY{o}{=} \PY{l+m+mi}{3}\PY{p}{,} \PY{n}{weights} \PY{o}{=} \PY{p}{[}\PY{l+m+mi}{1}\PY{p}{,}\PY{l+m+mi}{2}\PY{p}{,}\PY{l+m+mi}{3}\PY{p}{,}\PY{l+m+mi}{4}\PY{p}{,}\PY{l+m+mi}{5}\PY{p}{,}\PY{l+m+mi}{6}\PY{p}{]}\PY{p}{,} \PY{n}{max\PYZus{}weight} \PY{o}{=} \PY{l+m+mi}{24}\PY{p}{)}
\end{Verbatim}
\begin{Verbatim}[formatcom=\color{blue}]
genus =  3 , max_weight =  24 , matrix size =  2943021
estimate =  0.552773556459
Time: CPU 510.54 s, Wall: 511.90 s
\end{Verbatim}

\begin{Verbatim}[commandchars=\\\{\}]
\PY{n}{time} \PY{n}{test\PYZus{}weight}\PY{p}{(}\PY{n}{genus} \PY{o}{=} \PY{l+m+mi}{4}\PY{p}{,} \PY{n}{weights} \PY{o}{=} \PY{p}{[}\PY{l+m+mi}{1}\PY{p}{,}\PY{l+m+mi}{2}\PY{p}{,}\PY{l+m+mi}{3}\PY{p}{,}\PY{l+m+mi}{4}\PY{p}{,}\PY{l+m+mi}{5}\PY{p}{,}\PY{l+m+mi}{6}\PY{p}{,}\PY{l+m+mi}{7}\PY{p}{,}\PY{l+m+mi}{8}\PY{p}{]}\PY{p}{,} \PY{n}{max\PYZus{}weight} \PY{o}{=} \PY{l+m+mi}{24}\PY{p}{)}
\end{Verbatim}
\begin{Verbatim}[formatcom=\color{blue}]
genus =  4 , max_weight =  24 , matrix size =  4120495
estimate =  0.484122920682
Time: CPU 745.34 s, Wall: 747.25 s
\end{Verbatim}

\begin{Verbatim}[commandchars=\\\{\}]
\PY{k}{def} \PY{n+nf}{MsurfaceExt}\PY{p}{(}\PY{n}{g}\PY{p}{)}\PY{p}{:}
    \PY{l+s+sd}{"""Returns a (2g+2 times 2g+2) matrix M.}
\PY{l+s+sd}{    M[i,j] is the number of successors of modified type i of }
\PY{l+s+sd}{    a point of modified type j, in the surface group Gamma\PYZus{}g.}
\PY{l+s+sd}{    Compared to Cannon, we add two types 1prime and 2prime, corresponding to}
\PY{l+s+sd}{    the successors of a point of type 2g-1 that can be in a loop with}
\PY{l+s+sd}{    further non-uniqueness.}
\PY{l+s+sd}{       }
\PY{l+s+sd}{    The correspondance between the types and the matrix indices follows:}
\PY{l+s+sd}{           type i      <-> index i-1}
\PY{l+s+sd}{           type 1prime <-> index 2*g}
\PY{l+s+sd}{           type 2prime <-> index 2*g+1}
\PY{l+s+sd}{    We will also add later an "ambiguous" type, with index 2*g+2}
\PY{l+s+sd}{    """}
    \PY{n}{M} \PY{o}{=} \PY{n}{Matrix}\PY{p}{(}\PY{l+m+mi}{2}\PY{o}{*}\PY{n}{g}\PY{o}{+}\PY{l+m+mi}{2}\PY{p}{,} \PY{l+m+mi}{2}\PY{o}{*}\PY{n}{g}\PY{o}{+}\PY{l+m+mi}{2}\PY{p}{,} \PY{l+m+mi}{0}\PY{p}{)}

    \PY{k}{for} \PY{n}{j} \PY{o+ow}{in} \PY{n+nb}{xrange}\PY{p}{(}\PY{l+m+mi}{0}\PY{p}{,} \PY{l+m+mi}{2}\PY{o}{*}\PY{n}{g}\PY{o}{-}\PY{l+m+mi}{2}\PY{p}{)}\PY{p}{:}
        \PY{n}{M}\PY{p}{[}\PY{l+m+mi}{0}  \PY{p}{,}\PY{n}{j}\PY{p}{]}  \PY{o}{=} \PY{l+m+mi}{4}\PY{o}{*}\PY{n}{g}\PY{o}{-}\PY{l+m+mi}{3}
        \PY{n}{M}\PY{p}{[}\PY{l+m+mi}{1}  \PY{p}{,}\PY{n}{j}\PY{p}{]}  \PY{o}{=} \PY{l+m+mi}{1}
        \PY{n}{M}\PY{p}{[}\PY{n}{j}\PY{o}{+}\PY{l+m+mi}{1}\PY{p}{,}\PY{n}{j}\PY{p}{]} \PY{o}{+}\PY{o}{=} \PY{l+m+mi}{1}

    \PY{n}{M}\PY{p}{[}\PY{l+m+mi}{0}\PY{p}{,}     \PY{l+m+mi}{2}\PY{o}{*}\PY{n}{g}\PY{o}{-}\PY{l+m+mi}{2}\PY{p}{]} \PY{o}{=} \PY{l+m+mi}{4}\PY{o}{*}\PY{n}{g}\PY{o}{-}\PY{l+m+mi}{4}
    \PY{n}{M}\PY{p}{[}\PY{l+m+mi}{2}\PY{o}{*}\PY{n}{g}\PY{p}{,}   \PY{l+m+mi}{2}\PY{o}{*}\PY{n}{g}\PY{o}{-}\PY{l+m+mi}{2}\PY{p}{]} \PY{o}{=} \PY{l+m+mi}{1}
    \PY{n}{M}\PY{p}{[}\PY{l+m+mi}{1}\PY{p}{,}     \PY{l+m+mi}{2}\PY{o}{*}\PY{n}{g}\PY{o}{-}\PY{l+m+mi}{2}\PY{p}{]} \PY{o}{=} \PY{l+m+mi}{1}
    \PY{n}{M}\PY{p}{[}\PY{l+m+mi}{2}\PY{o}{*}\PY{n}{g}\PY{o}{-}\PY{l+m+mi}{1}\PY{p}{,} \PY{l+m+mi}{2}\PY{o}{*}\PY{n}{g}\PY{o}{-}\PY{l+m+mi}{2}\PY{p}{]} \PY{o}{=} \PY{l+m+mi}{1}

    \PY{n}{M}\PY{p}{[}\PY{l+m+mi}{0}\PY{p}{,} \PY{l+m+mi}{2}\PY{o}{*}\PY{n}{g}\PY{o}{-}\PY{l+m+mi}{1}\PY{p}{]} \PY{o}{=} \PY{l+m+mi}{4}\PY{o}{*}\PY{n}{g}\PY{o}{-}\PY{l+m+mi}{4}
    \PY{n}{M}\PY{p}{[}\PY{l+m+mi}{1}\PY{p}{,} \PY{l+m+mi}{2}\PY{o}{*}\PY{n}{g}\PY{o}{-}\PY{l+m+mi}{1}\PY{p}{]} \PY{o}{=} \PY{l+m+mi}{2}

    \PY{n}{M}\PY{p}{[}\PY{l+m+mi}{0}\PY{p}{,}     \PY{l+m+mi}{2}\PY{o}{*}\PY{n}{g}\PY{p}{]} \PY{o}{=} \PY{l+m+mi}{4}\PY{o}{*}\PY{n}{g}\PY{o}{-}\PY{l+m+mi}{3}
    \PY{n}{M}\PY{p}{[}\PY{l+m+mi}{1}\PY{p}{,}     \PY{l+m+mi}{2}\PY{o}{*}\PY{n}{g}\PY{p}{]} \PY{o}{=} \PY{l+m+mi}{1}
    \PY{n}{M}\PY{p}{[}\PY{l+m+mi}{2}\PY{o}{*}\PY{n}{g}\PY{o}{+}\PY{l+m+mi}{1}\PY{p}{,} \PY{l+m+mi}{2}\PY{o}{*}\PY{n}{g}\PY{p}{]} \PY{o}{=} \PY{l+m+mi}{1}

    \PY{n}{M}\PY{p}{[}\PY{l+m+mi}{0}\PY{p}{,} \PY{l+m+mi}{2}\PY{o}{*}\PY{n}{g}\PY{o}{+}\PY{l+m+mi}{1}\PY{p}{]} \PY{o}{=} \PY{l+m+mi}{4}\PY{o}{*}\PY{n}{g}\PY{o}{-}\PY{l+m+mi}{3}
    \PY{n}{M}\PY{p}{[}\PY{l+m+mi}{1}\PY{p}{,} \PY{l+m+mi}{2}\PY{o}{*}\PY{n}{g}\PY{o}{+}\PY{l+m+mi}{1}\PY{p}{]} \PY{o}{=} \PY{l+m+mi}{1}
    \PY{n}{M}\PY{p}{[}\PY{l+m+mi}{2}\PY{p}{,} \PY{l+m+mi}{2}\PY{o}{*}\PY{n}{g}\PY{o}{+}\PY{l+m+mi}{1}\PY{p}{]} \PY{o}{=} \PY{l+m+mi}{1}
    \PY{k}{return} \PY{n}{M}


\PY{k}{class} \PY{n+nc}{ExtendedTypesBuilderSurfaceWeightLength}\PY{p}{(}\PY{n}{ExtendedTypesBuilderWeightLength}\PY{p}{)}\PY{p}{:}
    \PY{l+s+sd}{"""This class constructs the extended type for the surface}
\PY{l+s+sd}{    group corresponding to its initialization parameters genus, }
\PY{l+s+sd}{    length, weights:}
\PY{l+s+sd}{    among all words ending at some point, it selects the}
\PY{l+s+sd}{    part that is common to all geodesics ending at that point (with }
\PY{l+s+sd}{    length and  weight at most the initialization parameters), }
\PY{l+s+sd}{    and constructs the types from them.}
\PY{l+s+sd}{    }
\PY{l+s+sd}{    There are 2*g+2 true types, as explained in MsurfaceExt(genus), and one }
\PY{l+s+sd}{    "undetermined" type, corresponding to parts of geodesics that are non-unique. }
\PY{l+s+sd}{    Therefore, weights should be of length 2*g+3. Moreover, the weight of the }
\PY{l+s+sd}{    undetermined type should be maximal among weights, so that extending a geodesic }
\PY{l+s+sd}{    and then possibly replacing some parts by undetermined parts one can only }
\PY{l+s+sd}{    increase the weight, and therefore shorten the part that is selected.}
\PY{l+s+sd}{    }
\PY{l+s+sd}{    matrix() returns the transition matrix for the extended type.}
\PY{l+s+sd}{    extended\PYZus{}types() returns all the words in the extended type.}
\PY{l+s+sd}{    """}
    \PY{k}{def} \PY{n+nf}{\PYZus{}\PYZus{}init\PYZus{}\PYZus{}}\PY{p}{(}\PY{n+nb+bp}{self}\PY{p}{,} \PY{n}{genus}\PY{p}{,} \PY{n}{weights}\PY{p}{,} \PY{n}{max\PYZus{}weight}\PY{p}{,} \PY{n}{length}\PY{p}{)}\PY{p}{:}
        \PY{k}{if} \PY{n+nb}{len}\PY{p}{(}\PY{n}{weights}\PY{p}{)} \PY{o}{!=} \PY{l+m+mi}{2}\PY{o}{*}\PY{n}{genus} \PY{o}{+} \PY{l+m+mi}{3} \PY{o+ow}{or} \PY{n}{weights}\PY{p}{[}\PY{o}{-}\PY{l+m+mi}{1}\PY{p}{]} \PY{o}{!=} \PY{n+nb}{max}\PY{p}{(}\PY{n}{weights}\PY{p}{)}\PY{p}{:}
            \PY{k}{raise} \PY{n+ne}{ValueError}\PY{p}{,} \PY{l+s}{"}\PY{l+s}{incorrect parameters}\PY{l+s}{"}
        \PY{n}{ExtendedTypesBuilderWeightLength}\PY{o}{.}\PY{n}{\PYZus{}\PYZus{}init\PYZus{}\PYZus{}}\PY{p}{(}\PY{n+nb+bp}{self}\PY{p}{,} \PY{n}{MsurfaceExt}\PY{p}{(}\PY{n}{genus}\PY{p}{)}\PY{p}{,}
                 \PY{l+m+mi}{4}\PY{o}{*}\PY{n}{genus}\PY{p}{,} \PY{n}{weights}\PY{p}{,} \PY{n}{max\PYZus{}weight}\PY{p}{,} \PY{n}{length}\PY{p}{)}
        \PY{n+nb+bp}{self}\PY{o}{.}\PY{n}{genus} \PY{o}{=} \PY{n}{genus}
        \PY{n+nb+bp}{self}\PY{o}{.}\PY{n}{words} \PY{o}{=} \PY{n+nb+bp}{None}
        \PY{n+nb+bp}{self}\PY{o}{.}\PY{n}{predecessors} \PY{o}{=} \PY{p}{[}\PY{l+m+mi}{1} \PY{k}{for} \PY{n}{i} \PY{o+ow}{in} \PY{n+nb}{xrange}\PY{p}{(}\PY{l+m+mi}{2}\PY{o}{*}\PY{n}{genus}\PY{o}{+}\PY{l+m+mi}{3}\PY{p}{)}\PY{p}{]}

    \PY{k}{def} \PY{n+nf}{mark\PYZus{}ambiguous}\PY{p}{(}\PY{n+nb+bp}{self}\PY{p}{,} \PY{n}{t}\PY{p}{)}\PY{p}{:}
        \PY{l+s+sd}{"""In the word t, locates the parts that are not common to all }
\PY{l+s+sd}{        geodesics ending by such a word, and replaces them by 2*g+2.}
\PY{l+s+sd}{        Accepts as input a word already with ambiguities.}
\PY{l+s+sd}{        """}
        \PY{n}{u} \PY{o}{=} \PY{n+nb}{list}\PY{p}{(}\PY{n}{t}\PY{p}{)}
        \PY{n}{g} \PY{o}{=} \PY{n+nb+bp}{self}\PY{o}{.}\PY{n}{genus}
        \PY{k}{try}\PY{p}{:}
            \PY{n}{i} \PY{o}{=} \PY{n}{t}\PY{o}{.}\PY{n}{index}\PY{p}{(}\PY{l+m+mi}{2}\PY{o}{*}\PY{n}{g}\PY{o}{-}\PY{l+m+mi}{1}\PY{p}{)}
            \PY{k}{while} \PY{n}{i} \PY{o}{<} \PY{n+nb}{len}\PY{p}{(}\PY{n}{u}\PY{p}{)}\PY{p}{:}
                \PY{k}{for} \PY{n}{j} \PY{o+ow}{in} \PY{n+nb}{xrange}\PY{p}{(}\PY{n}{i}\PY{o}{+}\PY{l+m+mi}{1}\PY{p}{,} \PY{n}{i}\PY{o}{+}\PY{l+m+mi}{2}\PY{o}{*}\PY{n}{g}\PY{p}{)}\PY{p}{:}
                    \PY{n}{u}\PY{p}{[}\PY{n}{j}\PY{p}{]} \PY{o}{=} \PY{l+m+mi}{2}\PY{o}{*}\PY{n}{g}\PY{o}{+}\PY{l+m+mi}{2}
                \PY{n}{i} \PY{o}{=} \PY{n}{i} \PY{o}{+} \PY{l+m+mi}{2}\PY{o}{*}\PY{n}{g}\PY{o}{-}\PY{l+m+mi}{1}
                \PY{k}{if} \PY{o+ow}{not} \PY{p}{(}\PY{n}{t}\PY{p}{[}\PY{n}{i}\PY{p}{]} \PY{o}{==} \PY{l+m+mi}{2}\PY{o}{*}\PY{n}{g}\PY{o}{-}\PY{l+m+mi}{1} \PY{o+ow}{or} \PY{p}{(}\PY{n}{t}\PY{p}{[}\PY{n}{i}\PY{p}{]} \PY{o}{==} \PY{l+m+mi}{2}\PY{o}{*}\PY{n}{g} \PY{o+ow}{and} \PY{n}{t}\PY{p}{[}\PY{n}{i}\PY{o}{-}\PY{l+m+mi}{1}\PY{p}{]} \PY{o}{==} \PY{l+m+mi}{2}\PY{o}{*}\PY{n}{g}\PY{o}{+}\PY{l+m+mi}{1}\PY{p}{)}\PY{p}{)}\PY{p}{:}
                    \PY{n}{i} \PY{o}{=} \PY{n}{t}\PY{o}{.}\PY{n}{index}\PY{p}{(}\PY{l+m+mi}{2}\PY{o}{*}\PY{n}{g}\PY{o}{-}\PY{l+m+mi}{1}\PY{p}{,} \PY{n}{i}\PY{p}{)}
        \PY{k}{except} \PY{p}{(}\PY{n+ne}{ValueError}\PY{p}{,} \PY{n+ne}{IndexError}\PY{p}{)}\PY{p}{:}
            \PY{k}{pass}
        \PY{k}{return} \PY{n+nb}{tuple}\PY{p}{(}\PY{n}{u}\PY{p}{)}

    \PY{k}{def} \PY{n+nf}{truncate}\PY{p}{(}\PY{n+nb+bp}{self}\PY{p}{,} \PY{n}{t}\PY{p}{)}\PY{p}{:}
        \PY{l+s+sd}{"""Truncates a word t to its admissible part.}
\PY{l+s+sd}{        Marks the ambiguous part with type 2g+2 if necessary.}
\PY{l+s+sd}{        """}
        \PY{k}{return} \PY{n}{ExtendedTypesBuilderWeightLength}\PY{o}{.}\PY{n}{truncate}\PY{p}{(}\PY{n+nb+bp}{self}\PY{p}{,}
                                                         \PY{n+nb+bp}{self}\PY{o}{.}\PY{n}{mark\PYZus{}ambiguous}\PY{p}{(}\PY{n}{t}\PY{p}{)}\PY{p}{)}

    \PY{k}{def} \PY{n+nf}{compute\PYZus{}extended\PYZus{}types}\PY{p}{(}\PY{n+nb+bp}{self}\PY{p}{)}\PY{p}{:}
        \PY{l+s+sd}{"""Puts in self.words a list of all admissible types.}
\PY{l+s+sd}{        The construction starts from a collection of types, }
\PY{l+s+sd}{        and tries to extend them.}
\PY{l+s+sd}{        If they are not extendable, put them in result.}
\PY{l+s+sd}{        Otherwise, go on with the new types.}
\PY{l+s+sd}{        """}
        \PY{n}{result}       \PY{o}{=} \PY{p}{[}\PY{p}{]}
        \PY{c}{# in current\PYZus{}pool, store (word, weight of the word, }
        \PY{c}{#    number of ambiguous letters at the end of word modulo 2g-1) }
        \PY{n}{current\PYZus{}pool} \PY{o}{=} \PY{p}{[}\PY{p}{(}\PY{p}{(}\PY{n}{i}\PY{p}{,}\PY{p}{)}\PY{p}{,} \PY{n+nb+bp}{self}\PY{o}{.}\PY{n}{weights}\PY{p}{[}\PY{n}{i}\PY{p}{]}\PY{p}{,} \PY{l+m+mi}{0}\PY{p}{)} \PY{k}{for} \PY{n}{i} \PY{o+ow}{in} \PY{n+nb}{xrange}\PY{p}{(}\PY{n+nb+bp}{self}\PY{o}{.}\PY{n}{types}\PY{p}{)}\PY{p}{]}
        \PY{n}{length}       \PY{o}{=} \PY{l+m+mi}{1}
        \PY{n}{g}            \PY{o}{=} \PY{n+nb+bp}{self}\PY{o}{.}\PY{n}{genus}
        \PY{n}{types}        \PY{o}{=} \PY{n+nb+bp}{self}\PY{o}{.}\PY{n}{types}
        \PY{k}{while} \PY{n}{length} \PY{o}{<} \PY{n+nb+bp}{self}\PY{o}{.}\PY{n}{length}\PY{p}{:}
            \PY{n}{length}   \PY{o}{+}\PY{o}{=} \PY{l+m+mi}{1}
            \PY{n}{temp\PYZus{}pool} \PY{o}{=} \PY{p}{[}\PY{p}{]}
            \PY{k}{for} \PY{p}{(}\PY{n}{t}\PY{p}{,} \PY{n}{w}\PY{p}{,} \PY{n}{undet}\PY{p}{)} \PY{o+ow}{in} \PY{n}{current\PYZus{}pool}\PY{p}{:}
                \PY{k}{if} \PY{n}{w} \PY{o}{>} \PY{n+nb+bp}{self}\PY{o}{.}\PY{n}{max\PYZus{}weight}\PY{p}{:}
                    \PY{n}{result} \PY{o}{+}\PY{o}{=} \PY{p}{[}\PY{n}{t}\PY{p}{]}
                \PY{k}{elif} \PY{n}{t}\PY{p}{[}\PY{o}{-}\PY{l+m+mi}{1}\PY{p}{]} \PY{o}{==} \PY{l+m+mi}{2}\PY{o}{*}\PY{n}{g}\PY{o}{-}\PY{l+m+mi}{1}\PY{p}{:}
                    \PY{n}{temp\PYZus{}pool} \PY{o}{+}\PY{o}{=} \PY{p}{[}\PY{p}{(}\PY{n}{t}\PY{o}{+}\PY{p}{(}\PY{l+m+mi}{2}\PY{o}{*}\PY{n}{g}\PY{o}{+}\PY{l+m+mi}{2}\PY{p}{,}\PY{p}{)}\PY{p}{,} \PY{n}{w}\PY{o}{+}\PY{n+nb+bp}{self}\PY{o}{.}\PY{n}{weights}\PY{p}{[}\PY{l+m+mi}{2}\PY{o}{*}\PY{n}{g}\PY{o}{+}\PY{l+m+mi}{2}\PY{p}{]}\PY{p}{,} \PY{l+m+mi}{1}\PY{p}{)}\PY{p}{]}
                \PY{k}{elif} \PY{n}{t}\PY{p}{[}\PY{o}{-}\PY{l+m+mi}{1}\PY{p}{]} \PY{o}{!=} \PY{l+m+mi}{2}\PY{o}{*}\PY{n}{g}\PY{o}{+}\PY{l+m+mi}{2}\PY{p}{:}
                    \PY{n}{temp\PYZus{}pool} \PY{o}{+}\PY{o}{=} \PY{p}{[}\PY{p}{(}\PY{n}{t}\PY{o}{+}\PY{p}{(}\PY{n}{i}\PY{p}{,}\PY{p}{)}\PY{p}{,} \PY{n}{w} \PY{o}{+} \PY{n+nb+bp}{self}\PY{o}{.}\PY{n}{weights}\PY{p}{[}\PY{n}{i}\PY{p}{]}\PY{p}{,} \PY{l+m+mi}{0}\PY{p}{)}
                                    \PY{k}{for} \PY{n}{i} \PY{o+ow}{in} \PY{n+nb}{xrange}\PY{p}{(}\PY{n}{types}\PY{p}{)}
                                    \PY{k}{if} \PY{n+nb+bp}{self}\PY{o}{.}\PY{n}{M}\PY{p}{[}\PY{n}{t}\PY{p}{[}\PY{o}{-}\PY{l+m+mi}{1}\PY{p}{]}\PY{p}{,} \PY{n}{i}\PY{p}{]} \PY{o}{>} \PY{l+m+mi}{0}\PY{p}{]}
                \PY{k}{else}\PY{p}{:}
                    \PY{c}{# last letter is ambiguous}
                    \PY{n}{undet} \PY{o}{+}\PY{o}{=} \PY{l+m+mi}{1}
                    \PY{k}{if} \PY{n}{undet} \PY{o}{==} \PY{l+m+mi}{2}\PY{o}{*}\PY{n}{g}\PY{p}{:}
                        \PY{c}{# a half-turn around an octagon is finished.}
                        \PY{c}{# the word can be extended with any definite type,}
                        \PY{c}{# or again with an ambiguity.}
                        \PY{n}{temp\PYZus{}pool} \PY{o}{+}\PY{o}{=} \PY{p}{[}\PY{p}{(}\PY{n}{t}\PY{o}{+}\PY{p}{(}\PY{n}{i}\PY{p}{,}\PY{p}{)}\PY{p}{,} \PY{n}{w}\PY{o}{+}\PY{n+nb+bp}{self}\PY{o}{.}\PY{n}{weights}\PY{p}{[}\PY{n}{i}\PY{p}{]}\PY{p}{,} \PY{l+m+mi}{0}\PY{p}{)}
                                      \PY{k}{for} \PY{n}{i} \PY{o+ow}{in} \PY{n+nb}{xrange}\PY{p}{(}\PY{n}{types}\PY{p}{)}\PY{p}{]}
                        \PY{n}{undet} \PY{o}{=} \PY{l+m+mi}{1}
                    \PY{n}{temp\PYZus{}pool} \PY{o}{+}\PY{o}{=} \PY{p}{[}\PY{p}{(}\PY{n}{t}\PY{o}{+}\PY{p}{(}\PY{l+m+mi}{2}\PY{o}{*}\PY{n}{g}\PY{o}{+}\PY{l+m+mi}{2}\PY{p}{,}\PY{p}{)}\PY{p}{,} \PY{n}{w}\PY{o}{+}\PY{n+nb+bp}{self}\PY{o}{.}\PY{n}{weights}\PY{p}{[}\PY{l+m+mi}{2}\PY{o}{*}\PY{n}{g}\PY{o}{+}\PY{l+m+mi}{2}\PY{p}{]}\PY{p}{,} \PY{n}{undet}\PY{p}{)}\PY{p}{]}
            \PY{n}{current\PYZus{}pool} \PY{o}{=} \PY{n}{temp\PYZus{}pool}
        \PY{n}{result} \PY{o}{+}\PY{o}{=} \PY{p}{[}\PY{n}{t} \PY{k}{for} \PY{p}{(}\PY{n}{t}\PY{p}{,}\PY{n}{w}\PY{p}{,} \PY{n}{undet}\PY{p}{)} \PY{o+ow}{in} \PY{n}{current\PYZus{}pool}\PY{p}{]}
        \PY{n+nb+bp}{self}\PY{o}{.}\PY{n}{words} \PY{o}{=} \PY{n}{result}

    \PY{k}{def} \PY{n+nf}{growth\PYZus{}and\PYZus{}A}\PY{p}{(}\PY{n+nb+bp}{self}\PY{p}{)}\PY{p}{:}
        \PY{n}{growth}\PY{p}{,} \PY{n}{Azero} \PY{o}{=} \PY{n+nb+bp}{self}\PY{o}{.}\PY{n}{growth\PYZus{}and\PYZus{}Azero}\PY{p}{(}\PY{p}{)}
        \PY{n}{types}         \PY{o}{=} \PY{n+nb+bp}{self}\PY{o}{.}\PY{n}{types}
        \PY{n}{g}             \PY{o}{=} \PY{n+nb+bp}{self}\PY{o}{.}\PY{n}{genus}
        \PY{n}{words}         \PY{o}{=} \PY{n+nb+bp}{self}\PY{o}{.}\PY{n}{extended\PYZus{}types}\PY{p}{(}\PY{p}{)}

        \PY{c}{# precompute all the possible values of A0(w\PYZus{}\PYZob{}k-1\PYZcb{})growth\PYZca{}\PYZob{}-k+1\PYZcb{}}
        \PY{n}{precompute} \PY{o}{=} \PY{p}{[}\PY{p}{[}\PY{l+m+mi}{0} \PY{k}{for} \PY{n}{i} \PY{o+ow}{in} \PY{n+nb}{xrange}\PY{p}{(}\PY{n+nb+bp}{self}\PY{o}{.}\PY{n}{max\PYZus{}length}\PY{p}{)}\PY{p}{]} \PY{k}{for} \PY{n}{j} \PY{o+ow}{in} \PY{n+nb}{xrange}\PY{p}{(}\PY{n}{types}\PY{p}{)}\PY{p}{]}
        \PY{k}{for} \PY{n}{j} \PY{o+ow}{in} \PY{n+nb}{xrange}\PY{p}{(}\PY{n}{types}\PY{p}{)}\PY{p}{:}
            \PY{n}{precompute}\PY{p}{[}\PY{n}{j}\PY{p}{]}\PY{p}{[}\PY{l+m+mi}{0}\PY{p}{]} \PY{o}{=} \PY{n}{Azero}\PY{p}{[}\PY{n}{j}\PY{p}{]}
            \PY{k}{for} \PY{n}{i} \PY{o+ow}{in} \PY{n+nb}{xrange}\PY{p}{(}\PY{l+m+mi}{1}\PY{p}{,} \PY{n+nb+bp}{self}\PY{o}{.}\PY{n}{max\PYZus{}length}\PY{p}{)}\PY{p}{:}
                \PY{n}{precompute}\PY{p}{[}\PY{n}{j}\PY{p}{]}\PY{p}{[}\PY{n}{i}\PY{p}{]} \PY{o}{=} \PY{n}{precompute}\PY{p}{[}\PY{n}{j}\PY{p}{]}\PY{p}{[}\PY{n}{i}\PY{o}{-}\PY{l+m+mi}{1}\PY{p}{]}\PY{o}{/}\PY{n}{growth}

        \PY{c}{# compute all the values of A}
        \PY{n}{A} \PY{o}{=} \PY{p}{[}\PY{l+m+mi}{0} \PY{k}{for} \PY{n}{i} \PY{o+ow}{in} \PY{n+nb}{xrange}\PY{p}{(}\PY{n+nb}{len}\PY{p}{(}\PY{n}{words}\PY{p}{)}\PY{p}{)}\PY{p}{]}
        \PY{k}{for} \PY{p}{(}\PY{n}{i}\PY{p}{,}\PY{n}{t}\PY{p}{)} \PY{o+ow}{in} \PY{n+nb}{enumerate}\PY{p}{(}\PY{n}{words}\PY{p}{)}\PY{p}{:}
            \PY{n}{m} \PY{o}{=} \PY{l+m+mi}{1}
            \PY{n}{ambiguous\PYZus{}count} \PY{o}{=} \PY{l+m+mi}{0}
            \PY{c}{# find the last non-ambiguous position}
            \PY{k}{for} \PY{n}{true\PYZus{}length} \PY{o+ow}{in} \PY{n+nb}{xrange}\PY{p}{(}\PY{n+nb}{len}\PY{p}{(}\PY{n}{t}\PY{p}{)}\PY{o}{-}\PY{l+m+mi}{1}\PY{p}{,} \PY{o}{-}\PY{l+m+mi}{1}\PY{p}{,} \PY{o}{-}\PY{l+m+mi}{1}\PY{p}{)}\PY{p}{:}
                \PY{k}{if} \PY{n}{t}\PY{p}{[}\PY{n}{true\PYZus{}length}\PY{p}{]} \PY{o}{!=} \PY{l+m+mi}{2}\PY{o}{*}\PY{n}{g}\PY{o}{+}\PY{l+m+mi}{2}\PY{p}{:}
                    \PY{k}{break}
            \PY{c}{# count the multiplicities between 0 and true\PYZus{}length}
            \PY{k}{for} \PY{n}{j} \PY{o+ow}{in} \PY{n+nb}{xrange}\PY{p}{(}\PY{n}{true\PYZus{}length}\PY{p}{)}\PY{p}{:}
                \PY{k}{if} \PY{n}{t}\PY{p}{[}\PY{n}{j}\PY{p}{]} \PY{o}{!=} \PY{l+m+mi}{2}\PY{o}{*}\PY{n}{g}\PY{o}{+}\PY{l+m+mi}{2} \PY{o+ow}{and} \PY{n}{t}\PY{p}{[}\PY{n}{j}\PY{o}{+}\PY{l+m+mi}{1}\PY{p}{]} \PY{o}{!=} \PY{l+m+mi}{2}\PY{o}{*}\PY{n}{g}\PY{o}{+}\PY{l+m+mi}{2}\PY{p}{:}
                    \PY{n}{m} \PY{o}{*}\PY{o}{=} \PY{n+nb+bp}{self}\PY{o}{.}\PY{n}{M}\PY{p}{[}\PY{n}{t}\PY{p}{[}\PY{n}{j}\PY{p}{]}\PY{p}{,} \PY{n}{t}\PY{p}{[}\PY{n}{j}\PY{o}{+}\PY{l+m+mi}{1}\PY{p}{]}\PY{p}{]}
                \PY{k}{elif} \PY{n}{t}\PY{p}{[}\PY{n}{j}\PY{p}{]} \PY{o}{==} \PY{l+m+mi}{2}\PY{o}{*}\PY{n}{g}\PY{o}{+}\PY{l+m+mi}{2} \PY{o+ow}{and} \PY{n}{t}\PY{p}{[}\PY{n}{j}\PY{o}{+}\PY{l+m+mi}{1}\PY{p}{]} \PY{o}{!=} \PY{l+m+mi}{2}\PY{o}{*}\PY{n}{g}\PY{o}{+}\PY{l+m+mi}{2}\PY{p}{:}
                    \PY{n}{mult} \PY{o}{=} \PY{l+m+mi}{4}\PY{o}{*}\PY{n}{g}\PY{o}{-}\PY{l+m+mi}{2} \PY{k}{if} \PY{n}{t}\PY{p}{[}\PY{n}{j}\PY{o}{+}\PY{l+m+mi}{1}\PY{p}{]} \PY{o}{!=} \PY{l+m+mi}{2}\PY{o}{*}\PY{n}{g}\PY{o}{-}\PY{l+m+mi}{1} \PY{o+ow}{and} \PY{n}{t}\PY{p}{[}\PY{n}{j}\PY{o}{+}\PY{l+m+mi}{1}\PY{p}{]} \PY{o}{!=} \PY{l+m+mi}{2}\PY{o}{*}\PY{n}{g}\PY{o}{-}\PY{l+m+mi}{2} \PY{k}{else} \PY{l+m+mi}{4}\PY{o}{*}\PY{n}{g}\PY{o}{-}\PY{l+m+mi}{3}
                    \PY{n}{m} \PY{o}{*}\PY{o}{=} \PY{n}{mult}
                    \PY{n}{ambiguous\PYZus{}count} \PY{o}{=} \PY{n}{ambiguous\PYZus{}count} \PY{o}{-} \PY{p}{(}\PY{l+m+mi}{2}\PY{o}{*}\PY{n}{g}\PY{o}{-}\PY{l+m+mi}{1}\PY{p}{)}
                \PY{k}{else}\PY{p}{:}
                    \PY{n}{ambiguous\PYZus{}count} \PY{o}{+}\PY{o}{=} \PY{l+m+mi}{1}
            \PY{n}{m} \PY{o}{=} \PY{n}{m} \PY{o}{*} \PY{l+m+mi}{2}\PY{o}{\PYZca{}}\PY{p}{(}\PY{n}{ambiguous\PYZus{}count}\PY{o}{/}\PY{p}{(}\PY{l+m+mi}{2}\PY{o}{*}\PY{n}{g}\PY{o}{-}\PY{l+m+mi}{1}\PY{p}{)}\PY{p}{)}
            \PY{c}{# count the multiplicities due to the ambiguities after true\PYZus{}length}
            \PY{n}{last\PYZus{}letter} \PY{o}{=} \PY{n}{t}\PY{p}{[}\PY{n}{true\PYZus{}length}\PY{p}{]}
            \PY{n}{remainder}   \PY{o}{=} \PY{n}{floor}\PY{p}{(}\PY{p}{(}\PY{n+nb}{len}\PY{p}{(}\PY{n}{t}\PY{p}{)} \PY{o}{-} \PY{n}{true\PYZus{}length} \PY{o}{-} \PY{l+m+mi}{2}\PY{p}{)}\PY{o}{/}\PY{p}{(}\PY{l+m+mi}{2}\PY{o}{*}\PY{n}{g}\PY{o}{-}\PY{l+m+mi}{1}\PY{p}{)}\PY{p}{)}
            \PY{k}{if} \PY{n}{remainder} \PY{o}{>} \PY{l+m+mi}{0}\PY{p}{:}
                \PY{n}{true\PYZus{}length} \PY{o}{+}\PY{o}{=} \PY{n}{remainder}\PY{o}{*}\PY{p}{(}\PY{l+m+mi}{2}\PY{o}{*}\PY{n}{g}\PY{o}{-}\PY{l+m+mi}{1}\PY{p}{)}
                \PY{n}{m} \PY{o}{=} \PY{n}{m}\PY{o}{*}\PY{l+m+mi}{2}\PY{o}{\PYZca{}}\PY{n}{remainder}
            \PY{n}{A}\PY{p}{[}\PY{n}{i}\PY{p}{]} \PY{o}{=} \PY{n}{m} \PY{o}{*} \PY{n}{precompute}\PY{p}{[}\PY{n}{last\PYZus{}letter}\PY{p}{]}\PY{p}{[}\PY{n}{true\PYZus{}length}\PY{p}{]}
        \PY{k}{return} \PY{n}{growth}\PY{p}{,} \PY{n}{A}


\PY{k}{class} \PY{n+nc}{ExtendedTypesBuilderSurfaceLength}\PY{p}{(}\PY{n}{ExtendedTypesBuilderSurfaceWeightLength}\PY{p}{)}\PY{p}{:}
    \PY{l+s+sd}{"""Extended type builder that only selects admissible}
\PY{l+s+sd}{    words through their length"""}
    \PY{k}{def} \PY{n+nf}{\PYZus{}\PYZus{}init\PYZus{}\PYZus{}}\PY{p}{(}\PY{n+nb+bp}{self}\PY{p}{,} \PY{n}{genus}\PY{p}{,} \PY{n}{length}\PY{p}{)}\PY{p}{:}
        \PY{n}{ExtendedTypesBuilderSurfaceWeightLength}\PY{o}{.}\PY{n}{\PYZus{}\PYZus{}init\PYZus{}\PYZus{}}\PY{p}{(}\PY{n+nb+bp}{self}\PY{p}{,}
                 \PY{n}{genus}\PY{p}{,} \PY{p}{[}\PY{l+m+mi}{0} \PY{k}{for} \PY{n}{i} \PY{o+ow}{in} \PY{n+nb}{xrange}\PY{p}{(}\PY{l+m+mi}{2}\PY{o}{*}\PY{n}{genus}\PY{o}{+}\PY{l+m+mi}{3}\PY{p}{)}\PY{p}{]}\PY{p}{,} \PY{l+m+mi}{1}\PY{p}{,} \PY{n}{length}\PY{p}{)}


\PY{k}{class} \PY{n+nc}{ExtendedTypesBuilderSurfaceWeight}\PY{p}{(}\PY{n}{ExtendedTypesBuilderSurfaceWeightLength}\PY{p}{)}\PY{p}{:}
    \PY{l+s+sd}{"""Extended type builder that only selects admissible}
\PY{l+s+sd}{    words through their weight"""}
    \PY{k}{def} \PY{n+nf}{\PYZus{}\PYZus{}init\PYZus{}\PYZus{}}\PY{p}{(}\PY{n+nb+bp}{self}\PY{p}{,} \PY{n}{genus}\PY{p}{,} \PY{n}{weights}\PY{p}{,} \PY{n}{max\PYZus{}weight}\PY{p}{)}\PY{p}{:}
        \PY{n}{ExtendedTypesBuilderSurfaceWeightLength}\PY{o}{.}\PY{n}{\PYZus{}\PYZus{}init\PYZus{}\PYZus{}}\PY{p}{(}\PY{n+nb+bp}{self}\PY{p}{,}
                 \PY{n}{genus}\PY{p}{,} \PY{n}{weights}\PY{p}{,} \PY{n}{max\PYZus{}weight}\PY{p}{,} \PY{l+m+mi}{10000}\PY{p}{)}


\PY{k}{def} \PY{n+nf}{test\PYZus{}length\PYZus{}surface}\PY{p}{(}\PY{n}{genus}\PY{p}{,} \PY{n}{max\PYZus{}length}\PY{p}{)}\PY{p}{:}
    \PY{n}{e} \PY{o}{=} \PY{n}{ExtendedTypesBuilderSurfaceLength}\PY{p}{(}\PY{n}{genus}\PY{p}{,} \PY{n}{max\PYZus{}length}\PY{p}{)}
    \PY{n}{r} \PY{o}{=} \PY{n}{RhoEstimatorExt}\PY{p}{(}\PY{n}{e}\PY{p}{)}
    \PY{k}{print} \PY{l+s}{"}\PY{l+s}{genus = }\PY{l+s}{"}\PY{p}{,} \PY{n}{genus}\PY{p}{,} \PY{l+s}{"}\PY{l+s}{, max\PYZus{}length = }\PY{l+s}{"}\PY{p}{,} \PY{n}{max\PYZus{}length}\PY{p}{,}\PYZbs{}
          \PY{l+s}{"}\PY{l+s}{, matrix size = }\PY{l+s}{"}\PY{p}{,} \PY{n}{r}\PY{o}{.}\PY{n}{M}\PY{o}{.}\PY{n}{nrows}\PY{p}{(}\PY{p}{)}
    \PY{k}{print} \PY{l+s}{"}\PY{l+s}{estimate = }\PY{l+s}{"}\PY{p}{,} \PY{n}{r}\PY{o}{.}\PY{n}{estimate}\PY{p}{(}\PY{p}{)}

\PY{k}{def} \PY{n+nf}{test\PYZus{}weight\PYZus{}surface}\PY{p}{(}\PY{n}{genus}\PY{p}{,} \PY{n}{weights}\PY{p}{,} \PY{n}{max\PYZus{}weight}\PY{p}{)}\PY{p}{:}
    \PY{n}{e} \PY{o}{=} \PY{n}{ExtendedTypesBuilderSurfaceWeight}\PY{p}{(}\PY{n}{genus}\PY{p}{,} \PY{n}{weights}\PY{p}{,} \PY{n}{max\PYZus{}weight}\PY{p}{)}
    \PY{n}{r} \PY{o}{=} \PY{n}{RhoEstimatorExt}\PY{p}{(}\PY{n}{e}\PY{p}{)}
    \PY{k}{print} \PY{l+s}{"}\PY{l+s}{genus = }\PY{l+s}{"}\PY{p}{,} \PY{n}{genus}\PY{p}{,} \PY{l+s}{"}\PY{l+s}{, max\PYZus{}weight = }\PY{l+s}{"}\PY{p}{,} \PY{n}{max\PYZus{}weight}\PY{p}{,}\PYZbs{}
          \PY{l+s}{"}\PY{l+s}{, matrix size = }\PY{l+s}{"}\PY{p}{,} \PY{n}{r}\PY{o}{.}\PY{n}{M}\PY{o}{.}\PY{n}{nrows}\PY{p}{(}\PY{p}{)}
    \PY{k}{print} \PY{l+s}{"}\PY{l+s}{estimate = }\PY{l+s}{"}\PY{p}{,} \PY{n}{r}\PY{o}{.}\PY{n}{estimate}\PY{p}{(}\PY{p}{)}
\end{Verbatim}

\begin{Verbatim}[commandchars=\\\{\}]
\PY{n}{time} \PY{n}{test\PYZus{}length\PYZus{}surface}\PY{p}{(}\PY{l+m+mi}{2}\PY{p}{,} \PY{l+m+mi}{11}\PY{p}{)}
\end{Verbatim}
\begin{Verbatim}[formatcom=\color{blue}]
genus =  2 , max_length =  11 , matrix size =  111331
estimate =  0.662752835287
Time: CPU 13.36 s, Wall: 13.40 s
\end{Verbatim}

\begin{Verbatim}[commandchars=\\\{\}]
\PY{n}{time} \PY{n}{test\PYZus{}weight\PYZus{}surface}\PY{p}{(}\PY{l+m+mi}{2}\PY{p}{,} \PY{p}{[}\PY{l+m+mi}{1}\PY{p}{,} \PY{l+m+mi}{2}\PY{p}{,} \PY{l+m+mi}{3}\PY{p}{,} \PY{l+m+mi}{4}\PY{p}{,} \PY{l+m+mi}{1}\PY{p}{,} \PY{l+m+mi}{2}\PY{p}{,} \PY{l+m+mi}{4}\PY{p}{]}\PY{p}{,} \PY{l+m+mi}{17}\PY{p}{)}
\end{Verbatim}
\begin{Verbatim}[formatcom=\color{blue}]
genus =  2 , max_weight =  17 , matrix size =  98406
estimate =  0.662754827875
Time: CPU 11.58 s, Wall: 11.60 s
\end{Verbatim}

 Again, weights give better estimates than length.

\begin{Verbatim}[commandchars=\\\{\}]
\PY{n}{time} \PY{n}{test\PYZus{}weight\PYZus{}surface}\PY{p}{(}\PY{l+m+mi}{2}\PY{p}{,} \PY{p}{[}\PY{l+m+mi}{1}\PY{p}{,} \PY{l+m+mi}{2}\PY{p}{,} \PY{l+m+mi}{3}\PY{p}{,} \PY{l+m+mi}{4}\PY{p}{,} \PY{l+m+mi}{1}\PY{p}{,} \PY{l+m+mi}{2}\PY{p}{,} \PY{l+m+mi}{4}\PY{p}{]}\PY{p}{,} \PY{l+m+mi}{24}\PY{p}{)}
\end{Verbatim}
\begin{Verbatim}[formatcom=\color{blue}]
genus =  2 , max_weight =  24 , matrix size =  5117838
estimate =  0.662770548031
Time: CPU 1170.24 s, Wall: 1173.11 s
\end{Verbatim}

\begin{Verbatim}[commandchars=\\\{\}]
\PY{n}{time} \PY{n}{test\PYZus{}weight\PYZus{}surface}\PY{p}{(}\PY{l+m+mi}{2}\PY{p}{,} \PY{p}{[}\PY{l+m+mi}{1}\PY{p}{,} \PY{l+m+mi}{2}\PY{p}{,} \PY{l+m+mi}{3}\PY{p}{,} \PY{l+m+mi}{4}\PY{p}{,} \PY{l+m+mi}{1}\PY{p}{,} \PY{l+m+mi}{2}\PY{p}{,} \PY{l+m+mi}{4}\PY{p}{]}\PY{p}{,} \PY{l+m+mi}{25}\PY{p}{)}
\end{Verbatim}
\begin{Verbatim}[formatcom=\color{blue}]
genus =  2 , max_weight =  25 , matrix size =  8999902
estimate =  0.662772114698
Time: CPU 2105.08 s, Wall: 2212.72 s
\end{Verbatim}

\begin{Verbatim}[commandchars=\\\{\}]
\PY{n}{time} \PY{n}{test\PYZus{}weight\PYZus{}surface}\PY{p}{(}\PY{l+m+mi}{3}\PY{p}{,} \PY{p}{[}\PY{l+m+mi}{1}\PY{p}{,} \PY{l+m+mi}{2}\PY{p}{,} \PY{l+m+mi}{3}\PY{p}{,} \PY{l+m+mi}{4}\PY{p}{,} \PY{l+m+mi}{5}\PY{p}{,} \PY{l+m+mi}{6}\PY{p}{,} \PY{l+m+mi}{1}\PY{p}{,} \PY{l+m+mi}{2}\PY{p}{,} \PY{l+m+mi}{6}\PY{p}{]}\PY{p}{,} \PY{l+m+mi}{25}\PY{p}{)}
\end{Verbatim}
\begin{Verbatim}[formatcom=\color{blue}]
genus =  3 , max_weight =  25 , matrix size =  7307293
estimate =  0.55277355933
Time: CPU 1662.42 s, Wall: 1666.67 s
\end{Verbatim}

 We will now compare these estimates with the best estimates up to now, due to Bartholdi.

 We redo Bartholdi's computations with high precision. The method is first to find the root $\zeta$ of the polynomial $((dx)^{m-1}-1)(x-1)+2(dx-1)$ (for $m=d=4g$) (beware, there is a typo in the article on Page 11 Line 10: he writes $d^m \zeta^m$, but he really means $d^{m-1}\zeta^{m-1}$ (this is what comes out of (3), and this is what he uses to compute effectively $\zeta$). Then, a solution of an explicit polynomial equation involving $\zeta$ is a lower bound for the spectral radius.

\begin{Verbatim}[commandchars=\\\{\}]
\PY{k}{def} \PY{n+nf}{bartholdiLower}\PY{p}{(}\PY{n}{g}\PY{p}{,} \PY{n}{print\PYZus{}zeta} \PY{o}{=} \PY{n}{false}\PY{p}{)}\PY{p}{:}
    \PY{l+s+sd}{"""returns Bartholdi's lower bound on the spectral radius }
\PY{l+s+sd}{       in the surface group of genus g.}
\PY{l+s+sd}{       Computations are exact, made only with algebraic numbers"""}
    \PY{n}{m} \PY{o}{=} \PY{l+m+mi}{4}\PY{o}{*}\PY{n}{g}
    \PY{n}{d} \PY{o}{=} \PY{l+m+mi}{4}\PY{o}{*}\PY{n}{g}
    \PY{n}{K}\PY{o}{.}\PY{o}{<}\PY{n}{x}\PY{o}{>} \PY{o}{=} \PY{n}{AA}\PY{p}{[}\PY{l+s}{'}\PY{l+s}{x}\PY{l+s}{'}\PY{p}{]}
    \PY{n}{P} \PY{o}{=} \PY{p}{(}\PY{p}{(}\PY{n}{d}\PY{o}{*}\PY{n}{x}\PY{p}{)}\PY{o}{\PYZca{}}\PY{p}{(}\PY{n}{m}\PY{o}{-}\PY{l+m+mi}{1}\PY{p}{)}\PY{o}{-}\PY{l+m+mi}{1}\PY{p}{)}\PY{o}{*}\PY{p}{(}\PY{n}{x}\PY{o}{-}\PY{l+m+mi}{1}\PY{p}{)}\PY{o}{+}\PY{l+m+mi}{2}\PY{o}{*}\PY{p}{(}\PY{n}{d}\PY{o}{*}\PY{n}{x}\PY{o}{-}\PY{l+m+mi}{1}\PY{p}{)}
    \PY{n}{zeta}\PY{o}{=}\PY{n+nb}{max}\PY{p}{(}\PY{n}{P}\PY{o}{.}\PY{n}{roots}\PY{p}{(}\PY{n}{multiplicities}\PY{o}{=}\PY{n+nb+bp}{False}\PY{p}{)}\PY{p}{)}
    \PY{k}{if} \PY{n}{print\PYZus{}zeta}\PY{p}{:}
        \PY{k}{print} \PY{l+s}{"}\PY{l+s}{zeta = }\PY{l+s}{"}\PY{p}{,} \PY{n}{zeta}

    \PY{c}{# We make intermediate computations in the symbolic ring,}
    \PY{c}{# since simplifications are handled much more efficiently.}
    \PY{c}{# The only drawback is that we have to convert back to polynomials in the end}
    \PY{n}{t}\PY{p}{,}\PY{n}{u}\PY{p}{,}\PY{n}{z} \PY{o}{=} \PY{n}{var}\PY{p}{(}\PY{l+s}{'}\PY{l+s}{t,u,z}\PY{l+s}{'}\PY{p}{)}
    \PY{n}{f} \PY{o}{=} \PY{l+m+mi}{2}\PY{o}{*}\PY{n}{d}\PY{o}{*}\PY{n}{t}\PY{o}{\PYZca{}}\PY{n}{m}
    \PY{n}{Hsing}  \PY{o}{=} \PY{n}{t}\PY{o}{/}\PY{p}{(}\PY{l+m+mi}{1}\PY{o}{+}\PY{p}{(}\PY{l+m+mi}{1}\PY{o}{-}\PY{n}{u}\PY{p}{)}\PY{o}{*}\PY{p}{(}\PY{n}{d}\PY{o}{-}\PY{l+m+mi}{1}\PY{o}{+}\PY{n}{u}\PY{p}{)}\PY{o}{*}\PY{n}{t}\PY{o}{\PYZca{}}\PY{l+m+mi}{2}\PY{p}{)}
    \PY{n}{g1sing} \PY{o}{=} \PY{n}{Hsing}\PY{p}{(}\PY{n}{t} \PY{o}{=} \PY{n}{t}\PY{o}{*}\PY{n}{z}\PY{p}{)}
    \PY{n}{g2sing} \PY{o}{=} \PY{n}{g1sing}\PY{p}{(}\PY{n}{t}\PY{o}{=} \PY{n}{t}\PY{o}{*}\PY{p}{(}\PY{n}{d}\PY{o}{-}\PY{n}{f}\PY{p}{)}\PY{o}{/}\PY{p}{(}\PY{n}{d}\PY{o}{-}\PY{p}{(}\PY{n}{d}\PY{o}{-}\PY{l+m+mi}{1}\PY{p}{)}\PY{o}{*}\PY{n}{f}\PY{p}{)}\PY{p}{,} \PY{n}{u} \PY{o}{=} \PY{p}{(}\PY{n}{d}\PY{o}{-}\PY{l+m+mi}{2}\PY{p}{)}\PY{o}{*}\PY{n}{f}\PY{o}{/}\PY{p}{(}\PY{n}{d}\PY{o}{-}\PY{n}{f}\PY{p}{)}\PY{p}{)}
    \PY{n}{A}\PY{o}{=}\PY{p}{(}\PY{n}{g2sing}\PY{o}{\PYZca{}}\PY{l+m+mi}{2}\PY{o}{-}\PY{l+m+mi}{1}\PY{o}{/}\PY{p}{(}\PY{l+m+mi}{4}\PY{o}{*}\PY{p}{(}\PY{n}{d}\PY{o}{-}\PY{l+m+mi}{1}\PY{p}{)}\PY{p}{)}\PY{p}{)}\PY{o}{.}\PY{n}{simplify\PYZus{}full}\PY{p}{(}\PY{p}{)}\PY{o}{.}\PY{n}{numerator}\PY{p}{(}\PY{p}{)}
    \PY{n}{A}\PY{o}{=}\PY{n}{A}\PY{p}{(}\PY{n}{z}\PY{o}{=}\PY{n}{SR}\PY{p}{(}\PY{n}{zeta}\PY{p}{)}\PY{p}{)}

    \PY{c}{#Convert everything back to polynomials}
    \PY{n}{R}\PY{o}{.}\PY{o}{<}\PY{n}{t}\PY{o}{>}\PY{o}{=}\PY{n}{AA}\PY{p}{[}\PY{l+s}{'}\PY{l+s}{t}\PY{l+s}{'}\PY{p}{]}
    \PY{n}{A}\PY{o}{=}\PY{n}{R}\PY{p}{(}\PY{n}{A}\PY{p}{)}
    \PY{n}{alpha}\PY{o}{=}\PY{n+nb}{min}\PY{p}{(}\PY{p}{[}\PY{n}{r} \PY{k}{for} \PY{n}{r} \PY{o+ow}{in} \PY{n}{A}\PY{o}{.}\PY{n}{roots}\PY{p}{(}\PY{n}{multiplicities} \PY{o}{=} \PY{n+nb+bp}{False}\PY{p}{)} \PY{k}{if} \PY{n}{r} \PY{o}{>} \PY{l+m+mi}{0}\PY{p}{]}\PY{p}{)}
    \PY{n}{rho} \PY{o}{=} \PY{n}{alpha}\PY{o}{/}\PY{p}{(}\PY{l+m+mi}{1}\PY{o}{+}\PY{p}{(}\PY{n}{d}\PY{o}{-}\PY{l+m+mi}{1}\PY{p}{)}\PY{o}{*}\PY{n}{alpha}\PY{o}{\PYZca{}}\PY{l+m+mi}{2}\PY{p}{)}
    \PY{k}{return} \PY{l+m+mi}{1}\PY{o}{/}\PY{p}{(}\PY{n}{d}\PY{o}{*}\PY{n}{rho}\PY{p}{)}
\end{Verbatim}

\begin{Verbatim}[commandchars=\\\{\}]
\PY{n}{bartholdiLower}\PY{p}{(}\PY{l+m+mi}{2}\PY{p}{,} \PY{n}{true}\PY{p}{)}
\end{Verbatim}
\begin{Verbatim}[formatcom=\color{blue}]
zeta =  0.999993324015561?
\end{Verbatim}

{\color{blue}
$\newcommand{\Bold}[1]{\mathbf{#1}}0.6624219223029230?$
}

 Let us check that our value of $\zeta$ coincides with the value given by Bartholdi:

\begin{Verbatim}[commandchars=\\\{\}]
\PY{l+m+mi}{1}\PY{o}{-}\PY{l+m+mf}{0.63}\PY{o}{*}\PY{l+m+mi}{10}\PY{o}{\PYZca{}}\PY{p}{(}\PY{o}{-}\PY{l+m+mi}{5}\PY{p}{)}
\end{Verbatim}

{\color{blue}
$\newcommand{\Bold}[1]{\mathbf{#1}}0.999993700000000$
}

 The value given by Bartholdi for $\zeta$ differs slightly from what we found, probably due to rounding errors. Since we made an exact computation relying only on algebraic numbers, our value should be the correct one. The lower bound in the end is slightly better than what Bartholdi claims in his article!

\begin{Verbatim}[commandchars=\\\{\}]
\PY{n}{bartholdiLower}\PY{p}{(}\PY{l+m+mi}{3}\PY{p}{,} \PY{n}{true}\PY{p}{)}
\end{Verbatim}
\begin{Verbatim}[formatcom=\color{blue}]
zeta =  0.9999999999703906?
\end{Verbatim}

{\color{blue}
$\newcommand{\Bold}[1]{\mathbf{#1}}0.5527735401122323?$
}

\begin{Verbatim}[commandchars=\\\{\}]
\PY{l+m+mi}{1}\PY{o}{-}\PY{l+m+mf}{0.29}\PY{o}{*}\PY{l+m+mi}{10}\PY{o}{\PYZca{}}\PY{p}{(}\PY{o}{-}\PY{l+m+mi}{10}\PY{p}{)}
\end{Verbatim}

{\color{blue}
$\newcommand{\Bold}[1]{\mathbf{#1}}0.999999999971000$
}

 Again, Bartholdi's value for $\zeta$ is not completely exact, but almost. The lower bound we get on $\rho$ agrees with Bartholdi's claims in his paper.

\begin{Verbatim}[commandchars=\\\{\}]
\PY{n}{bartholdiLower}\PY{p}{(}\PY{l+m+mi}{4}\PY{p}{)}
\end{Verbatim}

{\color{blue}
$\newcommand{\Bold}[1]{\mathbf{#1}}0.484122920740487?$
}

\begin{Verbatim}[commandchars=\\\{\}]
\PY{n}{bartholdiLower}\PY{p}{(}\PY{l+m+mi}{5}\PY{p}{)}
\end{Verbatim}

{\color{blue}
$\newcommand{\Bold}[1]{\mathbf{#1}}0.4358898943553?$
}

 Our bounds are really better in genus $2$ (even without taking extended types), marginally better in genus $3$ (but we need to take very long extended types, corresponding to very large matrices), but worse in genus $4$ (and certainly also in higher genus).

\begin{Verbatim}[commandchars=\\\{\}]
\PY{o}{%}\PY{n}{cython}
\PY{c}{# Since multiplication of double float sparse matrices is not}
\PY{c}{# specifically implemented in sage, it is very slow.}
\PY{c}{# We provide a Cython version.}
\PY{c}{# All the number crunching part (the while True loop) is plain C, }
\PY{c}{# and therefore very fast}
\PY{k}{from} \PY{n+nn}{libc.stdlib} \PY{k}{cimport} \PY{n}{malloc}\PY{p}{,} \PY{n}{free}

\PY{k}{def} \PY{n+nf}{max\PYZus{}expansion\PYZus{}sparse\PYZus{}double\PYZus{}cython}\PY{p}{(}\PY{n}{M}\PY{p}{,} \PY{n}{precision} \PY{o}{=} \PY{l+m+mf}{10}\PY{o}{\PYZca{}}\PY{p}{(}\PY{o}{-}\PY{l+m+mf}{50}\PY{p}{)}\PY{p}{)}\PY{p}{:}
    \PY{l+s+sd}{"""M is a matrix with nonnegative entries.}
\PY{l+s+sd}{    returns a number that is less than or equal to the}
\PY{l+s+sd}{    maximum of (q, M q) for q in the unit sphere. }
\PY{l+s+sd}{    Computations are done within the given precision.}
\PY{l+s+sd}{    """}
    \PY{k}{cdef} \PY{p}{:}
        \PY{n}{Py\PYZus{}ssize\PYZus{}t} \PY{n}{k}\PY{p}{,} \PY{n}{vect\PYZus{}size}\PY{p}{,} \PY{n}{mat\PYZus{}size}\PY{p}{,} \PY{n}{i}
        \PY{n}{double} \PY{o}{*}\PY{n}{w}\PY{p}{,} \PY{n}{diff}\PY{p}{,} \PY{n}{norm}\PY{p}{,} \PY{n}{expansion}\PY{p}{,} \PY{n}{norm\PYZus{}inv}
        \PY{n}{double} \PY{n}{prec} \PY{o}{=} \PY{n}{precision}

    \PY{n}{nz}        \PY{o}{=} \PY{n}{M}\PY{o}{.}\PY{n}{nonzero\PYZus{}positions}\PY{p}{(}\PY{n}{copy}\PY{o}{=}\PY{n+nb+bp}{False}\PY{p}{)}
    \PY{n}{mat\PYZus{}size}  \PY{o}{=} \PY{n+nb}{len}\PY{p}{(}\PY{n}{nz}\PY{p}{)}
    \PY{n}{vect\PYZus{}size} \PY{o}{=} \PY{n}{M}\PY{o}{.}\PY{n}{nrows}\PY{p}{(}\PY{p}{)}

    \PY{c}{# Allocate the memory}
    \PY{n}{cdef}\PY{p}{:}
        \PY{n+nb}{int} \PY{o}{*}\PY{n}{matrix\PYZus{}line}     \PY{o}{=} \PY{o}{<}\PY{n+nb}{int} \PY{o}{*}\PY{o}{>}\PY{n}{malloc}\PY{p}{(}\PY{n}{mat\PYZus{}size} \PY{o}{*} \PY{n}{sizeof}\PY{p}{(}\PY{n+nb}{int}\PY{p}{)}\PY{p}{)}
        \PY{n+nb}{int} \PY{o}{*}\PY{n}{matrix\PYZus{}column}   \PY{o}{=} \PY{o}{<}\PY{n+nb}{int} \PY{o}{*}\PY{o}{>}\PY{n}{malloc}\PY{p}{(}\PY{n}{mat\PYZus{}size} \PY{o}{*} \PY{n}{sizeof}\PY{p}{(}\PY{n+nb}{int}\PY{p}{)}\PY{p}{)}
        \PY{n}{double} \PY{o}{*}\PY{n}{matrix\PYZus{}coeff} \PY{o}{=} \PY{o}{<}\PY{n}{double} \PY{o}{*}\PY{o}{>}\PY{n}{malloc}\PY{p}{(}\PY{n}{mat\PYZus{}size} \PY{o}{*}\PY{n}{sizeof}\PY{p}{(}\PY{n}{double}\PY{p}{)}\PY{p}{)}
        \PY{n}{double} \PY{o}{*}\PY{n}{u}            \PY{o}{=} \PY{o}{<}\PY{n}{double} \PY{o}{*}\PY{o}{>} \PY{n}{malloc}\PY{p}{(}\PY{n}{vect\PYZus{}size} \PY{o}{*} \PY{n}{sizeof}\PY{p}{(}\PY{n}{double}\PY{p}{)}\PY{p}{)}
        \PY{n}{double} \PY{o}{*}\PY{n}{v}            \PY{o}{=} \PY{o}{<}\PY{n}{double} \PY{o}{*}\PY{o}{>} \PY{n}{malloc}\PY{p}{(}\PY{n}{vect\PYZus{}size} \PY{o}{*} \PY{n}{sizeof}\PY{p}{(}\PY{n}{double}\PY{p}{)}\PY{p}{)}
    \PY{k}{if} \PY{p}{(}\PY{p}{(}\PY{o+ow}{not} \PY{n}{matrix\PYZus{}line}\PY{p}{)} \PY{o+ow}{or} \PY{p}{(}\PY{o+ow}{not} \PY{n}{matrix\PYZus{}column}\PY{p}{)} \PY{o+ow}{or} \PY{p}{(}\PY{o+ow}{not} \PY{n}{matrix\PYZus{}coeff}\PY{p}{)}
            \PY{o+ow}{or} \PY{p}{(}\PY{o+ow}{not} \PY{n}{u}\PY{p}{)} \PY{o+ow}{or} \PY{p}{(}\PY{o+ow}{not} \PY{n}{v}\PY{p}{)}\PY{p}{)}\PY{p}{:}
        \PY{n}{free}\PY{p}{(}\PY{n}{matrix\PYZus{}line}\PY{p}{)}
        \PY{n}{free}\PY{p}{(}\PY{n}{matrix\PYZus{}column}\PY{p}{)}
        \PY{n}{free}\PY{p}{(}\PY{n}{matrix\PYZus{}coeff}\PY{p}{)}
        \PY{n}{free}\PY{p}{(}\PY{n}{u}\PY{p}{)}
        \PY{n}{free}\PY{p}{(}\PY{n}{v}\PY{p}{)}
        \PY{k}{raise} \PY{n+ne}{MemoryError}\PY{p}{(}\PY{p}{)}

    \PY{c}{# Store in the arrays the nonzero coefficients of M/2     }
    \PY{k}{for} \PY{n}{k} \PY{o+ow}{in} \PY{n+nb}{xrange}\PY{p}{(}\PY{n}{mat\PYZus{}size}\PY{p}{)}\PY{p}{:}
        \PY{n}{matrix\PYZus{}line}\PY{p}{[}\PY{n}{k}\PY{p}{]}   \PY{o}{=} \PY{n}{nz}\PY{p}{[}\PY{n}{k}\PY{p}{]}\PY{p}{[}\PY{l+m+mf}{0}\PY{p}{]}
        \PY{n}{matrix\PYZus{}column}\PY{p}{[}\PY{n}{k}\PY{p}{]} \PY{o}{=} \PY{n}{nz}\PY{p}{[}\PY{n}{k}\PY{p}{]}\PY{p}{[}\PY{l+m+mf}{1}\PY{p}{]}
        \PY{n}{matrix\PYZus{}coeff}\PY{p}{[}\PY{n}{k}\PY{p}{]}  \PY{o}{=} \PY{p}{<}\PY{k+kt}{double}\PY{p}{>}\PY{n}{M}\PY{p}{[}\PY{n}{matrix\PYZus{}line}\PY{p}{[}\PY{n}{k}\PY{p}{]}\PY{p}{,} \PY{n}{matrix\PYZus{}column}\PY{p}{[}\PY{n}{k}\PY{p}{]}\PY{p}{]}\PY{o}{/}\PY{l+m+mf}{2}

    \PY{c}{# Fill u with a uniform vector of norm 1    }
    \PY{n}{d} \PY{o}{=} \PY{l+m+mf}{1}\PY{o}{/}\PY{n}{sqrt}\PY{p}{(}\PY{p}{<}\PY{k+kt}{double}\PY{p}{>}\PY{n}{M}\PY{o}{.}\PY{n}{nrows}\PY{p}{(}\PY{p}{)}\PY{p}{)}
    \PY{k}{for} \PY{n}{k} \PY{o+ow}{in} \PY{n+nb}{xrange}\PY{p}{(}\PY{n}{vect\PYZus{}size}\PY{p}{)}\PY{p}{:}
        \PY{n}{u}\PY{p}{[}\PY{n}{k}\PY{p}{]} \PY{o}{=} \PY{n}{d}

    \PY{c}{# Apply iteratively the matrix (M+M.transpose())/2 to u,}
    \PY{c}{# to converge towards the maximal eigenvector}
    \PY{n}{expansion} \PY{o}{=} \PY{l+m+mf}{0}
    \PY{k}{while} \PY{n+nb+bp}{True}\PY{p}{:}
        \PY{c}{# v = M * u}
        \PY{k}{for} \PY{n}{k} \PY{o+ow}{in} \PY{n+nb}{xrange}\PY{p}{(}\PY{n}{vect\PYZus{}size}\PY{p}{)}\PY{p}{:}
            \PY{n}{v}\PY{p}{[}\PY{n}{k}\PY{p}{]} \PY{o}{=} \PY{l+m+mf}{0}
        \PY{k}{for} \PY{n}{k} \PY{o+ow}{in} \PY{n+nb}{xrange}\PY{p}{(}\PY{n}{mat\PYZus{}size}\PY{p}{)}\PY{p}{:}
            \PY{n}{v}\PY{p}{[}\PY{n}{matrix\PYZus{}line}\PY{p}{[}\PY{n}{k}\PY{p}{]}\PY{p}{]}   \PY{o}{+}\PY{o}{=} \PY{n}{matrix\PYZus{}coeff}\PY{p}{[}\PY{n}{k}\PY{p}{]} \PY{o}{*} \PY{n}{u}\PY{p}{[}\PY{n}{matrix\PYZus{}column}\PY{p}{[}\PY{n}{k}\PY{p}{]}\PY{p}{]}
            \PY{n}{v}\PY{p}{[}\PY{n}{matrix\PYZus{}column}\PY{p}{[}\PY{n}{k}\PY{p}{]}\PY{p}{]} \PY{o}{+}\PY{o}{=} \PY{n}{matrix\PYZus{}coeff}\PY{p}{[}\PY{n}{k}\PY{p}{]} \PY{o}{*} \PY{n}{u}\PY{p}{[}\PY{n}{matrix\PYZus{}line}\PY{p}{[}\PY{n}{k}\PY{p}{]}\PY{p}{]}

        \PY{c}{# norm = v.norm()}
        \PY{n}{norm} \PY{o}{=} \PY{l+m+mf}{0}
        \PY{k}{for} \PY{n}{k} \PY{o+ow}{in} \PY{n+nb}{xrange}\PY{p}{(}\PY{n}{vect\PYZus{}size}\PY{p}{)}\PY{p}{:}
            \PY{n}{norm} \PY{o}{+}\PY{o}{=} \PY{n}{v}\PY{p}{[}\PY{n}{k}\PY{p}{]} \PY{o}{*} \PY{n}{v}\PY{p}{[}\PY{n}{k}\PY{p}{]}
        \PY{n}{norm} \PY{o}{=} \PY{n}{sqrt}\PY{p}{(}\PY{n}{norm}\PY{p}{)}

        \PY{n}{diff}      \PY{o}{=} \PY{n}{norm} \PY{o}{-} \PY{n}{expansion}
        \PY{n}{expansion} \PY{o}{=} \PY{n}{norm}

        \PY{c}{# v = v/v.norm()}
        \PY{n}{norm\PYZus{}inv} \PY{o}{=} \PY{l+m+mf}{1}\PY{o}{/}\PY{n}{norm}
        \PY{k}{for} \PY{n}{k} \PY{o+ow}{in} \PY{n+nb}{xrange}\PY{p}{(}\PY{n}{vect\PYZus{}size}\PY{p}{)}\PY{p}{:}
            \PY{n}{v}\PY{p}{[}\PY{n}{k}\PY{p}{]} \PY{o}{*}\PY{o}{=} \PY{n}{norm\PYZus{}inv}

        \PY{c}{# swap u and v}
        \PY{n}{w} \PY{o}{=} \PY{n}{u}
        \PY{n}{u} \PY{o}{=} \PY{n}{v}
        \PY{n}{v} \PY{o}{=} \PY{n}{w}

        \PY{k}{if} \PY{p}{(}\PY{n}{diff} \PY{o}{<} \PY{n}{prec}\PY{p}{)}\PY{p}{:}
            \PY{k}{break}

    \PY{n}{free}\PY{p}{(}\PY{n}{matrix\PYZus{}line}\PY{p}{)}
    \PY{n}{free}\PY{p}{(}\PY{n}{matrix\PYZus{}column}\PY{p}{)}
    \PY{n}{free}\PY{p}{(}\PY{n}{matrix\PYZus{}coeff}\PY{p}{)}
    \PY{n}{free}\PY{p}{(}\PY{n}{u}\PY{p}{)}
    \PY{n}{free}\PY{p}{(}\PY{n}{v}\PY{p}{)}
    \PY{k}{return} \PY{n}{expansion}
\end{Verbatim}

\bibliography{biblio}
\bibliographystyle{amsalpha}
\end{document}